\documentclass{amsart}
\usepackage{amsfonts, amsmath, amssymb, amsthm}
\usepackage{url}
\usepackage[dvipdf]{graphicx}
\usepackage[dvipdf]{color}
\usepackage[dvips]{xcolor}
\usepackage[all]{xy}

\newtheorem{theorem}{Theorem}[section]
\newtheorem{lemma}{Lemma}[section]
\newtheorem{proposition}{Proposition}[section]
\newtheorem{corollary}{Corollary}[section]

\newtheorem{remark}{Remark}[section]

\newtheorem{convention}{Convention}[section]
\newtheorem{theoremAlphabet}{Theorem}

\newcommand{\ext}{{\rm Ext}}
\newcommand{\Mod}{{\rm Mod}}

\newcommand{\ZZ}{\mathcal{MF}_1}
\newcommand{\cl}[1]{{\rm cl}_{GM}(#1)}
\newcommand{\clred}[1]{{\rm cl}^{red}_{GM}(#1)}
\newcommand{\partialred}{\partial_{GM}^{red}}

\newcommand{\supportMF}[1]{{\rm Supp}(#1)}

\newcommand{\squ}[1]{{\tt USq}(#1)}
\newcommand{\sq}[1]{{\tt Sq}^\infty(#1)}
\newcommand{\sqt}[1]{{\tt sub}^\infty(#1)}

\newcommand{\gromov}[3]{\langle #1\,|\,#2\rangle_{#3}}

\newcommand{\trunc}[1]{{#1}^{\circ}}

\newcommand{\AC}{{\rm AC}}
\newcommand{\ACINV}{{\rm AC_{inv}}}
\newcommand{\ACAS}{{\rm AC_{as}}}

\newcommand{\qi}{{\rm QI}}
\newcommand{\qigroup}{{\mathfrak QI}}

\newcommand{\vani}[1]{Z^{0}_{#1}}
\newcommand{\augVani}[1]{Z_{#1}}

\newcommand{\nullsets}[2]{\mathfrak{N}_{#1}(#2)}

\newcommand{\ACmonoid}{\mathfrak{AC}_{as}}
\newcommand{\ACGroup}{\mathfrak{AC}}

\newcommand{\visualindist}{{\sf Vis}}
\newcommand{\complexity}[1]{{\rm cx}(#1)}
\newcommand{\RR}{\mathbb{R}_{+}^{\mathcal{S}}}
\newcommand{\PR}{{\rm P}\RR}

\newcommand{\extension}[1]{\partial_\infty(#1)}

\newcommand{\proj}{{\rm proj}}

\newcommand{\teich}{\mathcal{T}}

\newcommand{\GmInv}{\mathcal{C}_{GM}}
\newcommand{\GmTeich}{\mathcal{T}_{GM}}
\newcommand{\Gmbdy}{\tilde{\partial}_{GM}}

\newcommand{\curvecomplex}{\mathbb{X}}

\newcommand{\MCG}{{\rm MCG}}
\newcommand{\NMF}{\mathcal{N}_{M\!F}}

\newcommand{\accum}{\mathcal{ACM}}

\newcommand{\height}{{\rm ht}}

\newcommand{\equivalence}[1]{[[#1]]}

\begin{document}

\title[Geometry of the Gromov product]
{Geometry of the Gromov product : Geometry at infinity of Teichm\"uller space}

\author{Hideki \textsc{Miyachi}}
\address{Department of Mathematics,
Graduate School of Science,
Osaka University,
Machikaneyama 1-1,
Toyonaka, Osaka 560-0043, Japan}
\email{miyachi@math.sci.osaka-u.ac.jp}


%

\maketitle

\begin{abstract}
This paper is devoted to study of transformations on metric spaces.
It is done in an effort to produce qualitative version of quasi-isometries
which takes into account the asymptotic behavior of the Gromov product in hyperbolic spaces.
We characterize a quotient semigroup of such transformations on Teichm\"uller space
by use of simplicial automorphisms of the complex of curves,
and we will see that such transformation is recognized as
a ``coarsification" of isometries on Teichm\"uller space
which is rigid at infinity.
We also show a hyperbolic characteristic that
any finite dimensional Teichm\"uller space does not admit
(quasi)-invertible rough-homothety.
\end{abstract}

\section{Introduction}
\subsection{Backgrounds}
Let $(X,d_X)$ be a metric space.
The \emph{Gromov product}
with reference point $x_0\in X$
is defined by
\begin{equation}
\label{eq:Gromov_product}
\gromov{x_1}{x_2}{x_0}^X
=\gromov{x_1}{x_2}{x_0}=
\frac{1}{2}(d_X(x_0,x_1)+d_X(x_0,x_2)-d_X(x_1,x_2)).
\end{equation}
We define the \emph{Gromov product} of
two sequences ${\bf x}=\{x_{n}\}_{n\in \mathbb{N}}$,
${\bf y}=\{y_{n}\}_{n\in \mathbb{N}}$
in $X$ by
\begin{equation}
\gromov{{\bf x}}{{\bf y}}{x_{0}}=
\liminf_{n,m\to \infty}\gromov{x_n}{y_m}{x_0}
\end{equation}

\begin{convention}
When the metric space and the reference point in the discussion are
clear in the context,
we omit to specify them in denoting the Gromov product.
We always assume in this paper that any metric space is of infinite diamter.
\end{convention}

Let
$\squ{X}\subset X^{\mathbb{N}}$
is the set of unbounded sequences in $X$.
We call a sequence ${\bf x}\in \squ{X}$
\emph{convergent at infinity}
if
\begin{equation}
\label{eq:converging_to_infinity_definition}
\gromov{{\bf x}}{{\bf x}}{}=\infty
\end{equation}
(cf. \S8 in \cite{Gromov}).
Any sequence satisying \eqref{eq:converging_to_infinity_definition} is contained in $\squ{X}$.
The definition \eqref{eq:converging_to_infinity_definition}
is independent of the choice of the reference point.
Let
$$
\sq{X}=\{{\bf x}\in \squ{X}\mid
\gromov{{\bf x}}{{\bf x}}{}=\infty\}.
$$
We say that a sequence ${\bf y}\in X^{\mathbb{N}}$
is \emph{visually indistinguishable} from ${\bf x}\in X^{\mathbb{N}}$
if
\begin{equation}
\label{eq:visually_indistinguishable_definition}
\gromov{{\bf x}}{{\bf y}}{}
=\infty.
\end{equation}
For a sequence ${\bf x}\in X^{\mathbb{N}}$,
we define
$$
\visualindist({\bf x})=\{{\bf y}\in \sq{X}\mid
\mbox{$\gromov{{\bf y}}{{\bf x}}{}
=\infty$}\}.
$$
Notice that
$\visualindist({\bf x})=\emptyset$
when a sequence ${\bf x}$ is bounded.
Two sequences ${\bf x}^1,{\bf x}^2\in \squ{X}$
are said to be \emph{asymptotic}
if $\visualindist ({\bf x}^1)=\visualindist ({\bf x}^2)$.
%
%
%
%
%
%
When $X$ is a Gromov hyperbolic space,
$\visualindist({\bf x})$ defines a point in the Gromov boundary
(cf. \cite{BF}).

In this paper,
aiming for developing the coarse geometry on Teichm\"uller space,
we inverstigate the theory
of mappings on metric spaces
with respecting for asymptotic behavior of sequences converging at infinity
(cf. Tables \ref{table:1} and \ref{table:2}).
Namely,
we (pretend to) recognize that two unbounded sequences  ${\bf x}^{1}$ and ${\bf x}^{2}$
determine the same ideal point at infinity
if two sequences ${\bf x}^{1}$, ${\bf x}^{2}$ converging at infinity
are asymptotic.
Intuitively,
asymptotically conservative mappings given in this paper are mappings keeping
the divergence conditions of the Gromov products of two sequences converging at infinity

\begin{table}
\begin{tabular}{|p{10em}|p{20em}|}\hline
Coarse geometry
&
Geometry on the Gromov product \\ \hline\hline
Quasi-isometries (qi)
&
Asymptotically conservative (ac) mappings\\ \hline
Coarsely Lipschitz
&
weakly ac
\\ \hline
Coarsely co-Lipschitz
&
$\omega({\bf z})\in  \visualindist(\omega({\bf x}))$
$\Rightarrow$
${\bf z}\in  \visualindist({\bf x})$
\ ($\forall {\bf x}, {\bf z}\in \squ{X}$)
 \\\hline
 Cobounded
 &
 Asymptotically surjective \\ \hline
 Cobounded qi
 &
Invertible ac \\ \hline
 Quasi-inverse
 &
Asymptotic quasi-inverse\\
\hline 
Parallelism
&
Close at infinity \\ \hline
$\qigroup (X)$
(cf. \eqref{eq:qi-group})
&
$\ACGroup (X)$ (\S\ref{subsec:results})\\ \hline
\end{tabular}
\caption{Comparison with the coarse geometry for general metric spaces $X$.
For details, see \S\ref{subsec:definition} and
\S\ref{sec:comparison_with_coarse_geometry}.}
\label{table:1}
\end{table}

\begin{table}
\begin{tabular}{|p{10em}|p{20em}|}\hline
Coarse geometry
&
Geometry on the Gromov product \\ \hline\hline
Quasi-isometries
&
Asymptotically conservative mappings \\ \hline
Coarsely Lipschitz
&
$\gromov{{\bf x}}{{\bf z}}{}=\infty$
$\Rightarrow$
$\gromov{\omega({\bf x})}{\omega({\bf z})}{}=\infty$
\ ($\forall {\bf x}, {\bf z}\in \sq{X}$)
\\ \hline
Coarsely co-Lipschitz
&
$\gromov{\omega({\bf x})}{\omega({\bf z})}{}=\infty$
$\Rightarrow$
$\gromov{{\bf x}}{{\bf z}}{}=\infty$
\ ($\forall {\bf x}, {\bf z}\in \sq{X}$)
 \\\hline
\end{tabular}
\caption{Comparison with the coarse geometry for metric spaces which are WBGP.
For details, see \S\ref{sec:stable_at_infinity}.}
\label{table:2}
\end{table}

\subsection{Definitions}
\label{subsec:definition}

Let $X$ and $Y$ be metric spaces.
A mapping $\omega\in Y^{X}$ is said to be
\emph{asymptotically conservative with the Gromov product}
(\emph{asymptotically conservative} for short)
if
for any sequence ${\bf x}\in \squ{X}$,
the following two conditions hold;
\begin{enumerate}
\item
\label{item:definition-AC1}
$\omega(\visualindist({\bf x}))\subset
\visualindist(\omega({\bf x}))$.
\item
\label{item:definition-AC2}
For any ${\bf z}\in \squ{X}$,
if $\omega({\bf z})\in \visualindist(\omega({\bf x}))$,
then ${\bf z}\in \visualindist({\bf x})$.
\end{enumerate}
We will call a map $\omega\in Y^X$ with the condition \eqref{item:definition-AC1}
above \emph{weakly asymptotically conservative}
(cf. \S\ref{subsec:asymptoticallyconservativeincoarsegeometry}).

%

Here,
for a sequence ${\bf x}=\{x_{n}\}_{n\in \mathbb{N}}\in X^{\mathbb{N}}$,
$E\subset X^{\mathbb{N}}$
and a mapping $\omega\in Y^{X}$,
we define usually
$$
\omega({\bf x})=\{\omega(x_{n})\}_{n\in \mathbb{N}}\in Y^{\mathbb{N}},
\quad 
\omega(E)=\{\omega({\bf x})\mid {\bf x}\in E\}.
$$

Two mappings $\omega_1$,
$\omega_2\in Y^X$
are said to be \emph{close at infinity},
if for any ${\bf x}^1,{\bf x}^2\in \sq{X}$,
$\visualindist(\omega_1({\bf x}^1))=\visualindist(\omega_2({\bf x}^2))$
holds whenever $\visualindist({\bf x}^1)=\visualindist({\bf x}^2)$.
An asymptotically conservative mapping $\omega\in Y^X$ is said to be \emph{invertible}
if there is an asymptotically conservative mapping $\omega'\in X^Y$ such that
$\omega'\circ \omega$ and $\omega\circ\omega'$ are close to the identity
mappings on $X$ and $Y$,
respectively.
We call such $\omega'$ an \emph{asymptotic quasi-inverse} of $\omega$.
Let $\ACINV(X)$ be the set of invertible asymptotically conservative mappings
on $X$ to itself.
%
For instance,
any isometic isomorphism between metric spaces
is invertible asymptotically conservative.
The notions of mappings given above are stable under parallelism
(cf. Proposition \ref{prop:stability}).
In \S\ref{sec:comparison_with_coarse_geometry},
we will give more discussion.

\subsection{Results}
\label{subsec:results}
We first observe the following theorem
(cf. \S\ref{subsec:Monoild_semingroups}).

\begin{theoremAlphabet}[The group $\ACGroup (X)$]
\label{thm:arbitrary_quotient_space}
Let $X$ be a metric space.
The set $\ACINV(X)$ admits a monoid structure with respect to
the composition of mappings.
Furthermore,
the relation ``closeness at infinity" is a semigroup congruence on $\ACINV(X)$
and the quotient semigroup $\ACGroup (X)$ is a group.
\end{theoremAlphabet}

%

\paragraph{{\bf Large scale geometry of Teichm\"uller space}}
Our main interest is to clarify the large scale geometry of Teichm\"uller space
$\teich$
in respecting for asymptotic behaviors of sequences converging at infinity.

\subsubsection*{Rigidity theorem}
Let $S$ be a compact orientable surface.
We denote the \emph{complexity} of $S$
by
$$
\complexity{S}=3\,{\rm genus}(S)-3+\#\{\mbox{components of $\partial S$}\}.
$$
The Euler characteristic of $S$ is denoted by $\chi(S)$.
Throughout this paper,
we always assume that $\chi(S)<0$.
$S$ is said to be in the \emph{sporadic case} if $\complexity{S}\le 1$.

Let $\teich$ be the Teichm\"uller space of $S$
endowed with the Teichm\"uller distance.
The extended mapping class group $\MCG^*(S)$ of $S$ acts
isometrically on $\teich$ and we have a group homomorphism
$$
\mathcal{I}_0\colon \MCG^*(S)\to {\rm Isom}(\teich).
$$
We also have a monoid homomorphism
$$
\mathcal{I}\colon {\rm Isom}(\teich)\to \ACINV(\teich)
$$
defined by the inclusion
(see \S\ref{subsec:action_of_extended_mapping_class_groups}).
We will prove the following rigidity theorem
(cf. Theorem \ref{thm:induced_automorphism}).

\begin{theoremAlphabet}[Rigidity]
\label{thm:Teichmullerspace}
Suppose $\complexity{S}\ge 2$.
Let $\curvecomplex (S)$ be the complex of curves on $S$.
Then,
there is a monoid epimorphism
$$
\Xi\colon
\ACINV(\teich)\to {\rm Aut}(\curvecomplex (S))
$$
which descends to an isomorphism
$$
\ACGroup(\teich)\to {\rm Aut}(\curvecomplex (S))
$$
satisfying the following commutative diagram
$$
\xymatrix{
\MCG^*(S)
\ar[r]^{\mathcal{I}_0} &
{\rm Isom}(\teich)
\ar[r]^{\mathcal{I}} \ar[dr]^{\mbox{{\tiny group iso}}} & \ACINV(\teich) \ar[d]^{\mbox{{\tiny proj}}} \ar[dr]^{\Xi}& \\
& &\ACGroup(\teich)  \ar[r]^{\mbox{{\tiny group iso}}} &{\rm Aut}(\curvecomplex(S)), &
}
$$
where ${\rm Aut}(\curvecomplex (S))$ is the group of simplicial automorphisms of
$\curvecomplex (S)$.
\end{theoremAlphabet}



\subsubsection*{Relation to the coarse geometry}
Recently,
A. Eskin, H.Masur and K.Rafi observed a remarkable result that
any cobounded quasi-isometry of $\teich$ is parallel to an isometry,
and the inclusion
${\rm Isom}(\teich)\hookrightarrow {\rm QI}(\teich)$
induces an isomorphism
$$
{\rm Isom}(\teich)\cong \qigroup(\teich)=\qi(\teich)/(\mbox{parallelism})
$$
when $S$ is in the non-sporadic case.
Especially,
any self quasi-isometry on $\teich$ is weakly asymptotically conservative
(cf. 
Proposition \ref{prop:stability}).
Hence,
we have the following
sequence of monoids and monoid homomorphisms
\begin{equation}
\label{eq:sequence1}
{\rm Isom}(\teich)
\hookrightarrow \qi(\teich)
\hookrightarrow \ACINV(\teich)
\end{equation}
by inclusions
(see Corollary \ref{coro:criterion-quasi-isometries}).
Theorem \ref{thm:Teichmullerspace}
implies the following.

\begin{corollary}[Relation to the coarse geometry on $\teich$]
For non-sporadic cases,
a quotient set $\qi(\teich)/(\mbox{close at infinity})$ admits a group structure
equipped with the operation defined by composition,
and
the sequence \eqref{eq:sequence1} descends to the following 
sequence of isomorphisms
$$
{\rm Isom}(\teich)
\cong 
\qigroup(\teich)
\cong \qi(\teich)/(\mbox{close at infinity})
\cong \ACGroup(\teich).
$$
\end{corollary}

\begin{corollary}[Criterion for parallelism]
\label{coro:parallelism}
Let $\omega,\psi\colon \teich\to \teich$ be cobounded quasi-isometries.
The following are equivalent.
\begin{enumerate}
\item
\label{item:coro:parallelism2}
For any ${\bf x},{\bf y}\in \sq{\teich}$
with $\gromov{{\bf x}}{{\bf y}}{}=\infty$,
it holds $\gromov{\omega({\bf x})}{\psi({\bf y})}{}=\infty$.
\item
\label{item:coro:parallelism3}
$\psi$ is parallel to $\omega$.
\end{enumerate}
\end{corollary}
Corollary \ref{coro:parallelism}
gives an analogy between the hyperbolic space and Teichm\"uller space.
Indeed,
the conclusion holds when we consider the hyperbolic space
$\mathbb{H}^n$ ($n \ge 2$) instead of the Teichm\"uller space $\teich$.
%
%

\subsubsection*{No-rough homothety}
By applying the discussion in the proof of Theorem \ref{thm:Teichmullerspace},
we also obtain a hyperbolic characteristic of Teichm\"uller space.
In fact,
we will give a proof of the following folklore result
in \S\ref{subsec:rough_homothety}.

\begin{theoremAlphabet}[No rough-homothety with $K\ne 1$]
\label{thm:Teichmullerspace_homothety}
There is no $(K,D)$-rough homothety with
asymptotic quasi-inverse on the Teichm\"uller space of $S$
unless $K=1$.
\end{theoremAlphabet}
Here, a mapping $\omega\colon (X,d_X)\to (Y,d_Y)$
between metric spaces
is said to be a $(K,D)$-\emph{rough homothety} if
\begin{equation}
\label{eq:K_D_homothety}
|d_Y(\omega(x_1),\omega(x_2))-
Kd_X(x_1,x_2)|\le D
\end{equation}
for $x_1,x_2\in X$
(cf. Chapter 7 of \cite{Buyalo_Schroeder}).
Any rough-homothety is asymptotically conservative.
Theorem \ref{thm:Teichmullerspace_homothety}
implies that there is no non-trivial similarity in Teichm\"uller space,
like in hyperbolic spaces.
Since
rough homotheties are quasi-isometries,
if $\omega$ in Theorem \ref{thm:Teichmullerspace_homothety} is cobounded,
the rigidity in the theorem follows from Eskin-Masur-Rafi's quasi-isometry rigidity theorem.
However,
the author does not know whether rough homotheties in the statement are cobounded or not.

Though enormous hyperbolic characteristics have been observed in Teichm\"uller space,
we would like to notice a remarkable fact
proven by Athreya, Bufetov, Eskin and Mirzakhani in \cite{ABEM}.
Indeed,
they observed that the volume of the metric balls in Teichm\"uller space has exponential growth.
Thus,
their result might imply that
a measurable $(K,D)$-homothety on Teichm\"uller space
does not exist unless $K\ne 1$.

\subsection{Plan of this paper}
This paper is organized as follows:
In \S\ref{sec:asymptotic_conservative},
we will introduce asymptotically conservative mappings on metric spaces.
We first start with the basics for the Gromov product, and we next
develop the properties of asymptotically conservative mappings.
We will prove Theorem \ref{thm:arbitrary_quotient_space} in \S\ref{subsec:Monoild_semingroups}.
In \S\ref{sec:comparison_with_coarse_geometry} and \S\ref{sec:stable_at_infinity},
we will discuss a relation between our geometry and the coarse geometry.
%

From \S\ref{sec:teichmuller_theory} to \S\ref{sec:Null_sets_section},
we devote to prepare for the proofs of
Theorems \ref{thm:Teichmullerspace} and \ref{thm:Teichmullerspace_homothety}.
In \S\ref{sec:teichmuller_theory},
we give basic notions of Teichm\"uller theory
including the definitions of Teichm\"uller space,
measured foliations and extremal length.
In \S\ref{sec:Thurston_theory_extremal_length},
we recall our unification theorem for extremal length geometry on Teichm\"uller space
via intersection number.
One of the key for proving our rigidity theorem is
to characterize the null sets for points in the GM-cone
 (cf. Theorem \ref{thm:null_set})
The characterization is also applied to proving a rigidity theorem of holomorphic disks
in the Teichm\"uller space (cf. \cite{Mi6}).
In \S\ref{sec:action_Reduced_boundary},
We define the reduced Gardner-Masur compactification
and study the action of asymptotically conservative mappings on the reduced Gardiner-Masur compactification.
In \S\ref{sec:Rigidity_asymptotic_teichmuller},
we will prove Theorems \ref{thm:Teichmullerspace} and \ref{thm:Teichmullerspace_homothety}.
%

\subsection{Acknowledgements}
The author would like to express his hearty gratitude
to Professor Ken'ichi Ohshika for fruitful discussions
and his constant encouragements.
He also thanks Professor Yair Minsky for informing him about
a work by Athreya, Bufetov, Eskin and Mirzakhani in \cite{ABEM}.
The author also thanks Professor Athanase Papadopoulos for his kindness and useful comments.


\section{Asymptotically conservative with the Gromov product}
\label{sec:asymptotic_conservative}

\subsection{Basics of the Gromov product}
\label{subsec:Gromov_product}
Let $(X,d_X)$ be a metric space.
%
The following is known
for $x_1$,
$x_2$,
$x_3$,
$z_1$,
$w_1\in X$:
\begin{align}
\gromov{x_1}{x_2}{z_1} &\ge 0 \\
\gromov{x_1}{x_2}{z_1} &\le \min\{d_X(z_1,x_1),d_X(z_1,x_2)\}
\label{eq:Gromov_product_1} \\
\gromov{x_1}{x_1}{z_1}&=d_X(z_1,x_1)
\label{eq:Gromov_product_2} \\
|\gromov{x_1}{x_2}{z_1}-\gromov{x_1}{x_2}{w_1}|
&\le d_X(z_1,w_1)
\label{eq:Gromov_product_3} \\
|\gromov{x_1}{x_2}{z_1}-\gromov{x_1}{x_3}{z_1}|
&\le d_X(x_2,x_3).
\label{eq:Gromov_product_4}
\end{align}

\subsection{Sequences converging at infinity}
We notice the following.

\begin{remark}[Basic properties]
\label{remark:accompanying_sequence}
Let $X$ be a metric space.
The following hold:
\begin{enumerate}
\item
\label{item:BasicProperties1}
The relation ``visually indistinguishable"
is reflexive
on $\sq{X}$:
${\bf x}\in \sq{X}$
if and only if
${\bf x}\in \visualindist({\bf x})$.
\item
\label{item:BasicProperties2}
The relation ``visually indistinguishable"
is symmetric on $\sq{X}$:
If ${\bf z}\in \visualindist({\bf x})$,
then ${\bf x}\in \visualindist({\bf z})$.
\item
\label{item:BasicProperties3}
The relation ``visually indistinguishable"
is not transitive
in general.
Namely,
it is possible that
$\visualindist({\bf z})\ne  \visualindist({\bf x})$
for some unbounded sequences ${\bf x}, {\bf z}$
with $\visualindist({\bf x})\cap \visualindist({\bf z})\ne \emptyset$.
\item
\label{item:BasicProperties4}
For ${\bf x}\in X^{\mathbb{N}}$,
any subsequence ${\bf z}'$ of ${\bf z}\in \visualindist({\bf x})$
is also in $\visualindist({\bf x})$.
\item
\label{item:BasicProperties5}
Any subsequence ${\bf x}'$ of ${\bf x}\in X^{\mathbb{N}}$
satisfies 
$\visualindist({\bf x})\subset \visualindist({\bf x}')$.
\end{enumerate}
\end{remark}
Indeed,
\eqref{item:BasicProperties1} and
\eqref{item:BasicProperties2},
follow from the definitions.
Notice that for any subsequences ${\bf x}'\subset {\bf x}$ and ${\bf y}'\subset {\bf y}$,
it holds
\begin{equation}
\label{eq:subsequence-larger}
\gromov{{\bf x}}{{\bf y}}{}\le \gromov{{\bf x}'}{{\bf y}'}{}.
\end{equation}
In particular,
any subsequence of a sequnce converging at infinity also converges at infinity.
\eqref{item:BasicProperties4}
and \eqref{item:BasicProperties5}
follow from \eqref{eq:subsequence-larger}.
For \eqref{item:BasicProperties3},
we will see that on Teichm\"uller space $\teich$ equipped with 
the Teichm\"uller distance,
the relation ``visually indistinguishable" does not define
an equivalence relation on $\sq{\teich}$,
when the base surface is neither a torus with one hole nor
a sphere with four holes
(cf. \S\ref{sec:Remark_visually_indistinguishable}).

\subsection{Asymptotically conservative}
\label{subsec:asymptotically_conservative}
For metric spaces $X$ and $Y$,
we define
%
%
$$
\AC(X,Y)=\{\omega\in Y^{X}\mid \mbox{$\omega$ is
asymptotically conservative}\}
$$
(for the definition,
see \S\ref{subsec:definition}).
Set $\AC(X)=\AC(X,X)$.

\begin{proposition}
\label{prop:x-converges-infinity-omega-x-converges-infinity}
Let $\omega\in \AC(X,Y)$.
For a sequence
${\bf x}\in \squ{X}$,
${\bf x}\in \sq{X}$
if and only if $\omega({\bf x})\in \sq{Y}$.
\end{proposition}

\begin{proof}
Let
${\bf x}\in \squ{X}$.
Suppose first that ${\bf x}\in \sq{X}$.
From (\ref{item:BasicProperties1}) of Remark \ref{remark:accompanying_sequence},
${\bf x}\in \visualindist({\bf x})$.
Since $\omega$ is asymptotically conservative,
$\omega({\bf x})\in \omega(\visualindist({\bf x}))\subset
\visualindist(\omega({\bf x}))$
and $\omega({\bf x})\in \sq{Y}$.

Conversely,
assume that
$\omega({\bf x})\in \sq{Y}$.
Since $\omega({\bf x})\in \visualindist(\omega({\bf x}))$,
from the definition of asymptotically conservative mappings,
we have ${\bf x}\in \visualindist({\bf x})$
and 
hence ${\bf x}\in \sq{X}$.
\end{proof}

%

\begin{proposition}[Composition in $\AC$]
\label{prop:properties_quasimorphisms}
Let $X$,
$Y$
and $Z$ be metric spaces.
For $\omega_1\in \AC(Y,Z)$
and $\omega_2\in\AC(X,Y)$,
$\omega_1\circ \omega_2\in \AC(X,Z)$.
\end{proposition}

\begin{proof}
Let ${\bf x}\in \squ{X}$.
Then,
$$
\omega_1\circ \omega_2(\visualindist({\bf x}))
\subset 
\omega_1(\visualindist(\omega_2({\bf x})))
\subset
\visualindist(\omega_1\circ \omega_2({\bf x})).
$$
Let ${\bf z}\in \squ{X}$ with $\omega_1\circ \omega_2({\bf z})\in \visualindist(\omega_1\circ \omega_2({\bf x}))$.
Since $\omega_1$ is asymptotically conservative,
$\omega_2({\bf z})\in \visualindist(\omega_2({\bf x}))$.
Since $\omega_2$ is also asymptotically conservative again,
we have 
${\bf z}\in \visualindist({\bf x})$.
 \end{proof}

\subsection{Remark on closeness}
Recall that two mappings $\omega_1,\omega_2\in Y^{X}$
are close at infinity
if for any ${\bf x}^{1}$,
${\bf x}^{2}\in \sq{X}$,
it holds $\visualindist (\omega_1({\bf x}^{1}))=\visualindist (\omega_2({\bf x}^{2}))$
whenever $\visualindist ({\bf x}^{1})=\visualindist ({\bf x}^{2})$.
In particular,
such $\omega_{1}$ and $\omega_{2}$ satisfy
\begin{equation}
\label{eq:accompany_close_at_infinity}
\visualindist (\omega_1({\bf x}))=\visualindist (\omega_2({\bf x}))
\end{equation}
for all ${\bf x}\in \sq{X}$.

\begin{proposition}[Composition and closeness]
\label{prop:properties_quasimorphisms_close}
Let $X$,
$Y$
and $Z$ be metric spaces.
Let $\omega_1, \omega'_1\in Z^{Y}$
and
$\omega_2, \omega'_2\in Y^{X}$.
If $\omega_i$ and $\omega'_i$ are close at infinity for $i=1,2$,
$\omega_1\circ \omega_2$
is close to $\omega'_1\circ \omega'_2$ at infinity.
\end{proposition}

\begin{proof}
Let
${\bf x}^1,{\bf x}^2\in \sq{X}$
with $\visualindist ({\bf x}^1)=\visualindist ({\bf x}^2)$.
By definition,
$\visualindist(\omega_2({\bf x}^1))=\visualindist(\omega'_2({\bf x}^2))$,
and hence
$\visualindist(\omega_1\circ \omega_2({\bf x}^1))
=\visualindist(\omega'_1\circ \omega'_2({\bf x}^2))$.
\end{proof}

\subsection{Asymptotic surjectivity and Closeness at infinity}
A mapping $\omega\in Y^{X}$ is said to be \emph{asymptotically surjective}
if for any ${\bf y}\in\sq{Y}$,
there is ${\bf x}\in \sq{X}$ with $\visualindist({\bf y})
=\visualindist(\omega({\bf x}))$.
Let
$$
\ACAS(X,Y)
=\{\omega\in \AC(X,Y)\mid
\mbox{$\omega$ is asymptotically surjective}\}.
$$



\begin{proposition}
\label{prop:quasi-surjective_asymptotic}
Let $\omega\in \ACAS(X,Y)$.
For ${\bf x}^1, {\bf x}^2\in\sq{X}$,
if $\visualindist({\bf x}^2)\subset \visualindist({\bf x}^1)$,
then $\visualindist(\omega({\bf x}^2))\subset
\visualindist(\omega({\bf x}^1))$.
In particular,
if ${\bf x}^1$ and ${\bf x}^2$ are asymptotic,
so are $\omega({\bf x}^2)$ and $\omega({\bf x}^1)$.
\end{proposition}

\begin{proof}
Let ${\bf y}\in \visualindist (\omega({\bf x}^2))$.
Since $\omega$ is asymptotically surjective,
there is ${\bf z}\in \sq{X}$ such that
$\visualindist({\bf y})=\visualindist(\omega({\bf z}))$.
Since $\omega$ is asymptotically conservative
and
$\omega({\bf x}^2)\in \visualindist({\bf y})=\visualindist(\omega({\bf z}))$,
we have ${\bf x}^2\in \visualindist({\bf z})$
and
$$
{\bf z}\in\visualindist({\bf x}^2)
\subset \visualindist({\bf x}^1).
$$
Therefore,
${\bf x}^1\in \visualindist({\bf z})$
(cf. (\ref{item:BasicProperties2}) of Remark \ref{remark:accompanying_sequence}).
Hence we deduce
$$
\omega({\bf x}^1)\in \omega(\visualindist({\bf z}))
\subset
\visualindist(\omega({\bf z}))
= \visualindist({\bf y}),
$$
and ${\bf y}\in \visualindist(\omega({\bf x}^1))$.
\end{proof}

\begin{proposition}[Composition of mappings in $\ACAS$]
\label{prop:composition_as}
For $\omega_1\in \ACAS(Y,Z)$ and $\omega_2\in \ACAS(X,Y)$,
we have $\omega_1\circ \omega_2\in \ACAS(X,Z)$.
\end{proposition}

\begin{proof}
Let ${\bf z}\in \sq{Z}$.
By definition,
there are ${\bf y}\in \sq{Y}$ and ${\bf x}\in\sq{X}$
such that
$\visualindist({\bf z})=\visualindist(\omega_{1}({\bf y}))$
and 
$\visualindist({\bf y})=\visualindist(\omega_{2}({\bf x}))$.
Since $\omega_{1}$
is asymptotically surjective again,
from Proposition \ref{prop:quasi-surjective_asymptotic},
we conclude
$$
\visualindist({\bf z})=\visualindist(\omega_{1}({\bf y}))
=
\visualindist(\omega_{1}(\omega_{2}({\bf x})))
=\visualindist(\omega_{1}\circ \omega_{2}({\bf x}))
$$
and hence $\omega_{1}\circ \omega_{2}$ is asymptotically surjective.
\end{proof}

\begin{proposition}[Closeness is an equivalence relation on $\ACAS$]
\label{prop:GP_close_equivalence_relation}
Let $X$ and $Y$ be metric spaces.
The relation ``closeness at infinity" is an equivalence relation
on $\ACAS(X,Y)$.
\end{proposition}

\begin{proof}
Let ${\bf x}^1,{\bf x}^2\in \sq{X}$.
Suppose ${\bf x}^1$ and ${\bf x}^2$ are asymptotic.

\medskip
\noindent
(Reflexive law)\quad
This follows from Proposition \ref{prop:quasi-surjective_asymptotic}.

\medskip
\noindent
(Symmetric law)\quad
Take two mappings $\omega_1,\omega_2\in \ACAS(X,Y)$.
Since $\omega_1$ is close to $\omega_2$ at infinity,
$\visualindist(\omega_1({\bf x}^1))=\visualindist(\omega_2({\bf x}^2))$.
By interchanging the roles of ${\bf x}^1$ and ${\bf x}^2$,
$\visualindist(\omega_2({\bf x}^1))=\visualindist(\omega_1({\bf x}^2))$.
This means that $\omega_2$ is close to $\omega_1$ at infinity.

\medskip
\noindent
(Transitive law)\quad
Take three mappings $\omega_1,\omega_2,\omega_3\in \ACAS(X,Y)$.
Suppose that $\omega_i$ is close to $\omega_{i+1}$ at infinity ($i=1,2$).
Then,
from \eqref{eq:accompany_close_at_infinity},
$$
\visualindist(\omega_1({\bf x}^1))=\visualindist(\omega_2({\bf x}^1))
=\visualindist(\omega_3({\bf x}^2))
$$
and hence,
$\omega_1$
is close to $\omega_3$ at infinity.
\end{proof}

\subsection{Invertibility and Asymptotic quasi-inverse}
Define
$$
\ACINV(X,Y)=\{\omega\in \AC(X,Y)\mid \mbox{$\omega$ is invertible}\}
$$
(for the definition,
see \S\ref{subsec:definition}).
Set $\ACINV(X)=\ACINV(X,X)$ as the introduction.
Notice that any $\omega\in \ACINV(X,Y)$
admits an asymptotic quasi-inverse $\omega'\in \ACINV(Y,X)$,
and $\omega$ is also an asymptotic quasi-inverse of $\omega'$.

\begin{proposition}[Invertibility implies asymptotic-surjectivity]
\label{prop:invertible_means_surjectivity}
For any metric spaces $X$ and $Y$,
$\ACINV(X,Y)\subset \ACAS(X,Y)$.
\end{proposition}

\begin{proof}
Let $\omega\in \ACINV(X,Y)$.
Let $\omega'$ be an asymptotic quasi-inverse of $\omega$.
Let ${\bf y}\in \sq{Y}$.
Set ${\bf x}=\omega'({\bf y})$.
Since $\omega'$ is asymptotically conservative,
${\bf x}\in \sq{X}$.
Since $\omega\circ \omega'$ is close to the identity mapping on $Y$
at infinity,
from \eqref{eq:accompany_close_at_infinity},
$$
\visualindist(\omega({\bf x}))=\visualindist(\omega\circ \omega'({\bf y}))
=\visualindist({\bf y}).
$$
Therefore, 
we conclude $\omega\in \ACAS(X,Y)$.
\end{proof}

\begin{proposition}[Composition of mappings in $\ACINV$]
\label{prop:composition}
For $\omega_1\in \ACINV(Y,Z)$ and $\omega_2\in \ACINV(X,Y)$,
we have $\omega_1\circ \omega_2\in \ACINV(X,Z)$.
\end{proposition}

\begin{proof}
Let $\omega'_i$ be an asymptotic quasi-inverse of $\omega_i$ for $i=1,2$.
Suppose ${\bf z}^1,{\bf z}^2\in \sq{Z}$ are asymptotic.
From Propositions \ref{prop:x-converges-infinity-omega-x-converges-infinity} and
 \ref{prop:properties_quasimorphisms},
${\bf x}^i=\omega'_2\circ \omega'_1({\bf z}^i)\in \sq{X}$ for $i=1,2$.
Since $\omega_2\circ \omega'_2$ is close to the identity mapping on $Y$,
$$
\visualindist(\omega_2({\bf x}^1))
=\visualindist(\omega_2(\omega'_2\circ \omega'_1({\bf z}^1)))
=\visualindist(\omega_2\circ \omega'_2(\omega'_1({\bf z}^1)))
=
\visualindist(\omega'_1({\bf z}^1)).
$$
From Proposition \ref{prop:invertible_means_surjectivity},
$\omega'_1$ is asymptotically surjective,
and from Proposition \ref{prop:quasi-surjective_asymptotic},
$\omega'_1({\bf z}^1)$ and $\omega'_1({\bf z}^2)$ are asymptotic.
Therefore,
$\visualindist(\omega_2({\bf x}^1))=\visualindist(\omega'_1({\bf z}^2))$.
Since $\omega_1$ is also asymptotically surjective,
by applying Proposition \ref{prop:quasi-surjective_asymptotic} again,
we have
\begin{align*}
\visualindist(\omega_1\circ\omega_2({\bf x}^1))
&=
\visualindist(\omega_1(\omega_2({\bf x}^1)))\\
&=\visualindist(\omega_1(\omega'_1({\bf z}^2)))
=\visualindist(\omega_1\circ \omega'_1({\bf z}^2))
=\visualindist({\bf z}^2)
\end{align*}
since $\omega_1\circ \omega'_1$ is asymptotically close to the identity mapping on $Z$ at infinity.
Therefore,
$$
\visualindist((\omega_1\circ \omega_2)\circ (\omega'_2\circ \omega'_1)({\bf z}^1))
=\visualindist(\omega_1\circ \omega_2({\bf x}^1))=\visualindist({\bf z}^2),
$$
which means that $(\omega_1\circ \omega_2)\circ (\omega'_2\circ \omega'_1)$
is close to the the identity mapping on $Z$ at infinity.
By the same argument,
we can see that $(\omega'_2\circ \omega'_1)\circ (\omega_1\circ \omega_2)$
is close to the identity mapping on $X$.
Therefore,
$\omega'_2\circ \omega'_1$ is an asymptotic quasi-inverse of $\omega_1\circ \omega_2$
and $\omega_1\circ \omega_2\in \ACINV(X,Z)$.
\end{proof}

\begin{proposition}[Stability of $\ACINV$ in $\ACAS$]
\label{prop:relation_closeness_at_infinity}
Let $\omega_1,\omega_2\in \ACAS(X,Y)$.
Suppose that $\omega_1$ and $\omega_2$ are close at infinity.
If $\omega_1\in \ACINV(X,Y)$,
so is $\omega_2$.
In addition,
any asymptotic quasi-inverse of $\omega_1$ is also that of $\omega_2$.
\end{proposition}

\begin{proof}
Let $\omega'_1$ be an asymptotic quasi-inverse of $\omega_1$.
Suppose ${\bf x}^1,{\bf x}^2\in \sq{X}$ are asymptotic.
Since $\omega_1$ and $\omega_2$ are close at infinity,
$$
\visualindist(\omega_1({\bf x}^1))=\visualindist(
\omega_2({\bf x}^2)).
$$
Since $\omega'_1$ is asymptotically surjective,
by Proposition \ref{prop:quasi-surjective_asymptotic},
we have
\begin{equation}
\label{eq:quasi-inverse_1}
\visualindist(\omega'_1\circ \omega_2({\bf x}^2))
=
\visualindist(\omega'_1\circ \omega_1({\bf x}^1))
=
\visualindist({\bf x}^1).
\end{equation}

Suppose that ${\bf y}^1$,
${\bf y}^2\in \squ{Y}$
are asymptotic.
Since $\omega'_1$ is asymptotically surjective again,
$$
\visualindist(\omega'_1({\bf y}^1))
=
\visualindist(\omega'_1({\bf y}^2)).
$$
Since $\omega_1$ and $\omega_2$ are close at infinity,
we deduce
\begin{equation}
\label{eq:quasi-inverse_2}
\visualindist(\omega_2\circ \omega'_1({\bf y}^1))
=
\visualindist(\omega_1\circ \omega'_1({\bf y}^2))
=\visualindist({\bf y}^2).
\end{equation}
From \eqref{eq:quasi-inverse_1}
and \eqref{eq:quasi-inverse_2},
we conclude that $\omega'_1$ is an asymptotic quasi-inverse of $\omega_2$,
and hence $\omega_2\in \ACINV(X,Y)$.
\end{proof}
%
%
%
%

\subsection{Monoids and Semigroup congruence}
\label{subsec:Monoild_semingroups}
We have defined three kinds of classes of mappings between metric spaces.
By definition and Proposition \ref{prop:invertible_means_surjectivity},
the relation of the classes is given as
\begin{equation}
\label{eq:relation-ACs1}
\ACINV(X,Y)
\subset
\ACAS(X,Y)
\subset 
\AC(X,Y)
(\subset Y^{X})
\end{equation}
for metric spaces $X$ and $Y$.

The following theorem follows from
Propositions \ref{prop:properties_quasimorphisms},
\ref{prop:composition_as},
and \ref{prop:composition}.

\begin{theorem}
\label{thm:properties_quasimorphisms_identity}
$\AC(X)$ admits a canonical monoid structure with respect to
the composition of mappings.
The identity element of $\AC(X)$ is the identity mapping on $X$.
In addition, $\ACAS(X)$ and $\ACINV(X)$ are submonoids of $\AC(X)$.
\end{theorem}

Let $G$ be a semigroup.
A \emph{semigroup congruence} is an equivalence relation $\sim$ on $G$
with the property that for $x,y,z,w\in G$,
$x\sim y$ and $z\sim w$ imply
$xz\sim yw$.
Then,
the congruence classes
$$
G/\sim=\{[g]\mid g\in G\}
$$
is also a semigroup with the product $[g_1][g_2]=[g_1g_2]$.
We call $G/\sim$ the \emph{quotient semigroup} of $G$
with the semigroup congruence $\sim$.

We define a relation on $\ACAS(X)$
by using the closeness at infinity.
Namely,
for two $\omega_1$ and $\omega_2\in \ACAS(X)$,
$\omega_1$ is \emph{equivalent} to $\omega_2$
if $\omega_1$ is close to $\omega_2$ at infinity.
From Propositions
\ref{prop:properties_quasimorphisms_close}
and
\ref{prop:GP_close_equivalence_relation},
this relation is a subgroup congruence on $\ACAS(X)$.
We define the quotient monoid by
$$
\ACmonoid(X)=\ACAS(X)/(\mbox{close at infinity}).
$$
We also define the semigroup congruence on $\ACINV(X)$
in the same procedure,
and obtain the quetient semigroup by
$$
\ACGroup(X)=\ACINV(X)/(\mbox{close at infinity}).
$$
As a result,
we summarize as follows.

\begin{theorem}[Group $\ACGroup(X)$]
\label{thm:quotient_semigroup}
Let $X$ be a metric space.
Then,
the quotient semigroup $\ACGroup(X)$ is a group.
The identity element of $\ACGroup(X)$ is the congruence class of the identity mapping,
and the inverse of the congruence class $[\omega]$ of $\omega\in \ACINV(X)$ is the congruence class
of an asymptotic quasi-inverse of $\omega$.
\end{theorem}

\begin{corollary}
Let $X$ and $Y$ be metric spaces.
Let
$\omega\in \ACINV(X,Y)$ 
and $\omega'$ an asymptotic quasi-inverse of $\omega$.
Then,
the mapping
\begin{align*}
&\ACAS(X)\ni f\mapsto \omega\circ f\circ \omega'\in \ACAS(Y)
\\
&\ACINV(X)\ni f\mapsto \omega\circ f\circ \omega'\in \ACINV(Y)
\end{align*}
induces isomorphisms
\begin{align*}
\ACmonoid(X)\ni [f]\mapsto [\omega\circ f\circ \omega']
\in \ACmonoid(Y)
\\
\ACGroup(X)\ni [f]\mapsto [\omega\circ f\circ \omega']
\in \ACGroup(Y).
\end{align*}
\end{corollary}

Notice from Proposition \ref{prop:relation_closeness_at_infinity}
that any equivalence class in $\ACGroup(X)$
consists of elements in $\ACINV(X)$.
Hence,
we conclude the following.
\begin{theorem}
The inclusion
$\ACINV(X)\hookrightarrow \ACAS(X)$
induces a monoid monomorphism
\begin{equation}
\label{eq:inclusion-acgroup-acmonoid}
\ACGroup(X)\hookrightarrow \ACmonoid(X).
\end{equation}
In other words,
$\ACGroup(X)$ is a subgroup of $\ACmonoid(X)$.
\end{theorem}
The monomorphism \eqref{eq:inclusion-acgroup-acmonoid}
could be an isomorphism.
The author does not know whether it is true or not in general.

\section{Comparison with the coarse geometry}
\label{sec:comparison_with_coarse_geometry}
\subsection{Backgrounds from the coarse geometry}
\subsubsection{Parallelism}
Two mappings $\omega,\xi\in Y^{X}$ between metric spaces
are said to be \emph{parallel} if and only if
$$
\sup_{x\in X}d_Y(\omega(x),\xi(x))<\infty
$$
(cf. \S1.A' in \cite{Gromov2}).
The ``parallelism'' defines an equivalence relation 
on any subclass in $Y^{X}$.
If two mappings $\omega,\xi\in Y^{X}$ are parallel,
\begin{align}
\sup_{x,z\in X}
|\gromov{\omega(x)}{\omega(z)}{y_0}^Y
-
\gromov{\xi(x)}{\xi(z)}{y_0}^Y|
<\infty
\label{eq:proof-omega-xi}\\
\sup_{x\in X,y\in Y}
|\gromov{\omega(x)}{y}{y_0}^Y
-
\gromov{\xi(x)}{y}{y_0}^Y|
<\infty.
\label{eq:proof-omega-xi2}
\end{align}
From \eqref{eq:proof-omega-xi2},
for any sequence ${\bf x}\in \squ{X}$,
it holds
\begin{equation}
\visualindist(\omega({\bf x}))=\visualindist(\xi({\bf x})).
\label{eq:proof-omega-xi3}
\end{equation}

\subsubsection{Quasi-isometries}
A mapping $\omega\in Y^{X}$ satisfies
$$
d_{Y}(\omega(x_{1}),\omega(x_{2}))\le Kd_{X}(x_{1},x_{2})+D
$$
for all $x_{1},x_{2}\in X$,
we call $\omega$ a \emph{coarsely $(K,D)$-Lipschitz}.
If $\omega\in Y^X$ satisfies
$$
\frac{1}{K}d_{X}(x_{1},x_{2})-D\le d_{Y}(\omega(x_{1}),\omega(x_{2}))
$$
for all $x_{1},x_{2}\in X$,
we call $\omega$ a \emph{coarsely co-$(K,D)$-Lipschitz}.
A mapping $\omega\in Y^X$ is said to be $(K,D)$-\emph{quasi-isometry}
if $\omega$ is both coarsely $(K,D)$-Lipschitz and coarsely co-$(K,D)$-Lipschitz.
A mapping $\omega\in Y^{X}$ is $D$-\emph{cobounded}
if the $D$-neighborhood of the image of $X$ under $\omega$ coincides with $Y$.
A \emph{quasi-inverse}
of a mapping $\omega\in Y^X$ is a mapping $\omega'\in X^Y$ such that 
$\omega'\circ \omega$ and $\omega\circ \omega'$ are parallel to the identity mappings.
Usually,
quasi-inverses are assumed to be quasi-isometry.
However, we do not assume so for our purpose.
Any quasi-inverse of a quasi-isometry is automatically a quasi-isometry.
Any cobounded quasi-isometry admits a quasi-inverse.

Let $\qi(X,Y)$ be the set of cobounded quasi-isometries from $X$ to $Y$.
Then,
$\qi(X)={\rm QI}(X,X)$ admits a monoid structure
defined by composition.
One can easily check that the parallelism is a subgroup congruence on $\qi(X)$.
Hence we have a quotient group defined by
\begin{equation}
\label{eq:qi-group}
\qigroup(X)=\qi(X)/(\mbox{parallelism}).
\end{equation}
The group $\qigroup(X)$ is a central object in the coarse geometry
(cf. \S I.8 in \cite{BF})

\subsection{Asymptotically conservative mappings in the coarse geometry}
\label{subsec:asymptoticallyconservativeincoarsegeometry}
%
%

Recall that a mapping $\omega\in Y^{X}$ is called \emph{weakly asymptotically conservative}
if 
$
\omega(\visualindist({\bf x}))\subset \visualindist(\omega_{1}({\bf x}))
$
for all sequence ${\bf x}\in \squ{X}$.
Let 
$$
\AC^{w}(X,Y)=\{\omega\in Y^{X}\mid \mbox{$\omega$ is weakly asymptotically conservative}\}.
$$
Set $\AC^{w}(X)=\AC^{w}(X,X)$.
From \eqref{eq:relation-ACs1},
we have
\begin{equation}
\label{eq:relation-ACs2}
\ACINV(X,Y)
\subset
\ACAS(X,Y)
\subset 
\AC(X,Y)
\subset
\AC^{w}(X,Y)
(\subset Y^{X})
\end{equation}
for metric spaces $X$ and $Y$.

A class $M$ of $Y^{X}$ is said to be \emph{stable under parallelism}
if a mapping $\omega\in Y^{X}$ is parallel to some $\xi\in M$,
then $\omega\in M$.

\begin{proposition}[Stability under parallelism]
\label{prop:stability}
All classes
$\AC^{w}(X,Y)$,
$\AC(X,Y)$,
$\ACAS(X,Y)$ and $\ACINV(X,Y)$
are stable under the parallelism.
\end{proposition}

\begin{proof}
Suppose that $\omega,\xi\in Y^{X}$ are parallel.

\medskip
\noindent
(i)\quad
Suppose $\xi\in \AC^{w}(X,Y)$.
Let ${\bf x}\in \squ{X}$.
Take ${\bf z}\in \visualindist({\bf x})$.
Since
$\xi$ is weakly asymptotically conservative,
$\xi({\bf z})\in \xi(\visualindist({\bf x}))\subset \visualindist(\xi({\bf x}))$.
From \eqref{eq:proof-omega-xi},
we have $\omega({\bf z})\in \visualindist(\omega({\bf x}))$
and
$\omega(\visualindist({\bf x}))\subset \visualindist(\omega({\bf x}))$.
Hence $\omega\in \AC^{w}(X,Y)$.

\medskip
\noindent
(ii)\quad
Suppose $\xi\in \AC(X,Y)$.
From the argument above,
$\omega\in \AC^{w}(X,Y)$.
Suppose a sequence ${\bf z}$ in $X$ satisfies $\omega({\bf z})\in \visualindist(\omega({\bf x}))$.
From \eqref{eq:proof-omega-xi} again,
$\xi({\bf z})\in \visualindist(\xi({\bf x}))$.
Since $\xi$ is asymptotically conservative,
${\bf z}\in \visualindist({\bf x})$
and $\omega\in \AC(X,Y)$.

\medskip
\noindent
(iii)\quad
Suppose $\xi\in \ACAS(X,Y)$.
Let ${\bf y}\in \sq{Y}$.
Take ${\bf x}\in \sq{X}$ such that $\visualindist({\bf y})=\visualindist(\xi({\bf x}))$.
From \eqref{eq:proof-omega-xi3},
we have
$$
\visualindist({\bf y})=\visualindist(\xi({\bf x}))
=\visualindist(\omega({\bf x})),
$$
which implies $\omega\in \ACAS(X,Y)$.

\medskip
\noindent
(iv)\quad
Suppose $\xi\in \ACINV(X,Y)$.
Let $\xi'\in \ACINV(Y,X)$ be an asymptotic quasi-inverse of $\xi$.
Let ${\bf x}^{1},{\bf x}^{2}\in \sq{X}$
with $\visualindist({\bf x}^{1})=\visualindist({\bf x}^{2})$.
Since $\xi'$ is asymptotically surjective,
by Proposition \ref{prop:quasi-surjective_asymptotic}
and \eqref{eq:proof-omega-xi3},
we have
$$
\visualindist({\bf x}^{1})=\visualindist(\xi'\circ \xi({\bf x}^{2}))
=\visualindist(\xi'\circ \omega({\bf x}^{2})),
$$
and hence $\xi'\circ \omega$ is close to the identity mapping on $X$.

To prove the converse,
we notice from the above that 
$\omega$ is asymptotically surjective and asymptotically conservative
from Proposition \ref{prop:invertible_means_surjectivity}.
Then,
by Propositions
\ref{prop:properties_quasimorphisms} and \ref{prop:composition_as},
$\omega\circ \xi'$ is also asymptotically surjective and
asymptotically conservative.

Let ${\bf y}^{1},{\bf y}^{2}\in \sq{Y}$ with
with $\visualindist({\bf y}^{1})=\visualindist({\bf y}^{2})$.
Then,
since $\xi'({\bf y}^{i})\in \sq{X}$ for $i=1,2$,
by Proposition \ref{prop:quasi-surjective_asymptotic} and \eqref{eq:proof-omega-xi3}
again,
we have
$$
\visualindist(\xi'({\bf y}^{1}))=\visualindist(\xi'({\bf y}^{2}))
$$
and 
$$
\visualindist(\omega\circ \xi'({\bf y}^{1}))=\visualindist(\xi\circ \xi'({\bf y}^{2}))
=\visualindist({\bf y}^{2}),
$$
and $\omega\circ \xi'$ is close to the identity mapping on $Y$.
Therefore,
we conclude that $\omega\in \ACINV(X,Y)$.
\end{proof}

The following proposition gives comparisons between items in rows
of the correspondence table in the introduction
(cf. Table \ref{table:1}).

\begin{proposition}[Comparison]
\label{prop:comparison}
\begin{enumerate}

\item
\label{item:AsymptoticallyConservativeMappingsCoarseGeometry2}
A cobounded asymptotically conservative mapping
is asymptotically surjective.

\item
\label{item:AsymptoticallyConservativeMappingsCoarseGeometry4}
If two asymptotically surjective,
asymptotically conservative mappings are parallel,
they are close at infinity.

\item
\label{item:AsymptoticallyConservativeMappingsCoarseGeometry3}
If $\omega\in Y^X$ admits a quasi-inverse $\omega'\in X^Y$
in the sense of the coarse geometry,
$\omega'\circ \omega$ and $\omega\circ \omega'$
are close to the identity mappings on $X$ and $Y$ at infinity,
respectively.
%
\end{enumerate}
\end{proposition}

\begin{proof}
(\ref{item:AsymptoticallyConservativeMappingsCoarseGeometry2})\quad
Let ${\bf y}=\{y_n\}_{n\in \mathbb{N}}\in \sq{Y}$.
Take
$x_n\in X$ such that $d_Y(\omega(x_n),y_n)\le D_0$
where $D_0>0$ is independent of $n$.
Since
$$
|\gromov{\omega({\bf x})}{\omega({\bf x})}{y_0}^Y
-
\gromov{{\bf y}}{{\bf y}}{y_0}^Y|
\le 2D_0
$$
for $y_{0}\in Y$,
we have $\omega({\bf x})\in \sq{Y}$ and ${\bf x}\in \sq{X}$
from Proposition \ref{prop:x-converges-infinity-omega-x-converges-infinity}.
Since
$$
|\gromov{\omega({\bf x})}{{\bf z}}{y_0}^Y
-
\gromov{{\bf y}}{{\bf z}}{y_0}^Y|
\le D_0
$$
for every sequence ${\bf z}=\{z_n\}_{n\in \mathbb{N}}\in \squ{Y}$,
we obtain $\visualindist({\bf y})=\visualindist(\omega({\bf x}))$.
Thus,
$\omega$ is asymptotically surjective.

\medskip
\noindent
(\ref{item:AsymptoticallyConservativeMappingsCoarseGeometry4})\quad
Let $\omega_{1}$, $\omega_{2}\in \ACAS(X,Y)$.
Suppose that $\omega_{1}$ is parallel to $\omega_{2}$.
Take asymptotic sequences ${\bf x}^1,{\bf x}^2\in \sq{X}$.
Since $\omega_{2}$ is asymptotically surjective,
from Proposition \ref{prop:quasi-surjective_asymptotic}
and \eqref{eq:proof-omega-xi3},
we conclude that
$$
\visualindist(\omega_{1}({\bf x}^1))
=\visualindist(\omega_{2}({\bf x}^1))=\visualindist(\omega_{2}({\bf x}^2))
$$
and $\omega_{1}$ and $\omega_{2}$ are close at infinity.

\medskip
\noindent
(\ref{item:AsymptoticallyConservativeMappingsCoarseGeometry3})\quad
Since the identity mapping is asymptotically surjective and asymptotically conservative,
from Proposition \ref{prop:stability},
$\omega'\circ \omega$ and $\omega\circ \omega'$
are also asymptotically surjective and asymptotically conservative.
Hence,
From above (\ref{item:AsymptoticallyConservativeMappingsCoarseGeometry4}),
we conclude what we wanted.
%
\end{proof}

\subsection{Criterion for subclasses to be compatible in $\ACGroup$}
\label{subsection:Criterion}
Let $M$ be a subclass of $X^{X}$.
Consider the following conditions.
\begin{enumerate}
\item[(S1)]
$M$ is a monoid with the operation defined by composition,
and the parallelism is a semigroup congruence on $M$.
%
\item[(S2)]
Any element in $M$ is cobounded.
\item[(S3)]
Any element in $M$ admits a quasi-inverse in $M$ in the coarse geometry.
\end{enumerate}
Notice that the condition (S3) implies (S2).
Under the condition (S1),
the quotient set $\mathcal{M}=M/(\mbox{parallelism})$ has a canonical monoid structure,
and if $M$ satisifes all conditions,
$\mathcal{M}$ has a canonical group structure.
For instance,
the monoid $\qi(X)$ of cobounded self quasi-isometries on $X$ satisfies all conditions above.

\begin{proposition}[Criterion]
\label{prop:criterion_ACw-AC}
Let $M$ be a subclass of $X^{X}$
satisfying the condition {\rm (S1)} posed above.
\begin{enumerate}
\item
Suppose in addition that $M$ satisfies the condition {\rm (S2)} posed above.
When $M\subset \AC(X)$,
then $M\subset \ACAS(X)$.
The inclusion $M\hookrightarrow \ACAS(X)$ induces 
maps (as sets)
$$
\mathcal{M}=M/(\mbox{parallelism})
\to
M/(\mbox{close at infinity})
\to
\ACmonoid(X)
$$
such that the conposition of the maps
$\mathcal{M}\hookrightarrow \ACmonoid(X)$
is a monoid homomorphism.
\item
Suppose that $M$ satisfies the condition {\rm (S3)} posed above.
When $M\subset \AC^{w}(X)$,
then $M\subset \ACINV(X)$.
The inclusion $M\hookrightarrow \ACINV(X)$ induces 
maps (as sets)
\begin{equation}
\label{eq:MtoAC-group}
\mathcal{M}
\to
M/(\mbox{close at infinity})
\to
\ACGroup(X)
\end{equation}
such that the conposition of the maps
$\mathcal{M}\to \ACGroup(X)$
is a group homomorphism.
\end{enumerate}
\end{proposition}

\begin{proof}
(1)\quad
From (\ref{item:AsymptoticallyConservativeMappingsCoarseGeometry2}) of Proposition \ref{prop:comparison},
$M\subset \ACAS(X)$,
and the ``closeness at infinity''
is an equivalence relation on $M$ by Proposition \ref{prop:GP_close_equivalence_relation}.
Therefore,
from (\ref{item:AsymptoticallyConservativeMappingsCoarseGeometry4}) of Proposition \ref{prop:comparison},
we have well-defined mappings between quotient sets
\begin{align*}
\mathcal{M}
&\to
M/(\mbox{close at infinity})\\
&\to 
\ACAS(X)/(\mbox{close at infinity})
=\ACmonoid(X).
\end{align*}
From (\ref{item:AsymptoticallyConservativeMappingsCoarseGeometry4}) of Proposition \ref{prop:comparison} again,
the composition
$$
\mathcal{M}
\to
\ACmonoid(X)
$$
induces a monoid homomorphism.

\medskip
\noindent
(2)\quad
We first check that $M\subset \AC(X)$.
Let $\omega\in M$ and $\omega'\in M$ a quasi-inverse of $\omega$.
Let ${\bf z}$ be an unbounded sequence in $X$ with $\omega({\bf z})\in \visualindist(\omega({\bf x}))$.
Since $\omega'$ is weakly asymptotically conservative,
we have
$$
\omega'\circ \omega({\bf z})\in \omega'(\visualindist(\omega({\bf x})))
\subset \visualindist(\omega'\circ \omega({\bf x}))
=\visualindist({\bf x})
$$
and
$$
|\gromov{x}{z}{x_{0}}^{X}-\gromov{\omega'\circ \omega(x)}{\omega'\circ \omega(z)}{x_{0}}^{X}|=O(1)
$$
for all $x\in {\bf x}$ and $z\in {\bf z}$,
since is parallel to the identity mapping on $X$
and infinity from Proposition \ref{prop:quasi-surjective_asymptotic}
and Proposition \ref{prop:stability} and
(\ref{item:AsymptoticallyConservativeMappingsCoarseGeometry4}) of Proposition \ref{prop:comparison}.
Therefore, ${\bf z}\in \visualindist({\bf x})$,
and $\omega$ is asymptotically conservative.

%
Then,
by applying the same argument as above,
we have a mappings
$$
\mathcal{M}
\to
M/(\mbox{close at infinity})\\
\to 
\ACGroup(X).
$$
such that the composition
\begin{equation}
\label{eq:group-homo-M-to-AC}
\mathcal{M}
\to\ACGroup(X)
\end{equation}
is a monoid homomorphism.
From (\ref{item:AsymptoticallyConservativeMappingsCoarseGeometry3}) of Proposition \ref{prop:comparison},
any quasi-inverse of $\omega\in M$
corresponds to a asymptotic quasi-inverse of $\omega$
in $\ACINV(X)$
under the inclusion $\mathcal{M}\hookrightarrow \ACINV(X)$.
Hence \eqref{eq:group-homo-M-to-AC}
is a group homomorphism.
\end{proof}

\begin{corollary}[Criterion for quasi-isometries]
\label{coro:criterion-quasi-isometries}
Let $X$ be a metric space.
If $\qi(X)\subset \AC^w(X)$,
then the inclusion $\qi(X)
\hookrightarrow \ACINV(X)$
induces  a group homomorphism
$$
\qigroup (X)\to \ACGroup(X).
$$
\end{corollary}

\subsection{Remarks}
The notions of quasi-isometries and the asymptotically conservation
are indepenent in general:
\begin{enumerate}
\item
An asymptotically conservative mapping need not be a quasi-isometry.
Indeed,
let $X=[0,\infty)$ with $d_X(x_1,x_2)=|x_1-x_2|$.
Let $x_0=0$ be the reference point.
Then,
any increasing function $f\colon X\to X$ with $f(0)=0$
is asymptotically conservative.
\item
Meanwhile,
little is known as to when quasi-isometries become asymptotically conservative.
For instance,
any rough homothety is asymptotically conservative
(cf. \eqref{eq:K_D_homothety}).
Actually,
it follows from the following fact that
any rough homothety $\omega$ satisfies
\begin{equation}
\label{eq:homothetic_implies_rough_gromov_product}
|K\gromov{x_1}{x_2}{x_0}^X-
\gromov{\omega(x_1)}{\omega(x_2)}{y_0}^Y|
\le
D'\quad
(x_1,x_2\in X)
\end{equation}
for some $K,D'>0$.
\item
In general,
the homomorphism \eqref{eq:MtoAC-group} is not injective:
Take an increasing sequence $\{a_{n}\}_{n=0}^{\infty}\subset \mathbb{Z}$ such that $a_{0}=0$
and $a_{n+1}-a_{n}\to \infty$.
Consider a graph in $\mathbb{R}^{2}$ defined by
$$
X=[0,\infty)\times \{-1,1\}\cup \cup_{n=0}^{\infty}\{a_{n}\}\times [-1,1]\subset \mathbb{R}^{2}.
$$
Then,
$X$ is a metric space equipped with the graph metric
such that the length of the edges are measured by the Euclidean metric.
Let $x_{0}=(0,0)\in X$ as a base point.
In this case,
$\sq{X}=\squ{X}$
and $\visualindist({{\bf x}})=\sq{X}$ for all ${\bf x}\in \squ{X}$.
Hence,
$\ACGroup (X)$ is the trivial group.
Furthermore,
$X$ is WBGP in the sense of \S\ref{sec:stable_at_infinity},
and any quasi-isometry on $X$ is weakly asymptotically conservative
(cf. Proposition \ref{prop:relaxation-definition}).
Thus,
we have a homomorphism \eqref{eq:MtoAC-group} in this case.
Define an isometry $r$ on $X$ by $r(x,y)=(x,-y)$.
Then,
$r$ is not parallel to the identity mapping $id_{X}$,
and $id_{X}$ and $r$ are contained in the different classes in $\qigroup(X)$.

\item
In general,
the homomorphism \eqref{eq:MtoAC-group} is not surjective:
When $X=Y=\mathbb{D}$ equipped with the Poincar\'e distance,
$\ACGroup(\mathbb{D})$ is canonically isomorphic to
the group of homeomorphism on $\partial \mathbb{D}$ via extensions.
hence,
we can find an invertible asymptotically conservative mapping on $\mathbb{D}$ which is not parallel to
any cobounded quasi-isometry.
\end{enumerate}

\section{Relaxation of the definition}
\label{sec:stable_at_infinity}

\subsection{Metric spaces which are WBGP}
\label{subsec:stable_at_infinity}
Let $X$ be a metric space.
For ${\bf x}\in \squ{X}$,
we define
$$
\sqt{{\bf x}}=\{{\bf x}'\mid
\mbox{subsequences of ${\bf x}$
with ${\bf x}'\in \sq{X}$}\}.
$$
A metric space $X$ is called \emph{well-behaved at infinity with respect to the Gromov product}
(WBGP) if
$\sqt{{\bf x}}\ne \emptyset$
for all ${\bf x}\in \squ{X}$.

\subsubsection*{Examples}
The following are metric spaces which are WBGP:
\begin{enumerate}
\item
\label{item:example-WBGP1}
Proper geodesic spaces that are Gromov-hyperbolic (of infinite diameter).
\item
\label{item:example-WBGP2}
Teichm\"uller space equipped with the Teichm\"uller distance.
\item
\label{item:example-WBGP3}
The Cayley graphs for pairs $(G,S)$ of finitely generated infinite group $G$ and
a finite system $S$ of symmetric generators.
\end{enumerate}
Indeed,
\eqref{item:example-WBGP1} follows from the compactness of the Gromov's bordification
(compactification)
(e.g. Proposition 2.14 in \cite{KB}).
\eqref{item:example-WBGP2}
is proven at Proposition \ref{prop:T-space-WBGP}.

We check \eqref{item:example-WBGP3}.
Let $\Sigma(G,S)$ be the Caylay graph.
Let $F_{S}$ be the free group generated by $S$.
There is a canonical surjection $\pi\colon \Sigma(F_{S},S)\to \Sigma(G,S)$
induced by the quotient map $F_{S}\to G$.
Let ${\bf x}=\{x_{n}\}_{n\in \mathbb{N}}$ be an unbounded sequence in $\Sigma(G,S)$.
Take $y_{n}\in \Sigma(F_{S},S)$ such that $\pi(y_{n})=x_{n}$ and $d_{G}(id,x_{n})=d_{F_S}(id,y_{n})$.
Then ${\bf y}=\{y_{n}\}_{n\in \mathbb{N}}$ is an unbounded sequence in $G(F_{S},S)$,
and hence we can find a subsequence ${\bf y}'=\{y_{n_{j}}\}_{j}$ of ${\bf y}$
such that $\gromov{{\bf y}'}{{\bf y}'}{}=\infty$ from \eqref{item:example-WBGP1} above.
Since the projection $\pi$ is $1$-Lipschitz,
$d(x_{n_{j}},x_{n_{k}})\le d(y_{n_{j}},y_{n_{k}})$
and we have
$$
\gromov{\pi({\bf y}')}{\pi({\bf y}')}{id}\ge \gromov{{\bf y}'}{{\bf y}'}{id}=\infty.
$$
Thus $\pi({\bf y}')$ is a desired subsequence of ${\bf x}$.

\subsection{Properties}
We shall give a couple of properties of metric spaces which are WBGP.

\begin{proposition}
\label{subsec:stable_at_infinity_sequence}
Suppose a metric space $X$ is WBGP.
For any ${\bf x}\in \squ{X}$,
we have
$$
\visualindist({\bf x})=\cap_{{\bf x}'\in \sqt{{\bf x}}}\visualindist({\bf x}').
$$
\end{proposition}

\begin{proof}
From (\ref{item:BasicProperties5})  of Remark \ref{remark:accompanying_sequence},
we have
$$
\visualindist({\bf x})\subset \cap_{{\bf x}'\in \sqt{{\bf x}}}\visualindist({\bf x}').
$$
Let ${\bf z}\in \squ{X}-\visualindist({\bf x})$.
Then $\gromov{{\bf z}}{{\bf x}}{}<\infty$.
Since $X$ is WBGP,
we can take subsequences
${\bf z}'
\subset {\bf z}$
and
${\bf x}'\in \sqt{{\bf x}}$
such that
$\gromov{{\bf z}'}{{\bf x}'}{}<\infty$.
Hence ${\bf z}'\not\in \visualindist({\bf x}')$
and ${\bf z}\not\in \visualindist({\bf x}')$
from (\ref{item:BasicProperties4})  of Remark \ref{remark:accompanying_sequence}.
Therefore,
${\bf z}\not\in \cap_{{\bf x}'\in \sqt{{\bf x}}}\visualindist({\bf x}')$.
\end{proof}
 
\begin{proposition}[Relaxation of the definition]
\label{prop:relaxation-definition}
Let $X$ and $Y$ be metric spaces which are WBGP.
\begin{enumerate}
\item
A mapping $\omega\in Y^{X}$ is in $\AC^{w} (X,Y)$
if and only if
for any ${\bf x},{\bf z}\in \sq{X}$,
$\gromov{\omega({\bf x})}{\omega({\bf z})}{}=\infty$
whenever $\gromov{{\bf x}}{{\bf z}}{}=\infty$.
\item
A mapping $\omega\in Y^{X}$ is in $\AC (X,Y)$
if and only if
for any ${\bf x}, {\bf z}\in \sq{X}$,
$\gromov{\omega({\bf x})}{\omega({\bf z})}{}=\infty$
implies $\gromov{{\bf x}}{{\bf z}}{}=\infty$,
and vice versa.
\end{enumerate}
\end{proposition}

\begin{proof}
%
%
%
%
%
%
\medskip
\noindent
(1)\quad
The condition is paraphrased that $\omega(\visualindist({\bf x}))\subset \visualindist(\omega({\bf x}))$
for all ${\bf x}\in \sq{X}$.
Hence,
the ``only if'' part follows from the definition.
We show the ``if'' part.
Let ${\bf x}\in \squ{X}$.
Suppose to the contrary that there is ${\bf z}\in \visualindist({\bf x})$
such that $\omega({\bf z})\not\in \visualindist(\omega({\bf x}))$.
From Proposition \ref{subsec:stable_at_infinity_sequence},
there is ${\bf y}\in \sqt{\omega({\bf x})}$ such that
$\omega({\bf z})\not\in \visualindist({\bf y})$.
By taking subsequences ${\bf z}'\in \sqt{{\bf z}}$ and ${\bf x}'\in \sqt{{\bf x}}$ respectively,
we may assume that $\omega({\bf x}')\in \sqt{{\bf y}}\subset \sqt{\omega({\bf x})}$ and $\omega({\bf z}')\not\in \visualindist(\omega({\bf x}'))$.

On the other hand,
since ${\bf x}'\in \sqt{{\bf x}}\subset\sq{X}$,
from the condition (1),
we have
$\omega(\visualindist({\bf x}'))\subset \visualindist(\omega({\bf x}'))$.
Hence ${\bf z}'\not\in \visualindist({\bf x}')$,
which implies ${\bf z}'\not\in \visualindist({\bf x})$
because $\visualindist({\bf x})\subset \visualindist({\bf x}')$
from (\ref{item:BasicProperties5})
of Remark \ref{remark:accompanying_sequence}.
This contradicts to (\ref{item:BasicProperties4})
of Remark \ref{remark:accompanying_sequence}.
Thus, we conclude that
$\omega(\visualindist({\bf x}))\subset \visualindist(\omega({\bf x}))$.
%
%

\medskip
\noindent
(2)\quad
We only show the ``if'' part.
From (1) above,
$\omega\in \AC^{w}(X,Y)$.
Let ${\bf z}\in \squ{X}$ with $\omega({\bf z})\in \visualindist(\omega({\bf x}))$.
Suppose ${\bf z}\not\in \visualindist({\bf x})$.
From the argument in Proposition \ref{subsec:stable_at_infinity_sequence},
there is a subsequence ${\bf z}'\in \sqt{{\bf z}}$
with ${\bf z}'\not\in\visualindist({\bf x})$.
This means that
$\omega({\bf z}')\not\in\visualindist(\omega({\bf x}))$
from the assumption,
and hence $\omega({\bf z})\not\in \visualindist(\omega({\bf x}))$
from 
(\ref{item:BasicProperties4})
of Remark \ref{remark:accompanying_sequence}.
This is a contradiction.
\end{proof}

From Proposition \ref{prop:criterion_ACw-AC},
we conclude the following.

\begin{corollary}
Let $X$ is a metric space which is WBGP.
Let $M$ be a subclass of $X^{X}$ satisfying {\rm (S1)}
and {\rm (S3)}
in \S\ref{subsection:Criterion}.
Suppose that any $\omega\in M$ satisfies the condition
that $\gromov{\omega({\bf x})}{\omega({\bf z})}{}=\infty$
whenever $\gromov{{\bf x}}{{\bf z}}{}=\infty$
for all ${\bf x}$, ${\bf z}\in \sq{X}$.
Then,
$M\subset \ACINV(X)$ and the inclusion
$M\hookrightarrow \ACINV(X)$ induces a group homomorphism
$$
M/(\mbox{parallelism})\to \ACGroup (X).
$$
\end{corollary}

\section{Teichm\"uller theory}
\label{sec:teichmuller_theory}
In this section,
we recall basics in the Teichm\"uller theory.
For details,
the readers can refer to \cite{Ahlfors},
\cite{FLP},
\cite{IT}
and
\cite{Ivanov}.

\subsection{Teichm\"uller space}
\label{subsec:teichmuller-space}
The \emph{Teichm\"uller space} $\teich=\teich(S)$ of $S$
is the set of equivalence classes of marked Riemann surfaces
$(Y,f)$ where $Y$ is a Riemann surface of analytically finite type
and $f:{\rm Int}(S)\to Y$ is an orientation preserving homeomorphism.
Two marked Riemann surfaces $(Y_1,f_1)$
and $(Y_2,f_2)$ are \emph{Teichm\"uller equivalent}
if there is a conformal mapping $h:Y_1\to Y_2$
which is homotopic to $f_2\circ f_1^{-1}$.

Teichm\"uller space $\teich$ is topologized with a canonical complete distance,
called the \emph{Teichm\"uller distance $d_T$} (cf. \eqref{eq:Kerckhof_formula}).
It is known that the Teichm\"uller space $\teich=\teich(S)$
of $S$
is homeomorphic to $\mathbb{R}^{2\complexity{S}}$.

\begin{convention}
Throughout this paper,
we fix a conformal structure $X$ on $S$ and
consider $x_0=(X,id)$ as the base point of
the Teichm\"uller space $\teich$ of $S$.
\end{convention}

\subsection{Measured foliations}
\label{subsec:measured_foliations}
Let $\mathcal{S}$ be the set of homotopy classes of
non-trivial and non-peripheral simple closed curves on $S$.
Consider the set of weighted simple close curves
$\mathcal{WS}=\{t\alpha\mid t\ge 0,\alpha\in \mathcal{S}\}$,
where $t\alpha$ is the formal product between
$t\ge 0$ and $\alpha\in \mathcal{S}$.
We embed $\mathcal{WS}$ into the space
$\RR$
of non-negative functions on $\mathcal{S}$
by
\begin{equation}
\label{eq:embed_WS}
\mathcal{WS}\ni t\alpha\mapsto [\mathcal{S}\ni\beta
\mapsto t\,i(\alpha,\beta)]\in \RR
\end{equation}
where $i(\cdot,\cdot)$ is the geometric intersection number
on $\mathcal{S}$.
The closure $\mathcal{MF}$
of the image of the mapping \eqref{eq:embed_WS} is called
the \emph{space of measured foliations} on $S$.
The space $\RR$ admits
a canonical action of $\mathbb{R}_{> 0}$
by multiplication.
The quotient space
$$
\mathcal{PMF}=(\mathcal{MF}-\{0\})/\mathbb{R}_{>0}
\subset
\PR
=(\RR-\{0\})/\mathbb{R}_{>0}
$$
is said to be the \emph{space of projective measured foliations}.
By definition,
$\mathcal{MF}$ contains $\mathcal{WS}$ as a dense subset.
The intersection number function on $\mathcal{WS}$ defined by
$$
\mathcal{WS}\times \mathcal{WS}
\ni (t\alpha,s\beta)\mapsto ts\,i(\alpha,\beta)
$$
extends continuously on $\mathcal{MF}\times \mathcal{MF}$.
It is known that $\mathcal{MF}$ and $\mathcal{PMF}$
are homeomorphic to $\mathbb{R}^{2\complexity{S}}$
and $S^{2\complexity{S}-1}$,
respectively.

\medskip
\paragraph{{\bf Normal forms}}
Any $G\in \mathcal{MF}$ is represented by
a pair $(\mathcal{F}_G,\mu_G)$
of a singular foliation $\mathcal{F}_G$
and a transverse measure $\mu_G$ to $\mathcal{F}_G$.
The intersection number $i(G,\alpha)$ with $\alpha\in \mathcal{S}$
is obtained as
$$
i(G,\alpha)=\inf_{\alpha'\sim\alpha}\int_{\alpha'}d\mu_G.
$$
The \emph{support} $\supportMF{G}$ of a measured foliation $G$
is,
by definition,
the minimal essential subsurface containing the underlying foliation.
A measured foliation is said to be \emph{minimal}
if it intersects any 
curves in $\mathcal{S}$
in its support.

According to the structure of the underlying foliation,
any $G\in \mathcal{MF}$ has the \emph{normal form}:
Any measured foliation $G\in \mathcal{MF}$ is decomposed as
\begin{equation}\label{eq:decompositionMF}
G=G_1+G_2+\cdots+G_{m_1}+\beta_1+\cdots \beta_{m_2}
+\gamma_1+\cdots+\gamma_{m_3}.
\end{equation}
where $G_i$ is a minimal foliation in its support $X_i=\supportMF{G_{i}}$,
$\beta_j$ and $\gamma_k$ are simple closed curves
such that each $\beta_j$ cannot be deformed into any $X_i$
and $\gamma_k$ is homotopic to a component of $\partial X_i$
for some $i$
(cf. \S2.4 of \cite{Ivanov}).
In this paper,
we call $G_i$, $\beta_j$ and $\gamma_k$
a \emph{minimal component},
an \emph{essential curve},
and a \emph{peripheral curve} of $G$
respectively.

\subsection{Null sets of measured foliations}
\label{subsec:nullsets-measured-foliations}
For a measured foliation $G$,
we define
the \emph{null set} of $G$ by
\begin{equation}
\label{eq:nullset_pre}
\NMF(G)=\{F\in \mathcal{MF}\mid i(F,G)=0\}.
\end{equation}
We denote by $\trunc{G}$ the measured foliation
defined from $G$
by deleting the foliated annuli associated to the peripheral curves in $G$.
We here call $\trunc{G}$ the \emph{distinguished part of $G$ on nullity}.
Notice that $\trunc{(\trunc{G})}=\trunc{G}$.
The following might be well-known.
However,
we give a proof of the proposition for completeness.

\begin{proposition}[Null sets and Topologically equivalence]
\label{prop:measured_foliation_removing_peripehral}
Let $G,H\in \mathcal{MF}$.
Then, 
the following are equivalent.
\begin{itemize}
\item[{\rm (1)}]
$\NMF(G)=\NMF(H)$.
\item[{\rm (2)}]
$\trunc{G}$ is topologically equivalent to $\trunc{H}$.
\end{itemize}
In particular,
$\NMF(G)=\NMF(\trunc{G})$.
\end{proposition}

We say that two measured foliations $F_1$ and $F_2$ are
{\color{black}\emph{topologically equivalent}} if
the underlying foliations of $F_1$ and $F_2$ are modified by Whitehead operations to
foliations with trivalent singularities
such that the resulting foliations (without transversal measures) are isotopic
(cf. \S3.1 of \cite{Ivanov2}).

\begin{proof}
Suppose (1) holds.
We decompose $G$ as \eqref{eq:decompositionMF}:
$$
G=\sum_{i=1}^{m_1}G_i+\sum_{i=1}^{m_2}\beta_i
+\sum_{i=1}^{m_3}\gamma_i
$$
Since $i(G,H)=0$,
the decomposition of $H$ is represented as 
\begin{equation}
\label{eq:decompostionH3}
H=\sum_{i=1}^{m_1}H_i+\sum_{i=1}^{m_2}a_i\beta_i
+\sum_{i=1}^{m_1}\sum_{\gamma\subset \partial X_i}b_\gamma\gamma
+H_0
\end{equation}
where
$H_i$ is either topologically equivalent to $G_i$
or is $0$,
$a_i,b_\gamma\ge 0$
and $\supportMF{H_{0}}\subset X-\supportMF{G}$.
In the summation $\sum_{\gamma\subset \partial X_i}$ in \eqref{eq:decompostionH3},
$\gamma$ rums over all component of $\partial X_i$.
See Proposition 3.2 of Ivanov \cite{Ivanov}
or Lemma 3.1 of Papadopoulos \cite{Papadopoulos}.
Indeed,
Ivanov in \cite{Ivanov} works
under the assumption that each $G_i$ is a stable lamination
for some pseudo-Anosov mapping on $X_i$.
However,
the discussion of his proof can be applied to our case.

If $H_0\ne 0$,
there is an $\alpha\subset \mathcal{S}$ with $i(G,\alpha)=0$
but $i(H_0,\alpha)\ne 0$.
Since $\alpha\in \NMF(G)=\NMF(H)$ from the assumption,
this is a contradiction.
Hence $H_0=0$.

Suppose $a_i=0$ for some $i$.
Since $\beta_i$ is essential,
we can find an $\alpha\in \mathcal{S}$ such that
$i(G,\alpha)=i(\beta_i,\alpha)\ne 0$.
Such an $\alpha$ satisfies $i(H,\alpha)=a_ii(\beta_i,\alpha)=0$,
which is a contradiction.
With the same argument,
we can see that $H_i\ne 0$.
Thus,
\begin{align*}
\trunc{G}
&=\sum_{i=1}^{m_1}G_i+\sum_{i=1}^{m_2}\beta_i \\
\trunc{H}
&=\sum_{i=1}^{m_1}H_i+\sum_{i=1}^{m_2}a_i\beta_i
\end{align*}
are topologically equivalent.

Suppose (2) holds.
Let $F\in \NMF(G)$.
Consider the decomposition \eqref{eq:decompostionH3}
for $F$ instead of $H$,
one can easily deduce that $F\in \NMF(H)$.
\end{proof}

\subsection{Extremal length}
\label{subsec:extremal_length}
Let $X$ be a Riemann surface and let $A$ be a doubly connected domain
on $X$.
If $A$ is conformally equivalent to a round annulus $\{1<|z|<R\}$,
we define the \emph{modulus} of $A$ by
$$
\Mod(A)=\frac{1}{2\pi}\log R.
$$
\emph{Extremal length} of
a simple closed curve $\alpha$
on $X$ is defined by
\begin{equation}
\label{eq:gemetric_definition_extremal_length}
\ext_X(\alpha)=\inf\left\{\frac{1}{\Mod(A)}\mid
\mbox{the core curve of $A\subset X$ is homotopic to $\alpha$}
\right\}
\end{equation}
In \cite{Ker},
Kerckhoff showed that 
if we define the extremal length of $t\alpha\in \mathcal{WS}$
by
$$
\ext_X(t\alpha)=t^2\ext_X(\alpha),
$$
then the extremal length function $\ext_X$ on $\mathcal{WS}$
extends continuously to $\mathcal{MF}$.
For $y=(Y,f)\in \teich$ and $G\in \mathcal{MF}$,
we define
$$
\ext_y(G)=\ext_Y(f(G)).
$$
We define the \emph{unit sphere} in $\mathcal{MF}$ by
$$
\ZZ=\{F\in \mathcal{MF}\mid \ext_{x_0}(F)=1\}.
$$
The projection $\mathcal{MF}-\{0\}\to \mathcal{PMF}$
induces a homeomorphism $\mathcal{MF}_1\to \mathcal{PMF}$.

It is known that for any $G\in \mathcal{MF}$
and $y=(Y,f)\in \teich$,
there is a unique holomorphic quadratic differential
$J_{G,y}$ such that
$$
i(G,\alpha)=\inf_{\alpha'\sim f(\alpha)}\int_{\alpha'}|{\rm Re}\sqrt{J_{G,y}}|.
$$
Namely,
the \emph{vertical foliation} of $J_{G,y}$ is equal to $G$.
We call $J_{G,y}$ the \emph{Hubbard-Masur differential}
for $G$ on $y$
(cf. \cite{HM}).
The Hubbard-Masur differential $J_{G,y}=J_{G,y}(z)dz^2$
for $G$ on $y=(Y,f)$ satisfies
$$
\ext_y(G)=\|J_{G,y}\|=\iint_Y|J_{G,y}(z)|dxdy.
$$
In particular,
it is known that
\begin{equation}
\label{eq:JS_extremallength_norm_length}
\ext_y(\alpha)=\|J_{\alpha,y}\|=\frac{\ell_{J_{G,y}}(\alpha)^2}{\|J_{\alpha,y}\|}
\end{equation}
where $\ell_{J_{G,y}}(\alpha)$ is the length of the geodesic representative
homotopic to $f(\alpha)$ with respect to the singular flat metric
$|J_{\alpha,y}|=|J_{\alpha,y}(z)||dz|^2$.

\medskip
\noindent
\paragraph{{\bf Kerckhoff's formula}}
The Teichm\"uller distance $d_T$ is expressed by extremal length,
which we call \emph{Kerckhoff's formula}:
\begin{equation}
\label{eq:Kerckhof_formula}
d_T(y_1,y_2)=\frac{1}{2}\log \sup_{\alpha\in \mathcal{S}}
\frac{\ext_{y_2}(\alpha)}{\ext_{y_1}(\alpha)}
\end{equation}
(see \cite{Ker}).
%

\medskip
\noindent
\paragraph{{\bf Minsky's inequality}}
Minsky \cite{Minsky2} observed the following inequality,
which we recently call \emph{Minsky's inequality}:
\begin{equation}
\label{eq:Minsky_inequality}
i(F,G)^2\le \ext_y(F)\,\ext_y(G)
\end{equation}
for $y\in \teich$ and $F,G\in \mathcal{MF}$.
Minsky's inequality is sharp in the sense that
for any $y\in \teich$ and $F\in \mathcal{MF}$,
there is a unique $G\in \mathcal{MF}$ up to multiplication by a positive constant such that
$i(F,G)^2=\ext_y(F)\,\ext_y(G)$
(cf. \cite{GM}).

\subsection{Teichm\"uller rays}
\label{subsec:Teichmuller_rays}
Let $x=(X,f)\in \teich$
and $[G]\in \mathcal{PMF}$.
By the Ahlfors-Bers theorem,
we can define an isometric embedding
$$
[0,\infty)\ni t\mapsto R_{G,x}(t)\in \teich
$$
with respect to the Teichm\"uller distance
by assigning the solution of the Beltrami equation
defined by the Teichm\"uller Beltrami differential
\begin{equation} \label{eq:beltrami}
\tanh(t)\frac{|J_{G,x}|}{J_{G,x}}
\end{equation}
for $t\ge 0$.
We call $R_{G,x}$ the \emph{Teichm\"uller (geodesic) ray}
associated to $[G]\in \mathcal{PMF}$.
Notice that the differential \eqref{eq:beltrami} depends only on the projective class of $G$.
The \emph{exponential map}
\begin{equation}
\label{eq:exponential_maps}
\mathcal{PMF}\times [0,\infty)/(\mathcal{PMF}\times \{0\})\ni ([G],t)
\mapsto R_{G,t}(t)\in \teich
\end{equation}
which is a homeomorphism
(see also \cite{IT}).

\section{Thurston theory with extremal length}
\label{sec:Thurston_theory_extremal_length}
In this section,
we recall the unification of extremal length geometry
via intersection number
developed in \cite{Mi5}.
%

\subsection{Gardiner-Masur closure}
\label{subsec:GM_closure}
Consider a mapping
\begin{align*}
\tilde{\Phi}_{GM}&\colon
\teich\ni y\mapsto
[\mathcal{S}\ni \alpha\mapsto
\ext_y(\alpha)^{1/2}]
\in \RR \\
\Psi_{GM}&\colon
\teich\ni y\mapsto
[\mathcal{S}\ni \alpha\mapsto
e^{-d_T(x_0,y)}\ext_y(\alpha)^{1/2}]
\in \RR.
\end{align*}
Let $\proj\colon \RR-\{0\}\to \PR$
be the quotient mapping of the action.
In \cite{GM},
Gardiner and Masur showed that the mapping
$$
\Phi_{GM}=\proj\circ \Psi_{GM}=\proj\circ \tilde{\Phi}_{GM}:\teich\to \PR
$$
is an embedding with compact closure.
The closure $\cl{\teich}$ of the image is called the \emph{Gardiner-Masur closure} or the \emph{Gardiner-Masur compactification},
and the complement $\partial_{GM}\teich=\cl{\teich}-\Phi_{GM}(\teich)$
is called the \emph{Gardiner-Masur boundary}.
They also observed that the space $\mathcal{PMF}$ of projective measured
foliaitons is contained in $\partial_{GM}\teich$.


\subsection{Cones, the intersection number and the Gromov product}
\label{subsec:cones_intersectionNumber}
We define
\begin{align*}
\GmInv
&=\proj^{-1}(\cl{\teich})\cup \{0\}\subset \RR \\
\GmTeich
&=\proj^{-1}(\Phi_{GM}(\teich))\subset \RR \\
\Gmbdy
&=\proj^{-1}(\partial_{GM}\teich)\cup \{0\}\subset \RR.
\end{align*}
Since $\mathcal{PMF}\subset \partial_{GM}\teich$,
$\mathcal{MF}\subset \Gmbdy\subset \GmInv$.
From Proposition 1 of \cite{Mi5},
$\Psi_{GM}\colon \teich \to \GmInv$ extends to an injective continuous mapping on $\cl{\teich}$.
\begin{convention}
We denote by $[\mathfrak{a}]\in \cl{\teich}$ the projective class of
$\mathfrak{a}\in \GmInv$.
Unless otherwise stated,
we always identify $y\in \teich$ with the projective class $\Phi_{GM}(y)
=[\tilde{\Phi}_{GM}(y)]=[\Psi_{GM}(y)]$.
\end{convention}

In \cite{Mi5},
the author established the following \emph{unification}
of extremal length geometry via the intersection number.

\begin{theorem}[Theorem 1.1 in \cite{Mi5}]
\label{thm:main_realization}
Let $x_0\in \teich$ be the base point taken as above.
There is a unique continuous function
$$
i(\,\cdot\,,\,\cdot\,)\colon \GmInv\times \GmInv\to \mathbb{R}
$$
with the following properties.
\begin{itemize}
\item[{\rm (i)}]
$i(\tilde{\Phi}_{GM}(y),F)=\ext_y(F)^{1/2}$
for any
$y\in \teich$ and $F\in \mathcal{MF}$.
\item[{\rm (ii)}]
For $\mathfrak{a},\mathfrak{b}\in \GmInv$,
$i(\mathfrak{a},\mathfrak{b})=i(\mathfrak{b},\mathfrak{a})$.
\item[{\rm (iii)}]
For $\mathfrak{a},\mathfrak{b}\in \GmInv$ and $t,s\ge 0$,
$i(t\mathfrak{a},s\mathfrak{b})=ts\,i(\mathfrak{a},\mathfrak{b})$.
\item[{\rm (iv)}]
For any $y,z\in \teich$,
\begin{align*}
i(\tilde{\Phi}_{GM}(y),\tilde{\Phi}_{GM}(z)) &=\exp(d_T(y,z)) \\
i(\Psi_{GM}(y),\Psi_{GM}(z)) &=\exp(-2\gromov{y}{z}{x_0}).
\end{align*}
\item[{\rm (v)}]
For $F,G\in \mathcal{MF}\subset \GmInv$,
the value $i(F,G)$ is equal to the geometric intersection number
$I(F,G)$
between $F$ and $G$.
\end{itemize}
\end{theorem}

We define
the \emph{extremal length} of $\mathfrak{a}\in \GmInv$ on $y\in \teich$ by
\begin{equation}
\label{eq:extremallength_on_GMInv}
\ext_y(\mathfrak{a})=\sup_{F\in \mathcal{MF}-\{0\}}\frac{i(\mathfrak{a},F)^2}{\ext_y(F)}
\end{equation}
(cf. Corollary 4 in \cite{Mi5}).
One see that
\begin{equation}
\label{eq:extremallength_on_GMInv-1}
e^{-2d_{T}(x,y)}\ext_x(\mathfrak{a})
\le 
\ext_y(\mathfrak{a})
\le
e^{2d_{T}(x,y)}\ext_x(\mathfrak{a})
\end{equation}
(cf. (5.6) in \cite{Mi5}).
From \eqref{eq:Minsky_inequality}
and Gardiner-Masur's work in \cite{GM},
\eqref{eq:extremallength_on_GMInv} coincides with
the original extremal length when $\mathfrak{a}\in \mathcal{MF}$.
$\ext_y$ is continuous on $\GmInv$ and
satsfies
\begin{align}
e^{-d_T(x_0,y)}\ext_y(\Psi_{GM}(z))^{1/2}
&=\exp(-2\gromov{y}{z}{x_0})
=i(\Psi_{GM}(y),\Psi_{GM}(z)), 
\label{eq:intesection-number-1}\\
e^{-d_T(x_0,y)}\ext_y(\mathfrak{a})^{1/2}
&=i(\Psi_{GM}(y),\mathfrak{a})
\label{eq:intesection-number-2}
\end{align}
for $y,z\in \teich$
and $\mathfrak{a}\in \GmInv$
(cf. Theorem 4 and Proposition 7 in \cite{Mi5}).
The extremal length \eqref{eq:extremallength_on_GMInv}
also satisfies the following \emph{generalized Minsky inequality}:
\begin{equation}
\label{eq:Minsky_inequality_in_general}
i(\mathfrak{a},\mathfrak{b})^2\le
\ext_y(\mathfrak{a})\,\ext_y(\mathfrak{b})
\end{equation}
for all $y\in \teich$ and $\mathfrak{a},\mathfrak{b}\in \GmInv$
(cf. Corollary 3 in \cite{Mi5}).

\subsection{Intersection number with base point}
\label{subsec:Intersection_number_with_base_point}
We define
the \emph{intersection number with base point $x_0$}
by
\begin{equation}
\label{eq:weighted_intersection_number}
i_{x_0}(p,q)=i(\Psi_{GM}(p),\Psi_{GM}(q))
\end{equation}
for $p,q\in \cl{\teich}$ (cf. \S8.2 in \cite{Mi5}).
Since the intersection number is continuous,
so is $i_{x_0}$ on
the product $\cl{\teich}\times \cl{\teich}$.
From Theorem \ref{thm:main_realization},
the Gromov product
\begin{equation}
\label{eq:Gromov_product_boundary_continuous}
\gromov{y}{z}{x_0}=-\frac{1}{2}\log i_{x_0}(y,z)
\end{equation}
extends continuously
to $\cl{\teich}\times \cl{\teich}$
with values in the closed interval  $[0,\infty]$
(cf. Corollary 1 in \cite{Mi5}).


\begin{proposition}
\label{prop:T-space-WBGP}
Teichm\"uller space $(\teich,d_{T})$ is WBGP.
\end{proposition}

\begin{proof}
Let ${\bf x}=\{x_{n}\}_{n\in \mathbb{N}}\in \squ{\teich}$.
Since the Gardiner-Masur closure is compact,
we find a subsequence ${\bf x}'=\{x_{n(k)}\}_{k\in \mathbb{N}}$
and $p\in \partial_{GM}\teich$ such that $x_{n(k)}\to p$ as $k\to \infty$.
Since
$$
i_{x_{0}}(x_{n(k)},x_{n(l)})\to i_{x_{0}}(p,p)=0\quad (k,l\to \infty),
$$
we have ${\bf x}'\in \sq{\teich}$ from \eqref{eq:Gromov_product_boundary_continuous}.
\end{proof}

\begin{proposition}[Intersection number with base point]
\label{prop:IntersectionNumberBasePoint}
For any $[\mathfrak{a}],[\mathfrak{b}]\in \cl{\teich}$,
it holds
\begin{equation}
\label{eq:intersection_number_MF_GM}
i_{x_{0}}([\mathfrak{a}],[\mathfrak{b}])=
\frac{i(\mathfrak{a},\mathfrak{b})}{\ext_{x_{0}}(\mathfrak{a})^{1/2}\ext_{x_{0}}(\mathfrak{b})^{1/2}}.
\end{equation}
\end{proposition}

Notice that the intersection number in the right-hand side
of \eqref{eq:intersection_number_MF_GM} is the original intersection number on 
$\mathcal{MF}$
(cf. (v) of Theorem \ref{thm:main_realization}).

\begin{proof}[Proof of Proposition \ref{prop:IntersectionNumberBasePoint}]
Let $y\in \teich$.
Notice that
$$
\ext_{x_{0}}(\Psi_{GM}(y))=\exp(-2\gromov{x_{0}}{y}{x_{0}})=1.
$$
Since $\ext_{x_{0}}$ is continuous on $\GmInv$,
we have
\begin{equation}
\label{eq:image-psi-a}
\Psi_{GM}([\mathfrak{a}])=\frac{\mathfrak{a}}{\ext_{x_{0}}(\mathfrak{a})^{1/2}}.
\end{equation}
Therefore,
$$
i_{x_{0}}([\mathfrak{a}],[\mathfrak{b}])
=
i(\Psi_{GM}([\mathfrak{a}]),\Psi_{GM}([\mathfrak{b}])) \\
=\frac{i(\mathfrak{a},\mathfrak{b})}{\ext_{x_{0}}(\mathfrak{a})^{1/2}\ext_{x_{0}}(\mathfrak{b})^{1/2}}
$$
for $\mathfrak{a},\mathfrak{b}\in \GmInv$.
\end{proof}


\subsection{A short proof for non-Gromov hyperbolicity}
\label{sec:Remark_visually_indistinguishable}
We check that the relation
``visually indistinguishable" is not an equivalence relation
on $\sq{\teich}$ when $\complexity{S}\ge 2$.
This also implies that Teichm\"uller space $(\teich,d_T)$ is not Gromov hyperbolic.

Indeed,
let $\alpha,\beta,\gamma\in \mathcal{S}$ with $i(\alpha,\beta)=i(\alpha,\gamma)=0$,
but $i(\beta,\gamma)\ne 0$.
Consider sequences ${\bf x}=\{x_n\}_{n\in \mathbb{N}}$,
${\bf y}=\{y_n\}_{n\in \mathbb{N}}$ and
${\bf z}=\{z_n\}_{n\in \mathbb{N}}$
in $\teich$ with $x_n\to [\alpha]$,
$y_n\to [\beta]$ and $z_n\to [\gamma]$
in $\cl{\teich}$,
where the projective classes $[\alpha]$,
$[\beta]$ and
$[\gamma]$
are recognized as points in $\partial_{GM}\teich$.
Then,
\begin{align*}
i_{x_0}(x_n,y_n)
&\to i_{x_0}([\alpha], [\beta])=0 \\
i_{x_0}(x_n,z_n)
&\to i_{x_0}([\alpha], [\gamma])=0,
\end{align*}
but
$$
i_{x_0}(y_n,z_n)\to i_{x_0}([\beta], [\gamma])\ne 0.
$$
From \eqref{eq:Gromov_product_boundary_continuous},
these observations imply that ${\bf y},{\bf z}\in \visualindist({\bf x})$
but ${\bf y}\not\in \visualindist({\bf z})$.

\subsection{Subadditivity of the intersection number}
The intersection number has the following
\emph{subadditive property}.

\begin{lemma}[Subadditivity]
\label{lem:subadditivity}
Let $F,G\in \mathcal{MF}\subset \GmInv$ with $i(F,G)=0$.
Then,
for any $\mathfrak{a}\in \GmInv$
we have
\begin{equation}
\label{eq:subadditivity}
( i(\mathfrak{a},F)^2+i(\mathfrak{a},G)^2)^{1/2}
 \le i(\mathfrak{a},F+G)\le i(\mathfrak{a},F)+i(\mathfrak{a},G).
\end{equation}
\end{lemma}

\begin{proof}
Let $y\in \teich$.
Then,
we have
\begin{align*}
i(\tilde{\Phi}_{GM}(y),F+G)
&=
\ext_y(F+G)^{1/2}\\
&=\sup_{H\in \mathcal{MF}-\{0\}}\frac{i(H,F+G)}{\ext_y(H)^{1/2}}
=\sup_{H\in \mathcal{MF}-\{0\}}\frac{i(H,F)+i(H,G)}{\ext_y(H)^{1/2}} \\
&\le \sup_{H\in \mathcal{MF}-\{0\}}\frac{i(H,F)}{\ext_y(H)^{1/2}}
+\sup_{H\in \mathcal{MF}-\{0\}}\frac{i(H,G)}{\ext_y(H)^{1/2}} \\
&=\ext_y(F)^{1/2}+\ext_y(G)^{1/2}
=i(\tilde{\Phi}_{GM}(y),F)+i(\tilde{\Phi}_{GM}(y),G).
\end{align*}
Hence,
the right-hand side of \eqref{eq:subadditivity}
follows from the density of $\GmTeich$ in $\GmInv$.

We prove the left-hand side of \eqref{eq:subadditivity}.
We first show the case where $F$ and $G$ are rational.
Let $F=\sum_{i=1}^{N_1}t_i\alpha_i+\sum_{j=1}^{N_3}
u_j\gamma_j$ and
$G=\sum_{i=1}^{N_2}s_i\beta_i+\sum_{j=1}^{N_3}v_j\gamma_j$
where $\alpha_i,\beta_i,\gamma_j$ are mutually disjoint and distinct simple closed curves
and $t_i,s_i>0$ and $u_j,v_j\ge 0$.
Let $y\in \teich$ and $A_{\alpha_i}$,
$A_{\beta_i}$ $A_{\gamma_j}$ be the characteristic annuli
for $\alpha_i$,
$\beta_i$ and $\gamma_j$ of $J_{F+G,y}$
(cf. \cite{Strebel}).
From \eqref{eq:gemetric_definition_extremal_length}
and Theorem 20.5 in \cite{Strebel},
we have
\begin{align}
i(\tilde{\Phi}_{GM}(y),F+G)^2&=
\ext_y(F+G)=\|J_{F+G,y}\| \\
&=\sum_{i=1}^{N_1}\frac{t_i^2}{\Mod(A_{\alpha_i})}
+\sum_{i=1}^{N_2}\frac{s_i^2}{\Mod(A_{\beta_i})}+
\sum_{j=1}^{N_3}\frac{(u_j+v_j)^2}{\Mod(A_{\gamma_j})} 
\nonumber \\
&\ge 
\left(
\sum_{i=1}^{N_1}\frac{t_i^2}{\Mod(A_{\alpha_i})}
+
\sum_{j=1}^{N_3}\frac{u_j}{\Mod(A_{\gamma_j})}
\right)
\nonumber \\
&\quad
+
\left(
\sum_{i=1}^{N_2}\frac{s_i^2}{\Mod(A_{\beta_i})}+
\sum_{j=1}^{N_3}\frac{v_j^2}{\Mod(A_{\gamma_j})}\right)
\nonumber \\
&\ge
\|J_{F,y}\|+\|J_{G,y}\| 
=\ext_y(F)+\ext_y(G)
\nonumber
\label{eq:extremal_strebel}\\
&=i(\tilde{\Phi}_{GM}(y),F)^2+i(\tilde{\Phi}_{GM}(y),G)^2.
\nonumber
\end{align}
Since $\GmTeich$ is dense in $\GmInv$,
the above calculation implies
$$
i(\mathfrak{a},F)^2+
i(\mathfrak{a},G)^2
\le
i(\mathfrak{a},F+G)^2
$$
for all $\mathfrak{a}\in \GmInv$.
Hence,
the left-hand side of \eqref{eq:subadditivity}
also follows by approximating arrational components by
weighted multicurves
(cf. Theorem C of \cite{LM}).
\end{proof}

\section{Structure of the null sets}
\label{sec:Null_sets_section}
We define the \emph{null set} for $\mathfrak{a}\in\GmInv$ by
$$
\mathcal{N}(\mathfrak{a})=\{\mathfrak{b}\in \GmInv\mid
i(\mathfrak{a},\mathfrak{b})=0\}.
$$
This section is devoted to show the following theorem.
\begin{theorem}[Structure of the null set]
\label{thm:null_set}
For any $\mathfrak{a}\in \Gmbdy-\{0\}$,
any associated foliation $[G]\in \mathcal{PMF}$ for $\mathfrak{a}$
satisfies
$\mathcal{N}(\mathfrak{a})=\mathcal{N}(G)=\mathcal{N}(\trunc{G})$.
\end{theorem}
The associated foliation for $\mathfrak{a}$ in Theorem \ref{thm:null_set}
is defined in \S\ref{subset:associated_foliation}. 
We will see that
the associated foliations for $\mathfrak{a}$
are essentially uniquely determined from $\mathfrak{a}$
(cf. Theorem \ref{thm:uniqueness_associate_foliations}). 
The following is known
(cf. Proposition 9.1 in \cite{Mi5}).

\begin{proposition}
\label{prop:translation_null_space1}
For $\mathfrak{a}\in \GmInv-\{0\}$,
$\mathcal{N}(\mathfrak{a})\ne \{0\}$
if and only if $\mathfrak{a}\in \Gmbdy$.
In any case,
we have
$\mathcal{N}(\mathfrak{a})\subset \Gmbdy$,
and $\mathfrak{a}\in \mathcal{N}(\mathfrak{a})$
if $\mathfrak{a}\in \Gmbdy$.
\end{proposition}

\subsection{Associated foliations}
\label{subset:associated_foliation}
Let $[\mathfrak{a}]\in \partial_{GM}\teich$ and $\mathfrak{a}\in \Gmbdy-\{0\}$.
A projective measured foliation $[G]\in \mathcal{PMF}$
is said to be an \emph{associated foliation} for $[\mathfrak{a}]\in \partial_{GM}\teich$
if there exist $x\in \teich$,
 a sequence $[G_n]\in \mathcal{PMF}$ and
$t_n>0$ such that
$
R_{G_n,x}(t_n)\to [\mathfrak{a}]
$
and $[G_n]\to [G]$ as $n\to \infty$.
We call the point $x$ the \emph{base point} for the associated foliation $[G]$.
We denote by $\mathcal{AF}([\mathfrak{a}])$
the set of associated foliations for $[\mathfrak{a}]$.
\begin{figure}
\includegraphics[height=3.5cm]{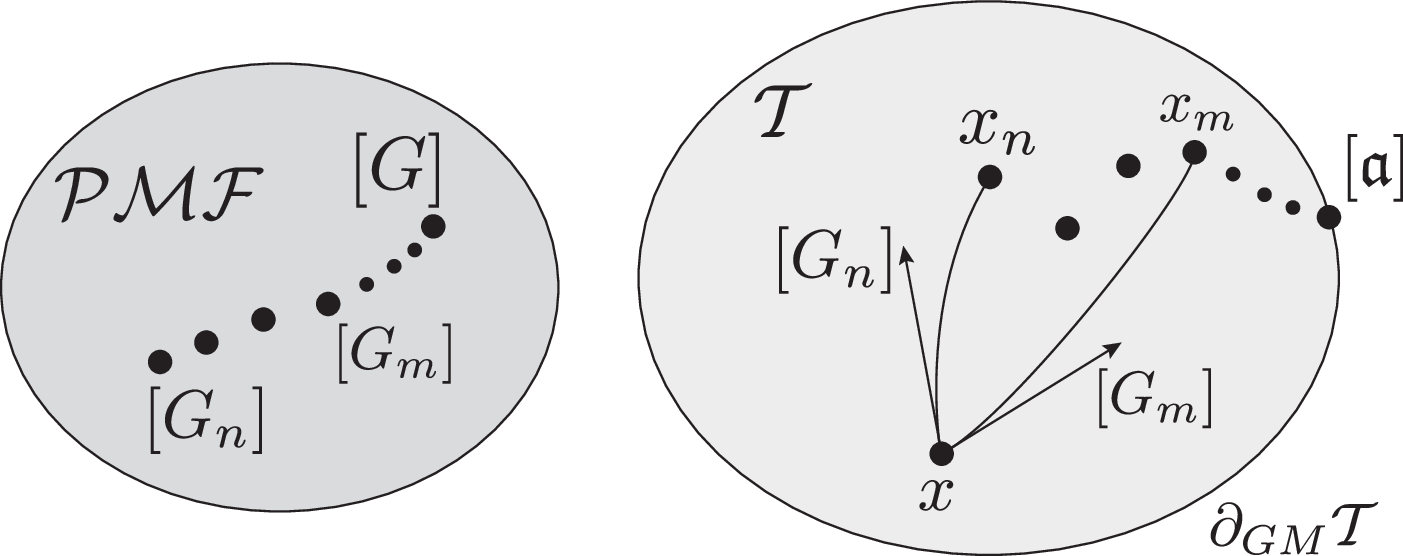}
\caption{Associated foliation $[G]$:
In the figure,
we set $x_{n}=
R_{G_n,x}(t_n)$.
$x$ is the base point for $[G]$.}
\end{figure}

In this section,
we prove the following.

\begin{proposition}[Uniqueness of vanishing curves]
\label{prop:charcterization_vanishing_curves}
Let $\mathfrak{a}\in \Gmbdy-\{0\}$.
For any $[G]\in \mathcal{AF}([\mathfrak{a}])$,
we have
$$
\mathcal{N}(G)\cap \mathcal{S}=
\mathcal{N}(\mathfrak{a})\cap \mathcal{S}.
$$
\end{proposition}

\subsubsection{Lemmas}
Let
\begin{equation}
\label{eq:null_set_measured_foliaiton}
\NMF(\mathfrak{a})=\mathcal{N}(\mathfrak{a})\cap \mathcal{MF}.
\end{equation}
When $\mathfrak{a}\in \mathcal{MF}$,
the set \eqref{eq:null_set_measured_foliaiton}
coincides with the set defined as \eqref{eq:nullset_pre}.
\begin{lemma}
\label{lem:self-intersection}
The following hold:
\begin{enumerate}
\item
\label{item:lem:self-intersection1}
$\{G\in \mathcal{MF}\mid [G]\in \mathcal{AF}([\mathfrak{a}])\}\subset
\NMF(\mathfrak{a})$.
\item
\label{item:lem:self-intersection2}
For $[G]\in \mathcal{AF}([\mathfrak{a}])$,
we have $\mathcal{N}(\mathfrak{a})\subset \mathcal{N}(G)$
and $\NMF(\mathfrak{a})\subset \NMF(G)$.
\end{enumerate}
In particular $i(G_{1},G_{2})=0$ for $[G_{1}],[G_{2}]\in \mathcal{AF}([\mathfrak{a}])$.
\end{lemma}

\begin{proof}
\eqref{item:lem:self-intersection1}\quad
Let $[G]\in \mathcal{AF}([\mathfrak{a}])$.
Take $x\in \teich$, $\{[G_n]\}_{n\in \mathbb{N}}\subset\mathcal{PMF}$,
and $t_n>0$ such that
$
R_{G_n,x}(t_n)\to [\mathfrak{a}]
$
and $G_n\to G$ as $n\to \infty$.
From \eqref{eq:image-psi-a},
$\Psi_{GM}\circ R_{G_n,x}(t_n)=e^{-t_{n}}\tilde{\Phi}_{GM}\circ R_{G_n,x}(t_n)$ converges to
$\mathfrak{a}'\in \GmInv-\{0\}$
which is projectively equivalent to $\mathfrak{a}$.
Therefore
\begin{align*}
i(\mathfrak{a}',G)
&=\lim_{n\to \infty}i(e^{-t_n}\tilde{\Phi}_{GM}\circ R_{G_n,x}(t_n),G_n) \\
&=\lim_{n\to \infty}e^{-t_n}\ext_{R_{G_n,x}(t_n)}(G_n)^{1/2} \\
&=\lim_{n\to \infty}e^{-2t_n}\ext_x(G_n)^{1/2}=0
\end{align*}
and $G\in \NMF(\mathfrak{a})$.

\medskip
\noindent
\eqref{item:lem:self-intersection2}\quad
Let $\mathfrak{b}\in \mathcal{N}(\mathfrak{a})$.
Take $x\in \teich$, $\{[G_n]\}_{n\in \mathbb{N}}\subset\mathcal{PMF}$,
$t_n>0$, and $\mathfrak{a}'$ as above.
From \eqref{eq:intesection-number-2}
and \eqref{eq:Minsky_inequality_in_general},
we have
\begin{align*}
i(G,\mathfrak{b})
&=\lim_{n\to \infty}i(G_n,\mathfrak{b}) \\
&\le\lim_{n\to \infty}\ext_{R_{G_n,x}(t_n)}(G_n)^{1/2}\ext_{R_{G_n,x}(t_n)}(\mathfrak{b})^{1/2}\\
&=\lim_{n\to \infty}e^{-t_n}\ext_x(G_n)^{1/2}\ext_{R_{G_n,x}(t_n)}(\mathfrak{b})^{1/2} \\
&=\lim_{n\to \infty}\ext_x(G_n)^{1/2}i(e^{-t_n}\tilde{\Phi}_{GM}\circ R_{G_n,x}(t_n),\mathfrak{b}) \\
&=\ext_x(G)^{1/2}i(\mathfrak{a}',\mathfrak{b})=0
\end{align*}
and $\mathfrak{b}\in\mathcal{N}(G)$.
From the definition,
$$
\NMF(\mathfrak{a})=\mathcal{N}(\mathfrak{a})\cap\mathcal{MF}
\subset \mathcal{N}(G)\cap\mathcal{MF}=\NMF(G).
$$
and we are done.
\end{proof}

For $\mathfrak{a}\in \Gmbdy-\{0\}$,
we define
\begin{align*}
\mathcal{AN}(\mathfrak{a})
&=
\cup_{[G]\in \mathcal{AF}([\mathfrak{a}])}\NMF(G)
\subset \mathcal{MF}.
%
\end{align*}
%
%
%
%

\begin{lemma}
\label{lem:vanish-intersection}
$\mathcal{AN}(\mathfrak{a})\cap \mathcal{S}\subset
\mathcal{N}(\mathfrak{a})\cap \mathcal{S}$
for all $\mathfrak{a}\in \GmInv-\{0\}$.
\end{lemma}

\begin{proof}
Let $\alpha\in \mathcal{AN}(\mathfrak{a})\cap \mathcal{S}$.
Let $[G]\in \mathcal{AF}([\mathfrak{a}])$ with $i(G,\alpha)=0$.
Then,
there are $x\in \teich$,
a sequence $\{[G_n]\}_{n\in \mathbb{N}}$ converging to $[G]$
and $t_n>0$
such that $
R_{G_n,x}(t_n)$
tends to $[\mathfrak{a}]$ as $n\to \infty$.
Let $y_t=(Y_t,f_t)=R_{G,x}(t)$.

We refer to the argument
in \S5.3 of \cite{Mi2}
(see also \cite{Ivanov2}
and \cite{Masur3}).
Let $\Gamma_G$ be the critical vertical graph of the holomorphic quadratic differential
of $J_{G,x}$.
We add mutually disjoint
critical vertical segments to $\Gamma_G$
emanating from critical points
to get a graph $\Gamma^0_G$ whose edges are all vertical.
The degree of a vertex $\Gamma^0_G$ is one-prong if it is one of endpoints
of an added vertical segment.
Take $\epsilon>0$ sufficiently small such that the $\epsilon$-neighborhood
$C(\epsilon)$ (with respect to the $|J_{G,x}|$-metric)
is embedded in $X$.
Then,
as the argument in the proof of Theorem 3.1 in \cite{Ivanov2},
by shrinking with a factor $e^{-t}$,
we get a canonical conformal embedding $g_t:C(\epsilon)\to Y_t$
such that $g_t(\Gamma_G)=f_t(\Gamma_G)$.
Since $i(\alpha,G)=0$,
$\alpha$ can be deformed into $C(\epsilon)$.
Hence,
by from the geometric definition
\eqref{eq:gemetric_definition_extremal_length} of extremal length,
the conformal embedding $g_t:C(\epsilon)\to Y_t$
induces
\begin{equation}
\label{eq:extremal-length-comparison}
\ext_{y_t}(\alpha)\le \ext_{C(\epsilon)}(\alpha)=: c_0
\end{equation}
for some $c_0>0$ independent of $t$.

Let $\epsilon>0$.
Take $T>0$ such that $2c_0e^{-T}<\epsilon$.
Since $[G_n]\to [G]$,
by \eqref{eq:exponential_maps},
there exists an $n_0>0$ such that
$d(R_{G,x}(T),R_{G_n,x}(T))\le (\log 2)/2$
and $t_n\ge T$
for $n\ge n_0$.
It has shown from Lemma 1 of \cite{Mi2}
that a function
$$
[0,\infty)\ni t\mapsto e^{-t}\ext_{y_t}(F)^{1/2}
$$
is a non-increasing function for any $F\in \mathcal{MF}$.
Hence,
from \eqref{eq:extremal-length-comparison},
we have
\begin{align*}
i(e^{-t_n}\tilde{\Phi}_{GM}\circ R_{G_n,x}(t_n),\alpha)
&=e^{-t_n}\ext_{R_{G_n,x}(t_n)}(\alpha)^{1/2} \\
&\le e^{-T}\ext_{R_{G_n,x}(T)}(\alpha)^{1/2} \\
&\le
2e^{-T}\ext_{y_T}(\alpha)^{1/2}\le 2c_0e^{-T}<\epsilon.
\end{align*}
Since 
$|t_n-d_T(x_0,R_{G_n,x}(t_n))|\le d_T(x,x_0)$,
by taking a subsequence,
$$
e^{-t_n}\tilde{\Phi}_{GM}\circ R_{G_n,x}(t_n)
=e^{t_n-d_T(x_0,R_{G_n,x}(t_n))}\cdot
\Psi_{GM}\circ R_{G_n,x}(t_n)
$$
converges to
$\mathfrak{a}'\in \GmInv-\{0\}$
with $[\mathfrak{a}']=[\mathfrak{a}]$.
Therefore,
we get
$$
i(\mathfrak{a}',\alpha)=
\lim_{n\to \infty}i(e^{-t_n}\tilde{\Phi}_{GM}\circ R_{G_n,x}(t_n),\alpha)
=0
$$
and 
$\alpha\in
\mathcal{N}(\mathfrak{a})\cap \mathcal{S}$.
\end{proof}

%
%
\subsubsection{Proof of Proposition \ref{prop:charcterization_vanishing_curves}}
Let $[G]\in \mathcal{AF}([\mathfrak{a}])$.
From \eqref{item:lem:self-intersection2} of Lemma \ref{lem:self-intersection}
and Lemma \ref{lem:vanish-intersection},
we have
\begin{align*}
\mathcal{N}(\mathfrak{a})\cap \mathcal{S}
&\subset\mathcal{N}(G)\cap \mathcal{S}
=\NMF(G)\cap \mathcal{S}
\\
&\subset 
(\cup_{[G]\in \mathcal{AF}([\mathfrak{a}])}\NMF(G))\cap \mathcal{S}
=\mathcal{AN}(\mathfrak{a})\cap \mathcal{S}\subset \mathcal{N}(\mathfrak{a})\cap \mathcal{S}.
\quad \square
\end{align*}

\subsection{Vanishing surface}
The aim of this section is to define the \emph{vanishing surface} for $\mathfrak{a}$,
which is used for proving Theorem \ref{thm:uniqueness_associate_foliations}
 stated in the
next section.

\subsubsection{Minimal vanishing surfaces}
Let $\mathfrak{a}\in \Gmbdy-\{0\}$.
Let $\vani{\mathfrak{a}}$ be the minimal essential subsurface of $X$
which contains all simple closed curve in
$\mathcal{N}(\mathfrak{a})\cap \mathcal{S}$.
We call $\vani{\mathfrak{a}}$ the \emph{minimal vanishing surface} for $\mathfrak{a}$.
By definition,
any component $Z_i$ of $\vani{\mathfrak{a}}$ contains a collection of curves in
$\mathcal{N}(\mathfrak{a})\cap \mathcal{S}$ which fills up $Z_i$.
It is possible that
either $\mathcal{N}(\mathfrak{a})\cap \mathcal{S}$
or $\vani{\mathfrak{a}}$ is empty.

\subsubsection{Properties of minimal vanishing surfaces}
From Lemma \ref{lem:fills_curves} in Appendix,
if $\alpha\in \mathcal{S}$ can be deformed into $\vani{\mathfrak{a}}$,
then $i(\mathfrak{a},\alpha)=0$
(see also Theorem 6.1 of \cite{GM}).
%

\begin{proposition}
\label{prop:betas}
Let $[G]\in \mathcal{AF}([\mathfrak{a}])$.
For $\alpha\in \mathcal{S}$,
the following are equivalent.
\begin{enumerate}
\item
\label{item:prop:betas1}
$\alpha$ is homotopic to a component of $\partial \vani{\mathfrak{a}}$.
\item
\label{item:prop:betas2}
$\alpha$ is homotopic to
either an essential curve
or a peripheral curve of $G$.
\end{enumerate}
\end{proposition}

\begin{proof}
{\rm \eqref{item:prop:betas1} $\Rightarrow$ \eqref{item:prop:betas2}.}\quad
Since
$i(\mathfrak{a},\alpha)=0$,
from Proposition \ref{prop:charcterization_vanishing_curves},
we have
$i(\alpha,G)=0$.
Suppose that $\alpha$ is non-peripheral in a component $W$
of $X-\supportMF{G}$.
Then,
there is an $\alpha'\in \mathcal{S}$
which is non-peripheral in $W$ satisfying $i(\alpha,\alpha')\ne 0$.
Since $i(\alpha',G)=0$,
$i(\alpha',\mathfrak{a})=0$
by Proposition \ref{prop:charcterization_vanishing_curves}.
This means that $\alpha$ cannot be homotopic to
a component of $\partial \vani{\mathfrak{a}}$
because $\vani{\mathfrak{a}}$ contains
a regular neighborhood of $\alpha\cup \alpha'$.
This contradicts our assumption.

\medskip
\noindent
{\rm \eqref{item:prop:betas2} $\Rightarrow$ \eqref{item:prop:betas1}.}\quad
Since $i(\alpha,G)=0$,
by Proposition \ref{prop:charcterization_vanishing_curves},
$\alpha$ can be deformed into the vanishing surface $\vani{\mathfrak{a}}$.
Suppose to the contrary that
$\alpha$ is non-peripheral in $\vani{\mathfrak{a}}$.
Then,
there is a non-peripheral curve $\delta$ in
a component of $\vani{\mathfrak{a}}$ with $i(\alpha,\delta)\ne 0$.
Since $\delta\in \mathcal{N}(\mathfrak{a})\cap \mathcal{S}$,
we have $i(\delta,G)=0$ by Proposition \ref{prop:charcterization_vanishing_curves} again.

If $\delta$ is a component of some $\partial X_i$,
$i(\alpha,G)\ge i(\alpha,G_i)\ne 0$ by Lemma 2.14 of \cite{Ivanov}.
This contradicts that $\alpha\subset \vani{\mathfrak{a}}$.
If
$\delta$ is non-peripheral in a component of $X-\supportMF{G}$,
so is $\alpha$
since $i(\alpha,\delta)\ne 0$.
This contradicts to the assumption.
\end{proof}

\begin{proposition}
\label{prop:vanishingsurface-components-pants}
For $\mathfrak{a}\in \Gmbdy-\{0\}$,
none of components of $\vani{\mathfrak{a}}$ are pairs of pants.
\end{proposition}

\begin{proof}
Let $Z$ be a component of $\vani{\mathfrak{a}}$.
Suppose to the contrary that $Z$ is a pair of pants.
Since any simple closed curve in $Z$ is homotopic to a component of $\partial Z$,
$(\vani{\mathfrak{a}}-Z)\cup N(\partial Z)$
contains all curves in $\mathcal{N}(\mathfrak{a})\cap \mathcal{S}$,
where $N(\partial Z)$ is the regular neighborhood of $\partial Z$.
This contradicts the minimality of $\vani{\mathfrak{a}}$.
%
%
%
\end{proof}

\subsubsection{Vanishing surface}
We define a subsurface $\augVani{\mathfrak{a}}$ of $X$ as follows:
\begin{enumerate}
\item
Remove annular components of $\vani{\mathfrak{a}}$ whose core is homotopic to a component of 
$\partial W$,
where $W$ runs components of $X-\vani{\mathfrak{a}}$ which are pairs of pants.
\item
 To the resulting surface,
 add components of $X-\vani{\mathfrak{a}}$
 which are pairs of pants.
\end{enumerate}
See Figure \ref{fig:vanishingsurface}.
\begin{figure}
\includegraphics[width=10cm]{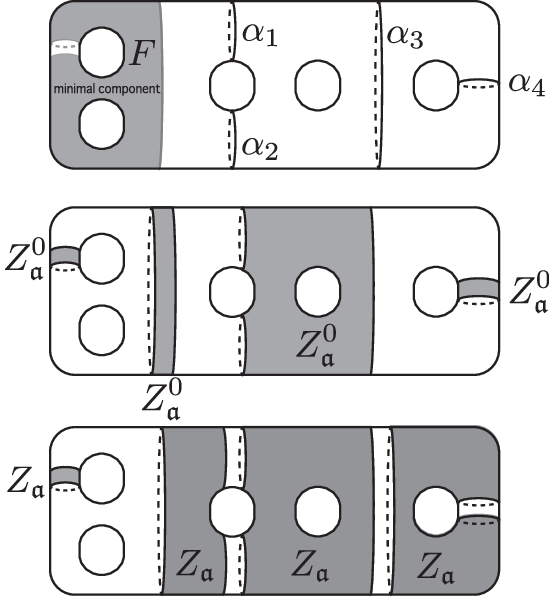}
\caption{Case of $G=F+\sum_{i=1}^{4}\alpha_{i}$.
In this case,
$\trunc{G}=G$.
$\vani{\mathfrak{a}}$ has three annular components.
Two have the core curves which are homotopic to a peripheral curve.
The other comes from an essential curve $\alpha_{4}$ of $G$.
The complement $X-\vani{\mathfrak{a}}$ has two components which are pairs of pants.}
\label{fig:vanishingsurface}
\end{figure}
We call $\augVani{\mathfrak{a}}$ the \emph{vanishing surface} for $\mathfrak{a}$.
Notice from definition
that $i(\partial Z,\mathfrak{a})=0$ for every component $Z$ of $\augVani{\mathfrak{a}}$,
and none of the components of $X-\augVani{\mathfrak{a}}$ are pairs of pants.
Recall that $\trunc{G}$ denotes the distinguished part of $G\in \mathcal{MF}$ on nullity
(cf. \S\ref{subsec:nullsets-measured-foliations}).

\subsection{Uniqueness of the underlying foliations}
\label{subsec:decomposition_theorem}
The following uniqueness theorem
implies that the underlying foliations of associated foliations for $\mathfrak{a}$
is essentially determined from $\mathfrak{a}$.


\begin{theorem}[Uniqueness of the underlying foliations]
\label{thm:uniqueness_associate_foliations}
For any $[G_1],[G_2]\in \mathcal{AF}([\mathfrak{a}])$,
$\trunc{G_1}$ and $\trunc{G_2}$ are topologically equivalent.
\end{theorem}

The above uniqueness theorem
follows from Proposition  \ref{prop:decomposition} below.

\begin{proposition}[Decomposition associated to $\mathfrak{a}$]
\label{prop:decomposition}
Let $\mathfrak{a}\in \Gmbdy-\{0\}$
and $\augVani{\mathfrak{a}}$ the vanishing surface for $\mathfrak{a}$.
Then,
the reference surface $X$ is decomposed into a union
of closed essential surfaces with mutually disjoint interiors as
\begin{equation} \label{eq:decomposition_Z_p}
\augVani{\mathfrak{a}}
\cup X_1\cup \cdots \cup X_{m_1}
\cup B_1\cup \cdots \cup B_{m_2}
\end{equation}
such that
for all $[G]\in \mathcal{AF}([\mathfrak{a}])$,
the following properties hold.
\begin{enumerate}
\item
\label{item:thm:decomposition2}
The family $\{X_i\}_{i=1}^{m_{1}}$ 
consists of all components of $X-\augVani{\mathfrak{a}}$ whose complexities are at least $1$.
The support of any minimal component of $G$ is some $X_{i}$.
For any for $i=1,\cdots,m_1$,
$X_{i}$ contains arrational foliation $F_i$ such that
the minimal component of $G$ whose support is $X_{i}$
is topologically equivalent to $F_i$.
Conversely,
for any $i$,
$G$ contains an arrational component whose support is isotopic to $X_{i}$.
\item
\label{item:thm:decomposition3}
The family $\{B_i\}_{i=1}^{m_{2}}$ consists of all annular components of $X-\augVani{\mathfrak{a}}$.
any essential curve of $G$ is homotopic to the core curve of some $B_{i}$.
Conversely,
the core of any $B_i$ is homotopic to some essential curve of $G$.
\item
\label{item:thm:decomposition4}
Any curve $\alpha\in \mathcal{S}$ deformed into $\augVani{\mathfrak{a}}$
satisfies
$i(\alpha,\mathfrak{a})=i(\alpha,G)=0$.
\end{enumerate}
\end{proposition}
%

\subsubsection*{Proof of Proposition \ref{prop:decomposition}}
Proposition \ref{prop:decomposition}
follows from the combination of the following four lemmas
given below.

\begin{lemma}[Non annular components of $\augVani{\mathfrak{a}}$]
\label{lem:vanishingsurface-components}
Let $\mathfrak{a}\in \Gmbdy-\{0\}$
and $[G]\in \mathcal{AF}([\mathfrak{a}])$.
Every non-annular component of $\augVani{\mathfrak{a}}$
is isotopic to a non-annular component of
$X-\supportMF{\trunc{G}}$,
and vice versa.
\end{lemma}

\begin{proof}
Let $Z$ be a non-annular component of $\augVani{\mathfrak{a}}$.
Suppose first that $Z$ is not a pair of pants.
Then,
$Z$ is also a component of $\vani{\mathfrak{a}}$
and
$Z$ contains a finite family of curves in $\mathcal{N}(\mathfrak{a})$
which fills up.
By Proposition \ref{prop:charcterization_vanishing_curves},
there is a component $W$ of the component of $X-\supportMF{G}$
such that $Z\subset W$ in homotopy sense.
Since $\complexity{W}\ge 1$,
$W$ is also a component of $X-\supportMF{\trunc{G}}$ in homotopy sense.
From Proposition \ref{prop:betas},
all component of $\partial Z$ is a peripheral curve in $W$.
Hence,
$\overline{Z}$ is isotopic to $\overline{W}$.

Suppose $Z$ is a pair of pants.
By definiton,
$i(\partial Z,\mathfrak{a})=0$ and $i(\partial Z,G)=0$.
Since $Z$ does not contain any minimal component of $G$,
$Z$ is contained in a component $W$ of $X-\supportMF{\trunc{G}}$.
By the same argument as above,
we obtain that $\overline{Z}$ is isotopic to $\overline{W}$.

The converse follows from the same argument.
However, let us give a sketch for the completeness.
Let $W$ be a non-annular component of $X-\supportMF{\trunc{G}}$.
If $\complexity{W}\ge 1$,
by Proposition \ref{prop:charcterization_vanishing_curves} again,
$W$ is contained in $\vani{\mathfrak{a}}$ in homotopy sense.
From Proposition \ref{prop:betas} again,
$W$ is isotopic to the component $Z$ of $\vani{\mathfrak{a}}$ containing $W$.
Since $\complexity{W}\ge 1$,
$Z$ is also a component of $\augVani{\mathfrak{a}}$ in homotopy sense.
Since $i(\partial W,G)=0$,
$i(\partial W,\mathfrak{a})=0$ and hence $\overline{Z}$ is isotopic to $\overline{W}$.
If $W$ is a pair of pants,
since $i(\partial W,\mathfrak{a})=0$ again,
we also conclude that $\overline{Z}$ is isotopic to $\overline{W}$.
\end{proof}

\begin{lemma}[Non annular components of $X-\augVani{\mathfrak{a}}$]
\label{lem:vanishingsurface-components-complement}
Let $\mathfrak{a}\in \Gmbdy-\{0\}$.
\begin{enumerate}
\item
\label{item:prop:vanishingsurface-components-complement1}
Let $[G]\in \mathcal{AF}([\mathfrak{a}])$.
Let $W$ be a component of $X-\augVani{\mathfrak{a}}$ with $\complexity{W}\ge 1$.
There is a minimal component $G_{i}$ of $G$ such that $\overline{W}=\supportMF{G_{i}}$
in homotopy sense.
Conversely,
the support of any arrational component of $G$ is isotopic to the closure of a component $W$ of
$X-\augVani{\mathfrak{a}}$ with $\complexity{W}\ge 1$.
\item
\label{item:prop:vanishingsurface-components-complement2}
For $[G_{1}],[G_{2}]\in \mathcal{AF}([\mathfrak{a}])$,
any arrational component of $G_{1}$ is topogically equivalent to that of $G_{2}$.
\end{enumerate}
\end{lemma}

\begin{proof}
\eqref{item:prop:vanishingsurface-components-complement1}\quad
Let $W$ be a component of $X-\augVani{\mathfrak{a}}$ with $\complexity{W}\ge 1$.
By definition,
$W$ is also a component of $X-\vani{\mathfrak{a}}$.
From Proposition \ref{prop:charcterization_vanishing_curves},
we have $i(\alpha,G)\ne 0$
for every curve $\alpha$ which is non-peripheral in $W$.
From Proposition \ref{prop:betas}
essential curves and peripheral curves of $G$
are deformed into $\vani{\mathfrak{a}}$.
Hence $\alpha$ intersects some minimal component $G_i$ of $G$.

We check that $\overline{W}=\supportMF{G_{i}}$ in homotopy sense.
We first check $\supportMF{G_{i}}\subset W$.
Otherwise,
there is a component $\gamma$ of $\partial W\subset \partial \vani{\mathfrak{a}}$
which intersects non-trivially to $\supportMF{G_{i}}$.
This means that $i(\gamma,G)\ge i(\gamma,G_i)\ne 0$
and hence $i(\gamma,\mathfrak{a})\ne 0$ from
Proposition \ref{prop:charcterization_vanishing_curves},
which is a contradiction.
If a component $\gamma$ of $\partial \supportMF{G_{i}}$ is non-peripheral in $W$,
$\gamma$ cannot be deformed into $\vani{\mathfrak{a}}$
and hence $i(\gamma,\mathfrak{a})\ne 0$.
Therefore, $i(\gamma,G)\ne 0$,
as we checked in the previous paragraph.
Thus,
we conclude that $\supportMF{G_{i}}\hookrightarrow \overline{W}$ is a deformation retract.

Let $G_{i}$ be a minimal component of $G$.
Since any simple closed curve which is non-peripheral in $\supportMF{G_{i}}$
satisfies $i(\alpha,G)=i(\alpha,G_{i})\ne 0$,
we have $i(\alpha,\mathfrak{a})\ne 0$.
Therefore,
$\supportMF{G_{i}}$ is disjoint from $\vani{\mathfrak{a}}$
(in homotopy sense).
Let $W$ be a component of $Z-\vani{\mathfrak{a}}$ with
$\supportMF{G_{i}}\subset W$ in homotopy sense.
Since $i(\partial \supportMF{G_{i}},G)=0$,
from Proposition \ref{prop:charcterization_vanishing_curves},
we can deduce that $\supportMF{G_{i}}$ is isotopic to $\overline{W}$.

\medskip
\noindent
\eqref{item:prop:vanishingsurface-components-complement2}\quad
Let $H_{1}$ be a minimal component of $G_{1}$.
From \eqref{item:prop:vanishingsurface-components-complement1} above,
there is a minimal component $H_{2}$ of $G_{2}$
such that $\supportMF{H_{2}}=\supportMF{H_{1}}$.
Since $i(H_{1},H_{2})\le i(G_{1},G_{2})=0$ from Lemma \ref{lem:self-intersection}.
Hence $H_{1}$ is topologically equivalent to $H_{2}$
(e.g. Theorem 1.1 in \cite{Rees}).
\end{proof}
%
%

\begin{lemma}[Annular components of $X-\augVani{\mathfrak{a}}$]
\label{lem:annular-component}
Let $\mathfrak{a}\in \Gmbdy-\{0\}$ and $[G]\in \mathcal{AF}([\mathfrak{a}])$.
The core curve of any annular component of $X-\augVani{\mathfrak{a}}$
is homotopic to an essential curve of $G$,
and vice versa.
\end{lemma}

\begin{proof}
Let $W$ be an annular component of $X-\augVani{\mathfrak{a}}$.
Let $Z_{1}$ and $Z_{2}$ be components of $\augVani{\mathfrak{a}}$
adjacent to $W$.
Possibly $Z_{1}=Z_{2}$.
Suppose some $Z_{i}$ is an annulus.
Then,
$Z_{i}$ is also a component of $Z_{\mathfrak{a}}$.
Since $W$ is also an annulus,
$Z_{i}$ is absorbed into the regular neighborhood of $\partial Z_{j}$ where $\{i,j\}=\{1,2\}$.
This contradicts to the minimality of $\vani{\mathfrak{a}}$,
because
each component of $\partial Z_{j}$ is in $\mathcal{N}(\mathfrak{a})\cap \mathcal{S}$
and the regular neighborhood of $\partial Z_{j}$ is contained in $\vani{\mathfrak{a}}$.
Hence,
the core curve $\delta$ of $W$ is not a peripheral curve of $G$
from Lemma \ref{lem:vanishingsurface-components}.

Since the core $\delta$ of $W$ is non-peripheral in $Z_{1}\cup W\cup Z_{2}$,
we can take a curve $\beta\in \mathcal{S}$ such that $\beta\subset Z_{1}\cup W\cup Z_{2}$
and $i(\delta,\beta)\ne 0$.
If $\delta$ is not an essential curve of $G$,
$i(\beta,G)=0$
and hence $i(\beta,\mathfrak{a})=0$.
Therefore,
$Z_{1}\cup W\cup Z_{2}$ is a non-annular component of $X-\supportMF{\trunc{G}}$,
since each component of $X-\supportMF{\trunc{G}}$ is incompressible.
This is a contradiction because $W$ can be deformed into the outside of
$\augVani{\mathfrak{a}}$ (cf. Lemma \ref{lem:vanishingsurface-components}).

Conversely,
let $\delta$ be a essential curve of $G$.
Let $W_{1}$ and $W_{2}$ be components of $X-\supportMF{\trunc{G}}$
which are adjacent to the annular component $N_{\delta}$ of $\supportMF{\trunc{G}}$ whose core is
$\delta$.
Since neither $W_{1}$ nor $W_{2}$ is not annulus,
from Lemma \ref{lem:vanishingsurface-components},
each $W_{i}$ is a component of $\augVani{\mathfrak{a}}$.
Therefore,
$N_{\delta}$ is a component of $X-\augVani{\mathfrak{a}}$.
\end{proof}

\begin{lemma}[Annular component of $\augVani{\mathfrak{a}}$]
\label{lem:annular-component-of-hat-Z}
Let $\mathfrak{a}\in \Gmbdy-\{0\}$ and $[G]\in \mathcal{AF}([\mathfrak{a}])$.
The core curve of any annular component of $\augVani{\mathfrak{a}}$ is
homotopic to a component of the boundary of the support of a minimal component of $G$.
\end{lemma}

\begin{proof}
Let $Z$ be an annular component of $\augVani{\mathfrak{a}}$.
Then,
$Z$ is also a component of $\vani{\mathfrak{a}}$.
Hence the core curve $\delta$ of $Z$ is not peripheral in $X$.
Let $\partial Z=\gamma_{1}\cup \gamma_{2}$
Let $W_{1}$ and $W_{2}$ be the closures of components of $X-\augVani{\mathfrak{a}}$
such that $\gamma_{i}\subset \partial W_{i}$ ($i=1,2$).
Possibly $W_{1}=W_{2}$.
Since each $W_{i}$ is not a pair of pants,
if some $W_{i}$ is an annulus,
$Z$ is absorbed in the component of $\augVani{\mathfrak{a}}$
which is on the opposite side of $W_{i}$ to $Z$.
This contradicts to the minimality of $\vani{\mathfrak{a}}$.
Hence,
each $W_{i}$ satisfies $\complexity{W_{i}}\ge 1$,
from Lemma \ref{lem:vanishingsurface-components-complement},
we conclude that $\delta$ is homotopic to a component of the boundary
of some minimal component of $G$.
\end{proof}

\subsection{Intersection number lemma}
The following intersection number lemma encodes the intersection number
for two points in $\partial_{GM}\teich$ to that of those associated foliations
up to multiple by positive constant.

\begin{lemma}[Intersection number lemma]
\label{lem:intersection_number}
Let $\mathfrak{a},\mathfrak{b}\in \Gmbdy-\{0\}$
and $[G]\in \mathcal{AF}([\mathfrak{a}])$ and $[H]\in \mathcal{AF}([\mathfrak{b}])$.
Then,
there is an $[F_\infty]\in \overline{\mathcal{AF}([\mathfrak{a}])}$ in $\mathcal{PMF}$ such that
\begin{equation}
\label{eq:intersection_number-lemma}
D_0\,i_{x_{0}}([G],[H])
\le 
i_{x_{0}}([\mathfrak{a}],[\mathfrak{b}])
\le
i_{x_{0}}([F_{\infty}],[\mathfrak{b}])
\end{equation}
where $D_0=e^{-d_T(x_0,x)-d_T(x_0,y)}$
and $x$ and $y$ are base points for the associated foliations $[G]$ and $[H]$ respectively.
\end{lemma}

\begin{proof}
By definition,
there are
$\{[G_n]\}_{n\in \mathbb{N}}$,
$\{[H_n]\}_{n\in \mathbb{N}}\subset \mathcal{PMF}$
and $t_n,s_n>0$ such that 
\begin{itemize}
\item
$
R_{G_n,x}(t_n)\to [\mathfrak{a}]$
and
$
R_{H_n,y}(s_n)\to [\mathfrak{b}]$
as $n\to \infty$,
and
\item
$G_n\to G$ and $H_n\to H$
as $n\to \infty$.
\end{itemize}
For simplicity,
let $x_n=R_{G_n,x}(t_n)$ and $y_n=R_{H_n,y}(s_n)$.

Since $d_T(x_0,x_n)\le t_n+d_T(x_{0},x)$
and $d_T(x_0,y_n)\le s_n+d_T(x_{0},y)$,
from Proposition \ref{prop:IntersectionNumberBasePoint},
we deduce
\begin{align*}
i_{x_{0}}(x_{n},y_{n})
&=\exp(-2\gromov{x_n}{y_n}{x_0})\\
&=\exp(d_T(x_n,y_n)-d_T(x_0,x_n)-d_T(x_0,y_n)) \\
&\ge D_0\exp(d_T(x_n,y_n))e^{-t_n}e^{-s_n} \\
&= D_0\exp(d_T(x_n,y_n))
\frac{\ext_{x_n}(G_n)^{1/2}}{\ext_{x_0}(G_n)^{1/2}}
\frac{\ext_{y_n}(H_n)^{1/2}}{\ext_{x_0}(H_n)^{1/2}} \\
&= D_0
\frac{\ext_{x_n}(G_n)^{1/2}}{\ext_{x_0}(G_n)^{1/2}}
\frac{\exp(d_T(x_n,y_n))\ext_{y_n}(H_n)^{1/2}}{\ext_{x_0}(H_n)^{1/2}} \\
&\ge D_0
\frac{\ext_{x_n}(G_n)^{1/2}}{\ext_{x_0}(G_n)^{1/2}}
\frac{\ext_{x_n}(H_n)^{1/2}}{\ext_{x_0}(H_n)^{1/2}} \\
&\ge D_0
\frac{i(H_n,G_n)}{\ext_{x_0}(G_n)^{1/2}\ext_{x_0}(H_n)^{1/2}}
=D_{0}i_{x_{0}}([G_{n}],[H_{n}]).
\end{align*}
By letting $n\to \infty$,
we obtain the left-hand side of \eqref{eq:intersection_number-lemma}.


Fix $n\in \mathbb{N}$.
Let $F_{m,n}\in \mathcal{MF}_1$
with $x_m=R_{F_{m,n},y_n}(u_{m,n})$,
where
$u_{m,n}=d_{T}(x_{n},y_{m})$.
Notice that
\begin{equation}
\label{eq:comparison_extremal_length_x_m_y_n}
\ext_{x_m}(F_{m,n})=e^{-2u_{m,n}}\ext_{y_n}(F_{m,n}).
\end{equation}
By taking a subsequence
(or by the diagonal argument),
we may assume that $F_{m,n}\to F_{\infty,n}\in \mathcal{MF}_1$ as $m\to \infty$
for each $n$,
and
$F_{\infty,n}$ converges to
$F_{\infty}\in \mathcal{MF}_1$.
Since $
x_m\to [\mathfrak{a}]$,
$[F_{\infty,n}]$ is an associated foliation for $\mathfrak{a}$
with base point $y_n$.
Therefore,
the limit $[F_\infty]$ is contained in the closure of $\mathcal{AF}([\mathfrak{a}])$ in $\mathcal{PMF}$.

Since $F_{m,n}\in \mathcal{MF}_1$,
from Theorem \ref{thm:main_realization},
\eqref{eq:intesection-number-2},
\eqref{eq:image-psi-a}
and \eqref{eq:comparison_extremal_length_x_m_y_n},
we deduce
\begin{align*}
i_{x_{0}}(y_{n},x_{m})
&=\exp(-2\gromov{y_n}{x_m}{x_0}) \\
&=\exp(u_{m,n}-d_T(x_0,x_m)-d_T(x_0,y_n)) \\
&=\exp(-d_T(x_0,y_n))\frac{\ext_{y_n}(F_{m,n})^{1/2}}
{\exp(d_T(x_0,x_m))\ext_{x_m}(F_{m,n})^{1/2}} \\
&\le \exp(-d_T(x_0,y_n))\ext_{y_n}(F_{m,n})^{1/2} \\
&=i(\Psi_{GM}(y_n),F_{m,n})
=i(\Psi_{GM}(y_n),\Psi_{GM}(F_{m,n})) \\
&=i_{x_{0}}(y_n,[F_{m,n}]).
\end{align*}
Letting $m\to \infty$,
we conclude
\begin{equation}
\label{eq:intersection_1}
i_{x_{0}}(y_{n},[\mathfrak{a}])
\le
i_{x_{0}}(y_n,[F_{\infty,n}]).
\end{equation}
Thus,
if $n\to \infty$ in \eqref{eq:intersection_1},
we obtain what we wanted.
\end{proof}

\subsection{Proof of Structure theorem}
\label{subsec:Proof_of_Theorem_null_sets}
We first check the following.

\begin{lemma}
\label{lem:zeros_topologically_equivalent}
Let $[G]\in \mathcal{AF}([\mathfrak{a}])$.
Let $F\in \mathcal{MF}$ be a measured foliation which is topologically equivalent
to a minimal component of $G$.
Then,
$i(F,\mathfrak{a})=0$.
\end{lemma}

\begin{proof}
Take $x\in \teich$, $[G_n]\in \mathcal{PMF}$,
and $t_n>0$ such that
$
R_{G_n,x}(t_n)\to [\mathfrak{a}]$
and $G_n\to G$ as $n\to \infty$.
Let $y_n=R_{G_n,x}(t_n)$.
Let $L_{F,y_n}$ be the geodesic current associated to
the singular flat structure
defined as $Q_{n}:=J_{F,y_n}/\|J_{F,y_n}\|$
given by Duchin,
Leininger and 
Rafi in \cite{DLR}.

Suppose on the contrary that $i(\mathfrak{a},F)\ne 0$.
Then,
by Proposition 4 in \cite{Mi2},
$\{Q_n\}_{n\in \mathbb{N}}$ is a stable sequence in the sense that
the set of accumulation points of $\{e^{-t_n}L_{F,y_n}\}_{n\in \mathbb{N}}$
in the space of geodesic currents is contained in $\mathcal{MF}-\{0\}$
(as geodesic currents).
In addition,
any accumulation point $L_\infty\in \mathcal{MF}-\{0\}$ satisfies
\begin{align}
i(L_\infty,F)
&=t_0i(\mathfrak{a},F)
\ne 0
\label{eq:limits_DLR1}
\\
i(L_\infty,H)
&\le t_0i(\mathfrak{a},H)
\label{eq:limits_DLR2}
\end{align}
for some $t_0>0$ and any $H\in \mathcal{MF}$
(see Proposition 5 in \cite{Mi2}).

Let $G_0$ be a minimal component of $G$
which is topologically equivalent to $F$
and $X_0$ be the support of $G_0$.
From \eqref{eq:limits_DLR2}, $i(L_\infty,G)=0$.
Hence,
if $L_\infty$ has a component $L_0$ whose support intersects $X_0$,
then $L_0$ is topologically equivalent to $G_0$
(cf. \cite{Ivanov}).
This means that $i(L_\infty,F)=0$,
which contradicts to \eqref{eq:limits_DLR1}.
\end{proof}

\begin{proof}[Proof of Theorem \ref{thm:null_set}]
We are ready to prove Theorem \ref{thm:null_set}.
%
%
Let $[G]\in \mathcal{AF}([\mathfrak{a}])$.
From \eqref{item:lem:self-intersection2} of Lemma \ref{lem:self-intersection},
we need to show the converse
$\mathcal{N}(\mathfrak{a})\supset \mathcal{N}(G)$.

We first claim that $\NMF(\mathfrak{a})=\NMF(G)$
for $[G]\in \mathcal{AF}([\mathfrak{a}])$.
We decompose $G$ as \eqref{eq:decompositionMF}:
$$
G=G'_1+G'_2+\cdots+G'_{m_1}+\beta_1+\cdots \beta_{m_2}
+\gamma_1+\cdots+\gamma_{m_3}.
$$

Let $H\in \NMF(G)$.
Then,
$H$ can be decomposed as
\begin{equation}
\label{eq:intersection_H}
H=\sum_{i=1}^{m_1}H_i+
\sum_{i=1}^{m_2}a_i\beta_i
+\sum_{i=1}^{m_1}\sum_{\gamma\subset \partial X_i} b_\gamma\gamma
+F_0
\end{equation}
where $a_i,b_\gamma\ge 0$,
$H_i$ is a measured foliation topologically equivalent to $G'_i$
 (possibly $H_i=0$),
and $F_0$ is a measured foliation whose support is
contained in the complement of $\supportMF{G}$
(cf. \cite{Ivanov}).
From Proposition \ref{prop:decomposition},
the support of $F_0$ is contained in
the vanishing surface
$\augVani{\mathfrak{a}}$.
Therefore,
$i(F_0,\mathfrak{a})=0$.
Since any component of $\partial X_i$ is deformed into $\vani{\mathfrak{a}}$,
from Lemma \ref{lem:subadditivity} and Lemma \ref{lem:zeros_topologically_equivalent},
we have
\begin{equation}
\label{eq:intersection_H_1}
i(H,\mathfrak{a})
\le
\sum_{i=1}^{m_1}i(H_i,\mathfrak{a})
+
\sum_{i=1}^{m_2}a_ii(\beta_i,\mathfrak{a})
+\sum_{i=1}^{m_1}\sum_{\gamma\subset \partial X_i}b_\gamma
i(\gamma,\mathfrak{a})
+i(F_0,\mathfrak{a})=0
\end{equation}
and hence $\NMF(G)\subset \NMF(\mathfrak{a})$.

Let $\mathfrak{b}\in \mathcal{N}(G)$
and take $\{y_{n}\}_{n=1}^{\infty}$
such that $y_n\to [\mathfrak{b}]$
as $n\to \infty$.
Let $H_n\in \mathcal{MF}_1$,
$s_n>0$ such that $y_n=R_{H_n,x_0}(s_n)$.
By taking a subsequence,
we may assume that $H_n\to H_\infty$.
Then,
$[H_\infty]\in \mathcal{AF}([\mathfrak{b}])$.
Let $[F_\infty]\in \overline{\mathcal{AF}([\mathfrak{a}])}$ as
Lemma \ref{lem:intersection_number} for $[G]$, $[H_{\infty}]$,
$\mathfrak{a}$ and $\mathfrak{b}$.
To show that $\mathfrak{b}\in \mathcal{N}(\mathfrak{a})$,
it suffices to show that $i(\mathfrak{b},F_{\infty})=0$ from Lemma \ref{lem:intersection_number}.

Since $\mathfrak{b}\in \mathcal{N}(G)$
and $\NMF(H_\infty)=\NMF(\mathfrak{b})$,
we have $i(G,H_\infty)=0$.
Therefore,
$H_\infty$ is decomposed as
\begin{equation}
\label{eq:decompositionH_infty}
H_{\infty}=\sum_{i=1}^{m_1}H'_i+\sum_{i=1}^{m_2}a_i\beta_i
+\sum_{i=1}^{m_1}\sum_{\gamma\subset \partial X_i}b_\gamma\gamma
+H_0
\end{equation}
where $H'_i$ is topologically equivalent to $G'_i$,
$a_i,b_\gamma\ge 0$,
$H_0$ is a measured foliation whose support is contained
in the complement of $\supportMF{G}$.
Since $[F_\infty]\in \overline{\mathcal{AF}([\mathfrak{a}])}$,
from Theorem \ref{thm:uniqueness_associate_foliations},
$F_{\infty}$ is decomposed as
\begin{equation}
\label{eq:decompositionF_infty}
F_{\infty}=\sum_{i=1}^{m_1}F'_i+\sum_{i=1}^{m_2}a_i\beta_i
+\sum_{i=1}^{m_1}\sum_{\gamma\subset \partial X_i}b_\gamma\gamma,
\end{equation}
where $F'_\infty$ is topologically equivalent to $G'_i$
(possibly $F'_i=0$)
and $a_i,b_\gamma\ge 0$.
From \eqref{eq:decompositionH_infty} and \eqref{eq:decompositionF_infty},
we have $i(F_\infty,H_\infty)=0$.
Since $\NMF(H_\infty)=\NMF(\mathfrak{b})$ again,
we conclude that $i(\mathfrak{b},F_\infty)=0$
as desired.
\end{proof}

\subsection{Topological equivalence revisited}
Before closing this section,
we notice the following expected property.

\begin{corollary}[Topological equivalence and Null sets]
\label{coro:topological_equivalence_null_set}
For $G,H\in \mathcal{MF}$,
the following are equivalent:
\begin{enumerate}
\item
\label{item:coro:topological_equivalence_null_set1}
$\trunc{G}$ and $\trunc{H}$ are topologically equivalent;
\item
\label{item:coro:topological_equivalence_null_set2}
$\NMF(G)=\NMF(H)$;
\item
\label{item:coro:topological_equivalence_null_set3}
$\mathcal{N}(G)=\mathcal{N}(H)$.
\end{enumerate}
In particular,
$\mathcal{N}(G)=\mathcal{N}(\trunc{G})$ for any $G\in \mathcal{MF}$.
\end{corollary}

\begin{proof}
From Proposition \ref{prop:measured_foliation_removing_peripehral},
the conditions \eqref{item:coro:topological_equivalence_null_set1}
and
\eqref{item:coro:topological_equivalence_null_set2} are equivalent.
Since
$\NMF(G)=\mathcal{N}(G)\cap \mathcal{MF}$,
\eqref{item:coro:topological_equivalence_null_set2} follows from \eqref{item:coro:topological_equivalence_null_set3}.
Hence,
we need to show that \eqref{item:coro:topological_equivalence_null_set1}
implies \eqref{item:coro:topological_equivalence_null_set3}.
From the symmetry of the topological equivalence,
it suffices to show that $\mathcal{N}(G)\subset\mathcal{N}(H)$.

Let $\mathfrak{a}\in \mathcal{N}(G)$
and $[F]\in \mathcal{AF}([\mathfrak{a}])$.
Then,
$i(G,F)=0$ from Theorem \ref{thm:null_set}.
Since $\trunc{H}$ is topologically equivalent to $\trunc{G}$,
by Proposition \ref{prop:measured_foliation_removing_peripehral},
we have $i(H,F)=0$.
Hence,
by applying Theorem \ref{thm:null_set}
again,
we have $i(H,\mathfrak{a})=0$ and $\mathfrak{a}\in \mathcal{N}(H)$.
\end{proof}

\section{Action on the Reduced boundary}
\label{sec:action_Reduced_boundary}
Let $S$ and $S'$ be compact orientable surfaces of non-sporadic type.
In this section,
we study maps in $\ACINV(\teich(S),\teich(S'))$.

\subsection{Null sets and accumulation sets}
For $p\in\cl{\teich(S)}$,
we define the \emph{null set} for $p$ by
$$
\nullsets{S}{p}
=\{q\in \cl{\teich(S)}\mid i_{x_0}(p,q)=0\}.
$$
For ${\bf x}\in \sq{\teich(S)}$,
we define
\begin{equation}
\label{eq:set_M}
\accum_S({\bf x})=\cup\{\overline{{\bf z}}\cap \partial_{GM}\teich(S)\mid {\bf z}\in
 \visualindist({\bf x})\},
\end{equation}
where $\overline{{\bf z}}$ is the closure of ${\bf z}$ in $\cl{\teich(S)}$.
The following proposition 
follows from \eqref{eq:Gromov_product_boundary_continuous}.

\begin{proposition}
\label{prop:accumulation_point_accompany}
Let $p,p^1,p^2\in \partial_{GM}\teich(S)$
and ${\bf x}$, ${\bf x}^1$, ${\bf x}^2\in \squ{\teich(S)}$.
\begin{enumerate}
\item
\label{item:prop:accumulation_point_accompany1}
If ${\bf x}$ converges to $p$,
$\nullsets{S}{p}=\accum_{S}({\bf x})$.
\item
\label{item:prop:accumulation_point_accompany2}
Suppose each ${\bf x}^{i}$ converges to $p^i$ for $i=1,2$.
Then,
$\nullsets{S}{p^2}\subset \nullsets{S}{p^1}$
if and only if $\visualindist({\bf x}^2)\subset \visualindist({\bf x}^1)$.
\end{enumerate}
\end{proposition}

\begin{proposition}[Structure of accumulation points]
\label{prop:null_sets_and_accompany}
Let ${\bf x}\in \sq{\teich(S)}$.
Then,
there is $G\in \mathcal{MF}$ such that
$$
\accum_{S}({\bf x})=\nullsets{S}{[G]}.
$$
Furthermore,
the following are equivalent for $q\in \partial_{GM}\teich(S)$:
\begin{enumerate}
\item
$q\in \nullsets{S}{[G]}$;
\item
for any $p\in \overline{{\bf x}}\cap \partial_{GM}\teich(S)$
and $[G_p]\in \mathcal{AF}(p)$,
$q\in \nullsets{S}{[G_p]}$;
\item
for any $p\in \overline{{\bf x}}\cap \partial_{GM}\teich(S)$
$q\in \nullsets{S}{p}$.
\end{enumerate}
\end{proposition}

\begin{proof}
For $p\in \overline{{\bf x}}\cap \partial_{GM}\teich(S)$,
fix $[G_p]\in \mathcal{AF}(p)$.
From
\eqref{eq:Gromov_product_boundary_continuous} and
Theorem \ref{thm:null_set},
$i(G_{p^1},G_{p^2})=0$ for $p^1,p^2\in \overline{{\bf x}}\cap \partial_{GM}\teich(S)$.
Hence,
we can find $G\in \mathcal{MF}$
such that 
\begin{enumerate}
\item[(1)]
for any $p\in \overline{{\bf x}}\cap \partial_{GM}\teich(S)$,
$\trunc{G_p}$ is topologically equivalent
to a subfoliation of $G$,
and
\item[(2)]
any component of $\trunc{G}$
is topologically equivalent
to a component of some $G_p$,
$p\in \overline{{\bf x}}\cap \partial_{GM}\teich(S)$.
\end{enumerate}
We check that $G$ satisfies the desired property.
Let $q\in \accum_{S}({\bf x})$
be an accumulation point of
${\bf z}\in \visualindist({\bf x})$.
Let $[H]\in \mathcal{AH}(q)$.
Since $i(H,G_p)=0$
for all $p\in {\bf x}\cap \partial_{GM}\teich(S)$,
from the condition (2) of $G$,
we have $i(G,H)=0$
and hence $\Psi_{GM}(q)\in \mathcal{N}(G)$
by Theorem \ref{thm:null_set}.
This means that
$q\in \nullsets{S}{[G]}$ and
$
\accum_{S}({\bf x})
\subset\nullsets{S}{[G]}$.

Conversely,
let $q\in \nullsets{S}{[G]}$.
Take a sequence ${\bf z}$ in $X$ converging to $q$.
By the condition (1) of $G$ above,
$\nullsets{S}{[G]}\subset \nullsets{S}{[G_p]}$
for all $p\in \overline{{\bf x}}\cap \partial_{GM}\teich(S)$.
In other words,
any subsequence of ${\bf x}$
contains a subsequence which is visually indistinguishable from ${\bf z}$.
Therefore,
we have
${\bf x}\in \visualindist({\bf z})$
and hence ${\bf z}\in \visualindist({\bf x})$.

The last statement follows from the construction of $G$
and Theorem \ref{thm:null_set}.
\end{proof}

\begin{proposition}
\label{prop:accumulation_points_and_accompany}
Let ${\bf x}^1,{\bf x}^2\in \sq{\teich(S)}$.
The following are equivalent:
\begin{enumerate}
\item[{\rm (1)}]
$\accum_{S}({\bf x}^1)\subset\accum_{S}({\bf x}^2)$;
\item[{\rm (2)}]
$\visualindist({\bf x}^1)\subset \visualindist({\bf x}^2)$.
\end{enumerate}
\end{proposition}

\begin{proof}
From the definition \eqref{eq:set_M},
the condition (2) implies (1).

Suppose the condition (1).
Assume to the contrary that
there is ${\bf z}\in \visualindist({\bf x}^1)\setminus
\visualindist({\bf x}^2)$.
Take subsequences ${\bf z}'
=\{z'_n\}_{n\in \mathbb{N}}
$ of ${\bf z}$
and ${{\bf x}'}^2=\{x'_n\}_{n\in \mathbb{N}}$ of ${\bf x}^2$
such that 
$$
\gromov{x'_n}{z'_n}{x_0}<M_1
$$
for all $n\in \mathbb{N}$.
Then,
any
$q'\in \overline{{\bf z}'}\cap \partial_{GM}\teich(S)$
($\subset \overline{{\bf z}}\cap \partial_{GM}\teich(S)$)
and $p'\in \overline{{{\bf x}'}^2}\cap \partial_{GM}\teich(S)$
($\subset \overline{{\bf x}^2}\cap \partial_{GM}\teich(S)$)
satisfy $i_{x_0}(p',q')\ne 0$
(cf. \eqref{eq:Gromov_product_boundary_continuous}).
By Proposition \ref{prop:null_sets_and_accompany},
$q'\not\in \accum_{S}({\bf x}^2)$.
Since $\accum_{S}({\bf x}^1)\subset \accum_{S}({\bf x}^2)$
from the assumption,
$q'\not\in \accum_{S}({\bf x}^1)$.
On the other hand,
since ${\bf z}\in \visualindist({\bf x}^1)$,
$q'\in \overline{{\bf z}}\cap \partial_{GM}\teich(S)
\subset \accum_{S}({\bf x}^1)$.
This is a contradiction.
%
%
%
\end{proof}

\subsection{Accumulation sets}
Let $\omega\in \AC(\teich(S),\teich(S'))$,
For $p\in \cl{\teich(S)}$,
we define the \emph{accumulation set}
by
$$
\mathcal{A}(\omega\colon p)=\{q\in \cl{\teich(S')}\mid
\mbox{$\exists\{y_n\}_{n\in \mathbb{N}}\in \sq{\teich(S)}$ s.t.
$y_n\to p$ and $\omega(y_n)\to q$}\}.
$$

The following lemma will be applied for defining
the extension to the reduced Gardiner-Masur closure
in \S\ref{subsec:extension_to_the_boundary}.

\begin{lemma}[Null sets and accumulation points]
\label{lem:equivalent_1}
Let $\omega\in \ACAS(\teich(S),\teich(S'))$.
Let $p^{1},p^{2}\in \partial_{GM}\teich(S)$
and $q^{i}\in \mathcal{A}(\omega\colon p_i)$ for $i=1,2$.
If $\nullsets{S}{p^{2}}\subset \nullsets{S}{p^{1}}$,
then $\nullsets{S'}{q^{2}}\subset \nullsets{S'}{q^{1}}$.
Especially,
$\nullsets{S'}{q^{2}}=\nullsets{S'}{q^{1}}$
for $p\in \partial_{GM}\teich(S)$ and $q^{1},q^{2}\in \mathcal{A}(\omega\colon p)$.
\end{lemma}

\begin{proof}
For $i=1,2$,
let ${\bf x}^i$ be a sequence converging to $p_i$
such that $\omega({\bf x}^i)$ converges to $q_i$.
From Proposition \ref{prop:accumulation_point_accompany},
the assumption $\nullsets{S}{p_2}\subset \nullsets{S}{p_1}$
implies $\visualindist({\bf x}^2)\subset \visualindist({\bf x}^1)$.
By Propositions \ref{prop:quasi-surjective_asymptotic},
we have $\visualindist(\omega({\bf x}^2))
\subset \visualindist(\omega({\bf x}^1))$.
Therefore,
by applying Proposition \ref{prop:accumulation_point_accompany} again,
we obtain
$\nullsets{S'}{q_2}\subset \nullsets{S'}{q_1}$.
\end{proof}

\subsection{Reduced Gardiner-Masur closure and boundary}
\label{subsec:definition_reduced}
We say two points $p,q\in \cl{\teich(S)}$ are equivalent if
one of the following holds:
\begin{itemize}
\item[(1)]
$p,q\in \teich(S)$ and $p=q$;
\item[(2)]
$p,q\in \partial_{GM}\teich(S)$
and $\nullsets{S}{p}=\nullsets{S}{q}$.
\end{itemize}
We denote by $\equivalence{p}$ the equivalence class of $p\in \cl{\teich (S)}$.
We abbreviate the equivalence class $\equivalence{[G]}$
of the projective class $[G]\in \mathcal{PMF}\subset\partial_{GM}\teich(S)$ as $\equivalence{G}$. 
We denote by $\clred{\teich(S)}$ the quotient of
$\cl{\teich(S)}$ under this equivalence relation.
Let $\pi_{GM}\colon \cl{\teich(S)}\to \clred{\teich(S)}$ be the quotient map.
We always identify $\pi_{GM}(\teich(S))$ with $\teich(S)$.
We call $\clred{\teich(S)}$ the \emph{reduced Gardiner-Masur closure} of $\teich(S)$.
From the definition,
the space $\clred{\teich(S)}$ contains $\teich(S)$ canonically.
We call the complement
$$
\partialred\teich(S)=\clred{\teich(S)}-\teich(S)
$$
the \emph{reduced Gardiner-Masur boundary} of $\teich(S)$.

The reduced Gardiner-Masur closure
is a variation of the reduced compactifications of Teichm\"uller space.
See \cite{Ohshika2}.

\subsection{Boundary extension}
\label{subsec:extension_to_the_boundary}
For $\omega\in \ACAS(\teich(S),\teich(S'))$,
we define the \emph{boundary extension} 
$\extension{\omega}\colon \clred{\teich(S)}\to \clred{\teich(S')}$
by
\begin{equation}
\label{eq:definition_extension}
\extension{\omega}(\equivalence{p})
=
\begin{cases}
\equivalence{\omega(p)} & (\mbox{$p\in \teich(S)$}) \\
\equivalence{q} & (\mbox{$q\in \mathcal{A}(\omega\colon p)$
if $p\in \partial_{GM}\teich(S)$}) 
\end{cases}
\end{equation}
From Lemma \ref{lem:equivalent_1},
the extension $\extension{\omega}$ is well-defined.

\begin{lemma}[Composition]
\label{lem:composition}
For
$\omega_1,\omega_2\in \ACAS(\teich(S),\teich(S'))$,
the extensions satisify
$$
\extension{\omega_1\circ \omega_2}=\extension{\omega_1}\circ \extension{\omega_2}
$$
on $\partialred{\teich(S)}$.
\end{lemma}

\begin{proof}
Let $\equivalence{p}\in \partialred{\teich(S)}$.
Take ${\bf x}=\{x_n\}_{n\in \mathbb{N}}\subset \teich(S)$ such that $x_n \to p$
and
$\omega_1\circ \omega_2(x_n)\to p'\in \partial_{GM}\teich(S')$.
By definition,
$$
\extension{\omega_1\circ\omega_2}(\equivalence{p})=\equivalence{p'}.
$$
On the other hand,
from Proposition \ref{lem:equivalent_1},
we may assume that
$\omega_2({\bf x})$ converges to $q\in \mathcal{A}(\omega_2\colon p)$.
From the definition,
we have $\extension{\omega_2}(\equivalence{p})=\equivalence{q}$.
Since $\omega_1\circ \omega_2({\bf x})
=\omega_1(\omega_2({\bf x}))$,
$p'\in \mathcal{A}(\omega_1\colon q)$ and hence
$$
\equivalence{p'}=
\extension{\omega_1}(\equivalence{q})=
\extension{\omega_1}\circ\extension{\omega_2}(\equivalence{p}).
$$
\end{proof}

\begin{lemma}[Close at infinity]
\label{lem:close-at-infinity}
Let $\omega_{1},\omega_{2}\in \ACAS(\teich(S),\teich(S'))$.
If $\omega_{1}$ is close to $\omega_{2}$ at infinity,
$\extension{\omega_1}=\extension{\omega_2}$
on $\partialred\teich(S)$.
\end{lemma}

\begin{proof}
Let $p\in \partial_{GM}\teich(S)$.
Take ${\bf x}=\{x_{n}\}_{n\in \mathbb{N}}\in \squ{\teich(S)}$
with $x_{n}\to p$ as $n\to\infty$
such that
$\omega_{i}(x_{n})\to q^{i}\in \mathcal{A}(\omega_{i}\colon p)$ for $i=1,2$.
Since $\visualindist(\omega_{1}({\bf x}))=\visualindist(\omega_{2}({\bf x}))$,
by Proposition \ref{lem:equivalent_1},
$$
\nullsets{S'}{q^{1}}=\accum_{S'}(\omega_{1}({\bf x}))=\accum_{S'}(\omega_{2}({\bf x}))
=\nullsets{S'}{q^{2}}.
$$
Hence 
$$
\extension{\omega_1}(\equivalence{p})=
\equivalence{q^{1}}=\equivalence{q^{2}}=
\extension{\omega_2}(\equivalence{p})
$$
and $\extension{\omega_1}=\extension{\omega_2}$ on $\partialred\teich(S)$.
\end{proof}

\begin{corollary}[Inverse]
\label{coro:quasi-inverse}
Let $\omega\in \ACINV(\teich(S),\teich(S'))$
and $\omega'$ be an asymptotic quasi-inverse of $\omega$.
Then,
$\extension{\omega'}\circ \extension{\omega}$
and
$\extension{\omega}\circ \extension{\omega'}
$
are identity mappings
on $\partialred{\teich}(S)$ and $\partialred{\teich}(S')$,
respectively.
\end{corollary}

\section{Rigidity of asymptotically conservative mappings}
\label{sec:Rigidity_asymptotic_teichmuller}

\subsection{Heights of reduced boundary points}
An ordered sequence $\{\equivalence{p_k}\}_{k=1}^m$ in $\partialred\teich(S)$
is said to be an \emph{adherence tower}
starting at $\equivalence{p_1}$ if
$$
\nullsets{S}{p_1}\supsetneqq
\nullsets{S}{p_2}\supsetneqq
\cdots
\supsetneqq
\nullsets{S}{p_m}.
$$
The adherence tower is named with referring 
Ohshika's paper \cite{Ohshika}.
See also Papadopoulos's paper \cite{Papadopoulos}.
We call the number $m$ the \emph{length} of the adherence tower.
Let $\equivalence{p}\in \partialred\teich(S)$.
We define the \emph{height} $\height(\equivalence{p})$ of $\equivalence{p}$
by
$$
\height(\equivalence{p})=\sup\{\mbox{lengths of adherence towers starting $\equivalence{p}$}\}.
$$

For a measured foliation $G$,
we set $\{X_i\}_{i=1}^{m_1}$ be the supports of the minimal components
of $G$.
We define the complexity of $G$ by
\begin{equation}
\label{eq:xi_0}
\xi_0(G)=
\left(
{\color{black}-}\sum_{i=1}^{m_1}\complexity{X_i},
{}^\#\{\mbox{essential curves in $G$}\}
\right)
\in \mathbb{Z}\times \mathbb{Z}.
\end{equation}
(cf. Theorem 1 in \cite{Ohshika}).

\begin{lemma}[Heights of boundary points]
\label{lem:height_boundary_points}
The height of any $\equivalence{p}\in \partialred\teich(S)$
is at most $\complexity{S}$.
The equality $\height(\equivalence{p})=\complexity{S}$
holds if and only if
the support of any $[G]\in \mathcal{AF}(p)$
is a simple closed curve.
\end{lemma}

\begin{proof}
We first discuss the associated foliations
of points in an adherence tower of length two.
Let $\equivalence{p_1},\equivalence{p_2}\in \partialred\teich(S)$.
Let $[G_i]\in \mathcal{AF}(p_i)$ for $i=1,2$.
Suppose that $\{\equivalence{p_1},\equivalence{p_2}\}$ is an adherence tower.
From the definition,
$\mathcal{N}(\Psi_{GM}(p_1))\supsetneqq \mathcal{N}(\Psi_{GM}(p_2))$.
From Corollary \ref{coro:topological_equivalence_null_set},
we see
$$
\NMF(\trunc{G_1})=
\NMF(\Psi_{GM}(p_1))\supsetneqq \NMF(\Psi_{GM}(p_2))
=\NMF(\trunc{G_2}).
$$
We decompose $\trunc{G_1}$ as in \eqref{eq:decompositionMF}:
\begin{equation}
\label{eq:truncation_1}
\trunc{G_1}=\sum_{i=1}^{m_1}G'_i+\sum_{i=1}^{m_2}\beta_i
\end{equation}
where $G'_i$ is a minimal component,
and $\beta_i$ is a (weighted) essential curve of $G_1$.
Since $G_2\in \NMF(G_1)$,
the decomposition of $\trunc{G_2}$ is represented as 
\begin{equation}
\label{eq:truncation_2}
\trunc{G_2}=\sum_{i=1}^{m_1}H'_i+\sum_{i=1}^{m_2}a_i\beta_i+G_3
\end{equation}
where $H'_i$ is topologically equivalent to $G'_i$,
$a_i\ge 0$ and the support of $G_3$ is contained in the complement
of the support of $\trunc{G_1}$
(At this moment,
$\trunc{G_2}$ may contain curves homotopic to boundary components
of arrational components of $G_1$
as essential curves).
Since $\NMF(G_2)\subset \NMF(G_1)$,
$H'_i\ne 0$ and $a_i\ne 0$.
Moreover,
from the assumption $\NMF(G_2)\ne \NMF(G_1)$
implies that $G_3\ne 0$.
Therefore,
from \eqref{eq:truncation_1} and \eqref{eq:truncation_2},
we have
$$
\xi_0(G_1)<\xi_0(G_2)
$$
in the lexicographical order in $\mathbb{Z}\times \mathbb{Z}$,
since $G_3$ in \eqref{eq:truncation_2}
contains either a minimal component or an essential curve
of $G_2$.
%
%
%
%

Let us return to the proof of the lemma.
Let $\{\equivalence{p_i}\}_{i=1}^m$ be an adherence tower of length $m$.
Let $[G_i]\in \mathcal{AF}(p_i)$.
From the above argument,
we have
\begin{equation}
\xi_0(G_1)<\xi_0(G_2)<\cdots<\xi_0(G_m).
\end{equation}
Since the number of essential curves is at most $\complexity{S}$
and the sum of the first and second coordinates of $\xi_0(G)$
is at most $\complexity{S}$
minus the number of boundary components of minimal foliations of $G$
which are non-periperal in $S$,
we have $m\le \complexity{S}$.
In addition,
%
if $m=\complexity{S}$,
each $G_i$ consists of essential curves.
Hence,
in this case,
the adherence tower starts with a simple closed curve.
\end{proof}

\subsection{Induced isomorphism}
Let $\curvecomplex_0(S)$ be the $0$-skeleton of $\curvecomplex(S)$.
We identify each vertex of $\curvecomplex_0(S)$ 
with its projective class in $\partial_{GM}\teich(S)$.

\begin{theorem}[Induced isomorphism]
\label{thm:induced-isomorphism}
Let $S$ and $S'$ be compact orientable surfaces of non-sporadic type.
For $\omega\in \ACINV(\teich(S),\teich(S'))$,
there is a simplicial isomorphism
$h_{\omega}\colon \curvecomplex(S)\to \curvecomplex(S)$
such that
for
any $\alpha\in \curvecomplex_0(S)$,
and any sequence $\{x_{n}\}_{n}\subset \teich(S)$
with $x_{n}\to [\alpha]$,
we have $\omega(x_{n})\to [h_{\omega}(\alpha)]$.
Furthermore,
When $\omega$ and $\omega'$ are close at infinity,
$h_{\omega}=h_{\omega'}$.
\end{theorem}

\begin{proof}
Let $\omega\in \ACINV(\teich(S),\teich(S'))$
and $\alpha\in \curvecomplex_0(S)$.
From Lemma \ref{lem:height_boundary_points},
there is an adherence tower $\{\equivalence{p_i}\}_{i=1}^{\complexity{S}}$
with $\equivalence{p_1}=\equivalence{\alpha}$.
From Lemma \ref{lem:equivalent_1},
$\{\extension{\omega}(\equivalence{p_i})\}_{i=1}^{\complexity{S}}$
is also an adherence tower starting
$
\extension{\omega}(\equivalence{p_1})
=\extension{\omega}(\equivalence{\alpha})$.
Applying the above argument for asymptotic quasi-inverse of $\omega$,
we see that
the adherence tower
$\{\extension{\omega}(\equivalence{p_i})\}_{i=1}^{\complexity{S}}$
has the maximal height.
From Lemma \ref{lem:height_boundary_points}
and Corollary \ref{coro:quasi-inverse},
we obtain a bijection $h_\omega\colon \curvecomplex_{0}(S)\to \curvecomplex_{0}(S')$
such that
\begin{equation}
\label{eq:definition_h_omega}
\extension{\omega}(\equivalence{\alpha})=\equivalence{h_\omega(\alpha)}.
\end{equation}

Let $\alpha,\beta\in \curvecomplex_0(S)$ with $i(\alpha,\beta)=0$.
Then,
$G=\alpha+\beta\in \mathcal{MF}$ and
$\mathcal{N}(\alpha)\cap \mathcal{N}(\beta)\supset \mathcal{N}(G)$.
Therefore,
$\{\equivalence{\alpha},\equivalence{G}\}$ and
$\{\equivalence{\beta},\equivalence{G}\}$
are adherence towers.
From Lemma \ref{lem:equivalent_1},
$\{\extension{\omega}(\equivalence{\alpha}),\extension{\omega}(\equivalence{G})\}$ and
$\{\extension{\omega}(\equivalence{\beta}),\extension{\omega}(\equivalence{G})\}$
are also adherence towers.
From Theorem \ref{thm:null_set},
there is an $H\in \mathcal{MF}$ such that $\extension{\omega}(\equivalence{G})=\equivalence{H}$.
Since $h_\omega$ is bijective,
$h_\omega(\alpha)$ and $h_\omega(\beta)$ represent
different components of $H$.
Therefore,
$i(h_\omega(\alpha),h_\omega(\beta))=0$.
This means that $h_\omega$ extends
a simplicial isomorphism from $\curvecomplex(S))$ to $\curvecomplex(S')$.
From Lemma \ref{lem:close-at-infinity},
one can easily see that
$h_{\omega'}=h_{\omega}$ when $\omega'$ is close to $\omega$ at infinity.
%

Let ${\bf x}=\{x_n\}_n$ be a sequence in $\teich(S)$ converging to
a simple closed curve $[\alpha]\in \cl{\teich(S)}$.
By \eqref{eq:definition_extension}
and \eqref{eq:definition_h_omega},
any accumulation point $q\in \partial_{GM}\teich(S')$
of a sequence $\omega({\bf x})$
satisfies $\nullsets{S'}{q}=\nullsets{S'}{[h_\omega(\alpha)]}$
from Lemma \ref{lem:equivalent_1}.
Hence $q$ satisfies
$i_{\omega(x_0)}(F,q)=0$
for all $F\in \mathcal{N}_{MF}(h_\omega(\alpha))$
($\subset \mathcal{MF}(S')$).
From Theorem 3 in \cite{Mi2},
we conclude that $q=[h_\omega(\alpha)]$ in $\partial_{GM}\teich(S')$.
This means that $\omega({\bf x})$ converges to
$[h_\omega(\alpha)]$ in $\cl{\teich(S')}$.
\end{proof}

\subsection{Rigidity theorem}

\subsubsection{Actions of extended mapping class group}
\label{subsec:action_of_extended_mapping_class_groups}
The \emph{extended mapping class group} $\MCG^*(S)$  of $S$
is the group of all isotopy classes of homeomorphisms on $S$.
The extended mapping class group $\MCG^*(S)$  acts on $\teich(S)$ isometrically
by 
$$
\teich(S)\ni y=(Y,f)\mapsto [h]_*(y)=(Y,f\circ h^{-1})\in \teich(S)
$$
for $[h]\in \MCG^*(S)$.
Hence,
we have a group homomorphism
\begin{equation}
\mathcal{I}_0\colon \MCG^*(S)\ni [h]\to [h]_*\in {\rm Isom}(\teich(S)),
\end{equation}
where ${\rm Isom}(\teich(S))$ is the group of all isometries of $\teich(S)$.

Let $\curvecomplex(S)$ be the complex of curves of $S$
and ${\rm Aut}(\curvecomplex(S))$ be the simplicial automorphisms on 
$\curvecomplex(S)$.
Since $\MCG^*(S)$ acts on $\curvecomplex (S)$ canonically,
we have a (group) homomorphism
\begin{equation}
\label{eq:homo_from_MCG_to_Aut}
\mathcal{J}\colon \MCG^*(S)\to {\rm Aut}(\curvecomplex (S)).
\end{equation}
It is known that $\mathcal{J}$ is an isomorphism
if $S$ is neither a torus with two holes nor
a closed surface of genus $2$,
and an epimorphism if 
$S$ is not a torus with two holes
(cf. Ivanov \cite{Ivanov2},
Korkmaz \cite{Korkmaz} and
Luo \cite{Luo}).

The action of any isometry on $\teich(S)$ extends homeomorphically
to the Gardiner-Masur boundary (cf. \cite{LiuSu}).
We can observe
that the extension of the action leaves
$\mathcal{S}\subset \partial_{GM}\teich(S)$ invariant,
and it induces a canonical homomorphism
$\mathcal{J}_1\colon {\rm Isom}(\teich(S))\to {\rm Aut}(\curvecomplex (S))$
such that
the diagram
$$
\xymatrix{
\MCG^*(S)
\ar[r]^{\mathcal{I}_0} \ar[dr]_{\mathcal{J}} & {\rm Isom}(\teich(S)) \ar[d]^{\mathcal{J}_1} \\
&  {\rm Aut}(\curvecomplex (S)) &  \\
}
$$
is commutative (cf. \cite{Mi5}).
The homomorphism $\mathcal{J}_1$ is an isomorphism for any $S$
with $\complexity{S}\ge 2$
(cf. \cite{Ivanov2}).
%
The reason why \eqref{eq:homo_from_MCG_to_Aut} is not surjective
when $S$ is a torus with two holes is that there is no homeomorphism on $S$
which sends a non-null-homologous curve to a null-homologous curve,
while each curve on $S'$ is null-homologous.
Thus,
in any case,
the homomorphism $\mathcal{J}_1$ is surjective
(cf. \cite{Luo}).

\subsubsection{Rigidity theorem}
Recall that any isometry is an invertible asymptotically conservative mapping.
Hence,
we have a monoid monomorphism
$$
\mathcal{I}\colon {\rm Isom}(\teich(S))\hookrightarrow \ACINV(\teich(S)).
$$
Our rigidity theorem is given as follows.

\begin{theorem}[Rigidity theorem]
\label{thm:induced_automorphism}
There is a monoid epimorphism
$$
\Xi\colon \ACINV(\teich(S))\to {\rm Aut}(\curvecomplex(S))
$$
with the following properties:
\begin{itemize}
\item[{\rm (1)}]
If $\omega'\in \ACINV(\teich(S))$ is an asymptotic quasi-inverse of
$\omega\in \ACINV(\teich(S))$,
$\Xi(\omega')=\Xi(\omega)^{-1}$;
\item[{\rm (2)}]
$\mathcal{J}_1=\Xi\circ \mathcal{I}$ as monoid homomorphisms.
\end{itemize}
In addition,
$\Xi$ descends to a group isomorphism
\begin{equation}
\label{eq:isomorphism}
\ACGroup(\teich(S))\to {\rm Aut}(\curvecomplex(S)).
\end{equation}
which satisfies the following commutative diagram:
$$
\xymatrix{
\MCG^*(S)
\ar[r]^{\mathcal{I}_0} &
{\rm Isom}(\teich(S))
\ar[r]^{\mathcal{I}} \ar[dr]^{\mbox{{\tiny group iso}}}
& \ACINV(\teich(S)) \ar[d]^{\mbox{{\tiny proj}}} \ar[dr]^{\Xi}& \\
& &\ACGroup(\teich(S))  \ar[r]^{\mbox{{\tiny group iso}}}
&{\rm Aut}(\curvecomplex(S)). &
}
$$
\end{theorem}

\begin{proof}
When $S$ is a torus with two holes,
the quotient map $S\to S'$ by the hyper-elliptic action
induces an isometry between
the Teichm\"uller spaces of $S$ and $S'$
and an isomorphism between $\curvecomplex(S)$ and $\curvecomplex(S')$,
where $S'$ is a sphere with five holes
(cf. \cite{EK} and \cite{Luo}).
Hence,
we may assume that $S$ is not a torus with two holes.
For $\omega\in \ACINV(\teich(S))$,
we take $h_{\omega}\in {\rm Aut}(\curvecomplex(S))$
as Theorem \ref{thm:induced-isomorphism}.
Define a homomorphism $\Xi$ by $\Xi(\omega)=h_\omega$.
Theorem \ref{thm:induced-isomorphism}
asserts that
$\Xi$ satisfies
the condition (1) in the statement
and 
descends to a homomorphism
\begin{equation}
\label{eq:equation_2}
\ACGroup(\teich(S))\ni [\omega]\mapsto \Xi(\omega)=
h_\omega\in {\rm Aut}(\curvecomplex(S)).
\end{equation}

We next check the condition (2) in the statement.
Since $\omega\in {\rm Isom}(\teich(S))$
preserves $\mathcal{S}$ in
$\mathcal{PMF}\subset\partial_{GM}\teich(S)$,
from the definition of $h_\omega$,
for any $\alpha\in \mathcal{S}$,
$\Xi(\omega)(\alpha)$ coincides with $\omega(\alpha)$
(cf. \S9 in \cite{Mi5}).
This means that $\mathcal{J}(\omega)=\Xi\circ \mathcal{I}(\omega)$.

We here check \eqref{eq:isomorphism} is an epimorphism.
%
Since $S$ is not a torus with two holes,
$\mathcal{J}$ is an epimorphism,
and so are $\Xi$ and \eqref{eq:equation_2}.
The injectivity of \eqref{eq:isomorphism} (or \eqref{eq:equation_2})
is proven in the next section.
\end{proof}

\subsection{Injectivity of homomorphism}
In this section,
we shall show that
the epimorphism \eqref{eq:isomorphism} is an isomorphism.
We first check the following.

\begin{proposition}
\label{prop:rigidity_2}
Suppose that $S$ is not a torus with two holes.
For $\omega\in \ACINV(\teich(S))$,
there is a homeomorphism $f_\omega$ of $S$ with the following property:
For any $p\in \partial_{GM} \teich(S)$,
$[G]\in \mathcal{AF}(p)$
and $q\in \mathcal{A}(\omega\colon p)$,
we have $\nullsets{S}{q}=\nullsets{S}{[f_\omega(G)]}$.
\end{proposition}

\begin{proof}
From the assumption and Theorem \ref{thm:induced_automorphism},
there is a homeomorphism $f_\omega$ of $S$ such that
$h_\omega(\alpha)=f_\omega(\alpha)$.
From Theorem \ref{thm:null_set},
if we take $[H]\in \mathcal{AF}(q)$,
then
$$
\nullsets{S}{q}=\nullsets{S}{[H]}.
$$
From Theorem \ref{thm:induced-isomorphism},
for any $\alpha\in \mathcal{S}$,
$i(G,\alpha)=0$ if and only if
$[f_\omega(\alpha)]=[h_\omega(\alpha)]\in \nullsets{S}{q}$.
Hence,
we deduce that
\begin{equation}
\label{eq:simple_closed_curves_1}
\nullsets{S}{[H]}\cap \mathcal{S}=
\nullsets{S}{q}\cap \mathcal{S}=\nullsets{S}{[f_\omega(G)]}\cap \mathcal{S}
\end{equation}
where $\mathcal{S}$ stands for a subset of $\partial_{GM}\teich(S)$
in \eqref{eq:simple_closed_curves_1}.
Therefore,
the support of $\trunc{H}$ coincides with the support of
$\trunc{f_\omega(G)}$.
In particular,
any essential curve of $H$ is also that of $f_\omega(G)$, and vice versa.
As \eqref{eq:decompositionMF},
we decompose $G$ as
$$
G=G_1+G_2+\cdots+G_{m_1}+\beta_1+\cdots \beta_{m_2}
+\gamma_1+\cdots+\gamma_{m_3}.
$$
Let  $X_i$ be the support of a minimal component $G_i$ of $G$.

It is known that  $\cl{\teich(S)}$ is metrizable.
For instance
\begin{equation}
\label{eq:metric-cl-GM}
d_\infty(p^1,p^2)=\sup_{p\in \partial_{GM}\teich(S)}
\left|
i_{x_0}(p^1,p)-i_{x_0}(p^2,p)
\right|
\end{equation}
is a metric on $\cl{\teich(S)}$
since $\mathcal{S}\subset \partial_{GM}\teich(S)$
(cf. Theorem 1.2 in \cite{Mi1}).

Fix $i=1,\cdots,k$.
Take a sequence $\{\alpha_n\}_{n\in \mathbb{N}}\subset \mathcal{S}$ such that
$\alpha_n\subset X_i$ and
$$
d_\infty([\alpha_n],[G_i])<1/n.
$$
Since $f_\omega$ is a homeomorphism,
$[f_\omega(\alpha_n)]$ tends to $[f_\omega(G_i)]$
in $\mathcal{PMF}$ (and hence in $\cl{\teich(S)}$)
as $n\to \infty$.
By taking a subsequence,
we may assume that
$$
d_\infty([f_\omega(\alpha_n)],[f_\omega(G_i)])<1/n
$$
for all $n\in \mathbb{N}$.

Let ${\bf x}^n=\{x^n_m\}_{m\in \mathbb{N}}$ be a sequence in $\teich(S)$
converging to $[\alpha_n]$ in $\cl{\teich(S)}$.
Since $\omega\in \ACINV(\teich (S))$,
by Theorem \ref{thm:induced-isomorphism},
$\omega({\bf x}^n)$ converges to $[f_\omega(\alpha_n)]$
in $\cl{\teich(S)}$
for all $n$.
By applying the diagonal argument and taking a subsequence if necessary,
we can take $m(n)\in \mathbb{N}$
such that
if we put $z_n=x^n_{m(n)}$ and ${\bf z}=\{z_n\}_{n\in \mathbb{N}}$,
then
\begin{equation}
\label{eq:prood_rigidity_2}
\max\{d_\infty(z_n,[G_i]),
d_\infty(\omega(z_n),[f_\omega(\alpha_n)])\}<2/n
\end{equation}
in $\cl{\teich(S)}$.
Since $f_\omega$ is a homeomorphism of $S$,
$[f_\omega(\alpha_n)]$ tends to $[f_\omega(G_i)]$
in $\mathcal{PMF}$ and hence in $\cl{\teich(S)}$.
From \eqref{eq:prood_rigidity_2},
we have
$$
d_\infty(\omega(z_n),f_\omega([G_i]))
\le d_\infty(\omega(z_n),f_\omega([\alpha_n]))
+d_\infty(f_\omega([\alpha_n]),[f_\omega(G_i)])
\to 0
$$
as $n\to \infty$.
Therefore $\omega({\bf z})$ converges to $[f_\omega(G_i)]$.
This means that
$[f_\omega(G_i)]\in \mathcal{A}(\omega\colon [G_i])$.
Since $G_i$ is a minimal component of $G$,
$\nullsets{S}{p}=\nullsets{S}{[G]}\subset \nullsets{S}{[G_i]}$.
Therefore,
by Lemma \ref{lem:equivalent_1},
we conclude
$$
\nullsets{S}{[H]}=
\nullsets{S}{q}\subset \nullsets{S}{[f_\omega(G_i)]}
$$
since $q\in \mathcal{A}(\omega\colon p)$.
Therefore,
$H$ contains a minimal component $H_i$
which is topologically equivalent to
$f_\omega(G_i)$.
Since the support of $\trunc{H}$ coincides with that of $\trunc{G}$,
minimal components of $H$ are contained in
$\cup f_\omega(X_i)$.
Hence,
the normal form of $H$ should be
$$
H=\sum_{i=1}^{m_1}H_i+\sum_{i=1}^{m_2}a_if_\omega(\beta_i)
+\sum_{i=1}^{m_1}\sum_{\gamma\subset \partial f_\omega(X_i)}b_\gamma \gamma
$$
where $a_i>0$ and $b_\gamma\ge 0$.
Thus,
$\trunc{H}$ is topologically equivalent to $\trunc{f_\omega(G)}$.
Hence
by Corollary \ref{coro:topological_equivalence_null_set},
we deduce 
$$
\nullsets{S}{[f_\omega(G)]}=
\nullsets{S}{[H]}=\nullsets{S}{q},
$$
which is what we desired.
\end{proof}

\begin{proposition}[Induced isometry]
\label{prop:rigidity_3}
For any $\omega\in \ACINV(\teich(S))$,
there is a unique isometry $\xi_\omega$ on $\teich(S)$
which is close to $\omega$ at infinity.
\end{proposition}

\begin{proof}
The case where $S$ is a torus with two holes
follows from the fact that
 the Teichm\"uller space of $S$ is isometric to 
the Teichm\"uller space of a sphere with five holes.
Hence,
we may suppose that $S$ is not a torus with two holes.

Take $f_\omega$ as in Proposition \ref{prop:rigidity_2}.
Since $f_\omega$ is a homeomorphism of $S$,
$f_\omega$ induces an isometry $\xi_\omega$ on $\teich(S)$.
When $S$ is a closed surface of genus $2$,
there is an ambiguity of the choice of $f_\omega$
which is caused by the hyperelliptic involution.
However,
the isometry $\xi_\omega$ is independent of the choice.

Let ${\bf x}^1$, ${\bf x}^2\in \sq{\teich(S)}$
satisfying $\visualindist({\bf x}^1)=\visualindist({\bf x}^2)$.
From Proposition \ref{prop:null_sets_and_accompany},
there are $G$,
$H_1$,
$H_2\in \mathcal{MF}$ such that
\begin{align*}
\nullsets{S}{[G]} &
=\accum_{S}({\bf x}^1)=\accum_{S}({\bf x}^2)\\
\nullsets{S}{[H_1]}
&=\accum_{S}(\omega({\bf x}^1)) \\
\nullsets{S}{[H_2]}
&=\accum_{S}(\xi_\omega({\bf x}^2)).
\end{align*}
Hence,
our assertion follows from
Proposition \ref{prop:accumulation_points_and_accompany}
and the following lemma.
\end{proof}

\begin{lemma}
\label{lem:GH_1H_2}
It holds
$$
\nullsets{S}{[H_1]}=\nullsets{S}{[f_\omega(G)]}=
\nullsets{S}{[H_2]}.
$$
\end{lemma}

\begin{proof}
Let $\overline{{\bf w}}\in \visualindist(\omega({\bf x}^1))$
and $q\in \overline{{\bf w}}\cap\partial_{GM}\teich(S)$.
Let $p\in \overline{{\bf x}^1}\cap \partial_{GM}\teich(S)$
and fix $[G_p]\in \mathcal{AF}(p)$.
From Proposition \ref{prop:rigidity_2},
$i_{x_0}(q,[f_\omega(G_p)])=0$.
Since $p$ is taken arbitrarily in $\overline{{\bf x}^1}\cap \partial_{GM}\teich(S)$,
from the proof of Proposition \ref{prop:null_sets_and_accompany},
we have $i_{x_0}(q,[f_\omega(G)])=0$.
Hence
\begin{equation}
\label{eq:proof_of_claim:GH_1H_2-1}
\nullsets{S}{[H_1]}
\subset \nullsets{S}{[f_\omega(G)]}.
\end{equation}

Let $q'\in \nullsets{S}{[f_\omega(G)]}$.
For $q\in \overline{\omega({\bf x}^1)}\cap \partial_{GM}\teich(S)$,
we take a subsequence ${\bf z}$ of ${\bf x}^1$
such that $\omega({\bf z})$ converges to $q$
and
 ${\bf z}$ converges to
some $p\in \overline{{\bf x}^1}\cap\partial_{GM}\teich(S)$.
Fix $[G_p]\in \mathcal{AF}(p)$.
From Proposition \ref{prop:rigidity_2},
we have $\nullsets{S}{q}=\nullsets{S}{[f_\omega(G_p)]}$.
From the construction of $G$,
$\trunc{G_p}$ is topologically equivalent to a subfoliation
of $G$.
Hence,
$q'\in \nullsets{S}{[f_\omega(G_p)]}=\nullsets{S}{q}$.
Since $q$ is taken arbitrarily from
$\overline{\omega({\bf x}^1)}\cap \partial_{GM}\teich(S)$,
from \eqref{eq:proof_of_claim:GH_1H_2-1},
we deduce that 
$q'\in \nullsets{S}{[H_1]}$
and
\begin{equation}
\label{eq:H_1_equal_to_f_omega_G-1}
\nullsets{S}{[H_1]}=\nullsets{S}{[f_\omega(G)]}.
\end{equation}

Since $\xi_\omega$ is an isometry,
$$
\visualindist(\xi_{\omega}({\bf x}^2))=
\xi_\omega(\visualindist({\bf x}^2)).
$$
Since $\xi_\omega$
extends to $\cl{\teich(S)}$ homeomorphically
and coincides with the action of $f_\omega$ on
$\mathcal{PMF}\subset \partial_{GM}\teich(S)$,
we have
\begin{align*}
\nullsets{S}{[H_2]}\cap \mathcal{PMF}
&=\accum_{S}(\xi_\omega({\bf x}^2))\cap \mathcal{PMF}
=\xi_\omega(\accum_{S}({\bf x}^2))\cap \mathcal{PMF} \\
&=\xi_\omega(\nullsets{S}{[G]})\cap \mathcal{PMF}
=\nullsets{S}{[f_\omega(G)]}\cap \mathcal{PMF}.
\end{align*}
This equality means that $\NMF(H_2)=\NMF(f_\omega(G))$.
By Corollary \ref{coro:topological_equivalence_null_set},
we have $\mathcal{N}(H_2)=\mathcal{N}(f_\omega(G))$
and $\nullsets{S}{[H_2]}=\nullsets{S}{[f_\omega(G)]}$.
\end{proof}
%

\begin{proof}[Proof of the injectivity of the homomorphism \eqref{eq:isomorphism}]
Let $\omega\in \ACAS(\teich(S))$
be in the kernel of $\Xi$.
From Proposition \ref{prop:rigidity_3},
there is an isometry $\xi_\omega$ which is close to $\omega$.
Since $\Xi(\xi_\omega)=\Xi(\omega)=id$,
$\xi_\omega$ is the identity mapping on $\teich(S)$,
and hence $\omega$ is close to the identity.
\end{proof}

%

\subsection{Rough homotheties on Teichm\"uller space}
\label{subsec:rough_homothety}
In this section,
we shall prove Theorem \ref{thm:Teichmullerspace_homothety}.
%
%

Suppose first that $\complexity{S}\ge 2$.
We may assume that $S$ is not a torus with two holes.
Suppose to the contrary that there is a $(K,D)$-rough
homothety $\omega$ with asymptotic quasi-inverse
for some $K\ne 1$.
Notice 
that $\omega\in \ACINV(\teich(S))$.
Take a homeomorphism $f_\omega$ on $S$ as Proposition \ref{prop:rigidity_2}.

Let $\alpha,\beta\in \mathcal{S}$.
Consider the projective classes $[\alpha]$ and $[\beta]$ as
points in $\partial_{GM}\teich(S)$.
Then,
from \eqref{eq:homothetic_implies_rough_gromov_product},
Theorem \ref{thm:induced-isomorphism} and Proposition \ref{prop:rigidity_2},
we have
\begin{equation}
\label{eq:K-homothety_1}
e^{-D_0}i_{x_0}([\alpha],[\beta])^{K}
\le
i_{x_0}([f_\omega(\alpha)],[f_\omega(\beta)])
\le
e^{D_0}i_{x_0}([\alpha],[\beta])^{K}
\end{equation}
where $D_0$ is a constant depending only on $D$ and
$d_T(x_0,\omega(x_0))$.
On the other hand,
let $K_0=e^{2d_T(x_0,\xi_\omega(x_0))}$,
where $\xi_\omega$ is an isometry associated to $\omega$
taken as Proposition \ref{prop:rigidity_3}.
From the definition,
$\ext_{x_0}(f_\omega(G))=\ext_{\xi_\omega^{-1}(x_0)}(G)$
for $G\in \mathcal{MF}$.
By the quasiconformal invariance of extremal length,
we obtain
$$
K_0^{-1}i_{x_0}([\alpha],[\beta])
\le
i_{x_0}([f_\omega(\alpha)],[f_\omega(\beta)])
\le 
K_0
i_{x_0}([\alpha],[\beta])
$$
since $f_\omega$ is a homeomorphism on $S$
and $i(f_\omega(\alpha),f_\omega(\beta))=i(\alpha,\beta)$.
Therefore,
we deduce
\begin{align}
i_{x_0}([\alpha],[\beta])^{1-K}
&\le K_0e^{D_0}
\label{eq:K-homothety_4} \\
i_{x_0}([\alpha],[\beta])^{K-1}
&\le K_0e^{D_0}
\label{eq:K-homothety_4_2}
\end{align}
for any $\alpha,\beta\in \mathcal{S}$ with $i_{x_0}([\alpha],[\beta])\ne 0$.
Since the left-hand sides of \eqref{eq:K-homothety_4}
and \eqref{eq:K-homothety_4_2}
are projectively invariant,
when the projective classes
$[\alpha],[\beta]$ tend together
to some projective measured foliation $[G]\in \mathcal{PMF}$
with keeping satisfying $i(\alpha,\beta)\ne 0$,
the left-hand side in \eqref{eq:K-homothety_4} diverges
if $K>1$,
otherwise
the left-hand side in \eqref{eq:K-homothety_4_2} diverges.
In any case,
we get a contradiction.

We now consider the case where $\complexity{S}=1$.
This case is indeed a prototype of our study.
In this case,
there is an isometry $\teich(S) \to \mathbb{D}$ sending $x_0$ to the origin $0$.
Furthermore,
the Gromov product $\gromov{x_1}{x_2}{0}$
for $x_1,x_2\in \mathbb{D}$ satisfies
\begin{equation}
\label{eq:linear_inequality}
|\gromov{x_1}{x_2}{0}-d_{\mathbb{D}}(0,[x_1,x_2])|\le D_1
\end{equation}
for some universal constant $D_1>0$,
where $[x_1,x_2]$ is the geodesic connecting between $x_1$ and $x_2$
(cf. \S2.33 in \cite{Vaisala}).

Suppose on the contrary that
there is a $(K,D)$-rough homothety  $\omega$ with $K\ne 1$.
Notice 
from the definition that any $\omega\in \ACINV(\mathbb{D})$
extends to a bijective mapping on $\partial \mathbb{D}$.
We can easily see that the extension is continuous,
and hence,
$\omega$ extends to a self-homeomorphism on $\partial \mathbb{D}$.
We may assume that $\omega(0)=0$.

From \eqref{eq:linear_inequality},
for $x_1,x_2\in \mathbb{D}$,
$$
|d_{\mathbb{D}}(0,[\omega(x_1),\omega(x_2)])-Kd_{\mathbb{D}}(0,[x_1,x_2])|\le D_2
$$
for some constant $D_2>0$.
Therefore,
for any $p_1,p_2\in \partial \mathbb{D}$,
we have
\begin{equation}
\label{eq:extension_K_Holder}
C_1|p_1-p_2|^K\le |\omega(p_1)-\omega(p_2)|\le C_2|p_1-p_2|^K
\end{equation}
with positive constants $C_1,C_2$.
If $K>1$,
$\omega$ is differentiable and the derivative is zero at any $\partial \mathbb{D}$.
Hence,
$\omega$ should be a constant on $\partial \mathbb{D}$,
which is a contradiction.
Suppose $K<1$.
Since the lift of a self-homeomorphism on $\partial \mathbb{D}$
to $\mathbb{R}$ is a monotone function,
the extension of $\omega$ to $\partial \mathbb{D}$
is differentiable almost everywhere on $\partial \mathbb{D}$.
However,
from \eqref{eq:extension_K_Holder},
$\omega$ is not differentiable any point on $\partial \mathbb{D}$.
This is also a contradiction.
\section{Appendix}
\label{sec:appendix}
The main result of this section
is Lemma \ref{lem:fills_curves}.
The estimates in the lemma looks similar to
that in Theorem 6.1 of \cite{GM}.
However,
our advantage here is that we treat
the extremal lengths of all non-trivial  (possibly peripheral) curves
of subsurfaces
and give a constant $C_\gamma$ concretely
(cf. \eqref{eq:lemma_extremal_length_3}
and \eqref{eq:comparison_extremal_length}).

\subsection{Measured foliations and intersection numbers}
Let $Q$ be a holomorphic quadratic differential on $X$.
The differential $|{\rm Re}\sqrt{Q}|$ defines a measured foliation on $X$.
We say that such a measured foliation the \emph{vertical foliation} of $Q$.
The vertical foliation of $-Q$ is called the \emph{horizontal foliation}
of $Q$.

By a \emph{step curve},
we mean a geodesic polygon in $X$ the sides of
which are horizontal and vertical arcs of $Q$
(cf. Figure \ref{fig:stepcurve}).
\begin{figure}[t]
\includegraphics[height=3cm]{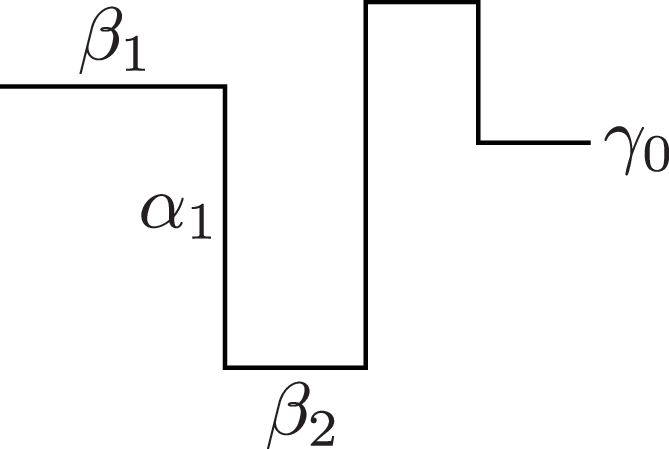}
\caption{A step curve with the property stated in
Proposition \ref{prop_strebel_stepcurve}.}
\label{fig:stepcurve}
\end{figure}
For the intersection number functions defined by the vertical foliations of
holomorphic quadratic differentials,
it is known the following.

\begin{proposition}[Theorem 24.1 of \cite{Strebel}] \label{prop_strebel_stepcurve}
Let $Q$ be a quadratic differential and $F$ the vertical foliation of $Q$.
Let $\gamma_0$ be a simple closed step curve with the additional property that
for any vertical side $\alpha_1$ of $\gamma_0$ the two neighboring
horizontal sides $\beta_1$ and $\beta_2$ are on different sides of $\alpha_1$
(there are no zeros of $Q$ on $\gamma_0$).
Then,
$$
i(\gamma,F)=\int_{\gamma_0}|{\rm Re}\sqrt{Q}|,
$$
where $\gamma$ is the homotopy class containing $\gamma_0$.
\end{proposition}

It can be also observed that a step curve with the property
stated in Theorem \ref{prop_strebel_stepcurve} is quasi-transversal.
For instance, see the proof of Proposition II.6 or the curve (4) of Figure 10
of Expos\'e 5 in \cite{FLP}.

\subsection{Filling curves and Extremal length}
Let $X_0$ be an essential subsurface of $X$.
Denote by $\mathcal{S}(X_0)$ a subset of $\mathcal{S}$
consisting of curves which are non-peripheral in $X_0$.
Let $\mathcal{S}_\partial (X_0)$ be a subset of $\mathcal{S}$
consisting of curves which can be deformed into $X_0$.

\begin{lemma} \label{lem:fills_curves}
Let $X_0$ be a connected, compact and essential subsurface of $X$
with negative Euler characteristic.
Let $\{\alpha_i\}_{i=1}^m\subset \mathcal{S}(X_0)$
be a system of curves which fills up $X_0$.
Then,
for $\gamma\in \mathcal{S}_\partial(X_0)$,
we have
\begin{equation} \label{eq:lemma_extremal_length}
\ext_X(\gamma)\le
C_\gamma
\max_ {1\le i\le m}\ext_X(\alpha_i),
\end{equation}
where 
\begin{equation} \label{eq:lemma_extremal_length_3}
C_\gamma=C(g,n,m)\left(
\sum_{i=1}^m i(\alpha_i,\gamma)\right)^2
+4(6g-6+n)^2
\end{equation}
and
$C(g,n,m)$ depends only on the topological type $(g,n)$
of $X$ and the number $m$ of the system
$\{\alpha_i\}_{i=1}^m$.
In particular,
we have
\begin{equation} \label{eq:lemma_extremal_length_2}
\ext_X(F)\le
C(g,n,m)\left(
\sum_{i=1}^m i(\alpha_i,F)\right)^2
\max_ {1\le i\le m}\ext_X(\alpha_i),
\end{equation}
for all $F\in \mathcal{MF}(X_0)\subset \mathcal{MF}$.
\end{lemma}


\begin{proof}
Let $\gamma\in \mathcal{S}_\partial(X_0)$.
We divide the proof into two cases.

\medskip
\noindent
\paragraph{{\bf Case 1 : $\gamma$ is peripheral in $X_0$.}}
Suppose first that
$\gamma$ is represented by a component of $\partial X_0$.
When $\gamma$ is homotopic to a puncture of $X$,
$\ext_X(\gamma)=0$ since $X$ contains an arbitrary wide annulus whose
core is homotopic to $\gamma$.
Hence we have nothing to do
(in fact,
we can set $C_\gamma=0$).

Suppose that $\gamma$ is not peripheral in $X$.
Let $J_\gamma$ be a Jenkins-Strebel differential for $\gamma$ on $X$.
Let $A_\gamma$ be the characteristic annulus of $J_\gamma$.
We consider a ``compactification" $\overline{A_\gamma}$ by attaching
two copies of circles as its boundaries.
The induced flat structure on $A_\gamma$ from $J_\gamma$
canonically extends to the compactification $\overline{A_\gamma}$
and components of
the boundary $\partial \overline{A_\gamma}$ are closed regular trajectories
under this flat structure.
There is 
a canonical surjection $I_\gamma:\overline{A_\gamma}\to \overline{X}$
(the completion of $X$ at the punctures).
Namely, $\overline{X}$ is reconstructed by identifying
disjoint vertical straight arcs in $\partial \overline{A_\gamma}$
along vertical saddle connections of $J_\gamma$.
(In this sense, $I_\gamma$ is a quotient map).
Without any confusion,
we may recognize the characteristic annulus $A_\gamma$ itself as a subset of $X$.

Let $\gamma^*$ and $\alpha_i^*$ be the core trajectory
in $A_\gamma$
and the geodesic representative of $\alpha_i$ with respect to $J_\gamma$
respectively.
Since $\gamma$ is parallel to $\partial X_0$,
by taking an isotopy,
we may assume that $\gamma^*$ is a component  of $\partial X_0$.
Furthermore,
since $\alpha_i\in \mathcal{S}(X_0)$,
$\gamma$ does not intersect any $\alpha_i$ for all $i$.
Hence,
each $\alpha_i^*$ consists of vertical saddle connections.
In other words,
$\alpha_i^*$ is contained in the critical graph
$\Sigma_\gamma=I_\gamma(\partial \overline{A_\gamma})$ of $J_\gamma$ in $X$,
which consists of vertical saddle connections of $J_\gamma$.

Let $\gamma_1$ and $\gamma_2$ be components of $\partial \overline{A_\gamma}$.
Each $\gamma_i^*:=I_\gamma(\gamma_i)$ is canonically recognized as a path in $\Sigma_\gamma$
consisting of vertical saddle connections.
We claim:

\medskip
\noindent
{\bf Claim 1.}\
One of $\gamma_i^*$,
say $\gamma_1^*$,
is contained in the union $\cup_{i=1}^m\alpha_i^*$.

\begin{proof}[Proof of Claim 1]
Suppose $\gamma_1^*\cap \alpha_i^*\ne \emptyset$
for some $i$
and $\gamma_1^*$ contains a vertical saddle connection $s_0$
such that $s_0\not\subset \alpha^*_i$ for all $i$.
Then,
$s_0$ intersects all $\alpha^*_i$ at most
at endpoints (critical points of $J_\gamma$).
Let ${\rm Int}(s_0)=s_0\setminus \partial s_0$.
Let $h_1$ be a horizontal arc in $A_\gamma$
starting at $p_1\in \gamma^*$
and terminating at a point of ${\rm Int}(s_0)$.
Since the both side of $s_0$ is in $A_\gamma$,
after $h_1$ passes through $s_0$,
$h_1$ terminates at a point $p_2\in \gamma^*$.
Let $\gamma_0^*$ be a segment of $\gamma^*$
connecting $p_1$ and $p_2$
 (cf. Figure \ref{fig:A_gamma}).
Set $\beta=\gamma_0^*\cup h_1$.
By definition,
$\beta$ does not intersects any $\alpha_i^*$
and hence
$i(\beta,\alpha_i)=0$
for all $i$.
\begin{figure}[t]
\includegraphics[height=5cm]{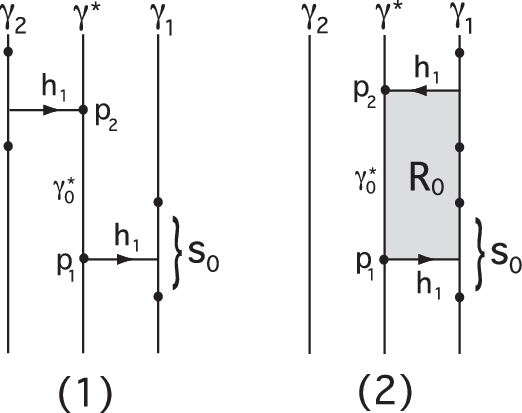}
\caption{Trajectories in $A_\gamma$.}
\label{fig:A_gamma}
\end{figure}

Suppose first that $h_1$ arrived $p_2$ from the different side
from that where $h_1$
departed at $p_1$
(cf. (1) of Figure \ref{fig:A_gamma}).
Then,
we have
$$
i(\gamma,\beta)=\int_{\beta}|{\rm Re}\sqrt{J_\gamma}|=1,
$$
since the width of $A_\gamma$ is one
and $\beta$ is a step curve with the property
stated in Proposition \ref{prop_strebel_stepcurve}.
Hence $\beta$ is non-trivial and non-peripheral simple closed curve in $X$.
However,
this contradicts that $\{\alpha_i\}_{i=1}^m$ fills up $X_0$,
since such a $\beta\cap X_0$
contains homotopically non-trivial arc connecting $\partial X_0$
because $\gamma$ is parallel to a component of $\partial X_0$.

Suppose $h_1$ arrived at $p_2$ from the same side
as that where $h_1$
departed
(cf. (2) of Figure \ref{fig:A_gamma}).
We may also assume that $h_1$ departs from $p_1$ into $X_0$.
Indeed,
suppose we cannot assume so.
Then,
the component of $\partial A_\gamma$
that lies on the same side as that of $X_0$ (near $\gamma^*$)
is covered by $\{\alpha_i^*\}_{i=1}^m$,
which contradicts what we assumed first.

Then,
there is an open rectangle $R_0$ in $A_\gamma$
such that $\beta$ and a segment in $\gamma_1$ surround $R_0$ in $A_\gamma$.
From the assumption,
we may assume that the closure of $I_\gamma(R_0)$,
say $X_1$,
intersects some $\alpha_i^*$.
Suppose that $\beta$ is trivial.
Then, $X_1$ is a disk in $X$
surrounded by $\beta$,
since $\gamma^*$ can be homotopic to the outside of
$X_1$.
This means that $\alpha_i^*$ is contained in a disk
$X_1$
because $\alpha_i^*$ does not intersect $\beta$,
which is a contradiction.
By the same argument,
we can see that $\beta$ is non-peripheral
(otherwise, $\alpha_i^*$ were peripheral).
Since $h_1$ departs into $X_0$ at $p_1$
and returns to $\gamma^*$ on the side where $X_0$ lies,
after taking an isotopy if necessary,
we can see that $h_1$ contains a subsegment which is nontrivial in $X_0$
and connecting $\partial X_0$,
which contradicts again
that $\{\alpha_i\}_{i=1}^m$ fills $X_0$ up.
\end{proof}

Let us continue to prove Lemma \ref{lem:fills_curves} for peripheral $\gamma
\in \mathcal{S}_\partial(X_0)$.
We take $\gamma_1^*$ as in Claim 1.
Since both sides of every vertical saddle connection face $A_\gamma$,
$\gamma_1^*$ visits each vertical saddle connection
at most twice.
Notice that the number of vertical saddle connections
is at most $6g-6+n$.
Since each vertical saddle connection in $\gamma_1^*$
is contained in some $\alpha_i^*$,
we have
\begin{align*}
\ell_{J_\gamma}(\gamma)
=\ell_{J_\gamma}(\gamma_1^*)
&\le 2(6g-6+n)\max\{\ell_{J_\gamma}(\alpha_i^*)\mid i=1,\cdots,n\}  \nonumber \\
&= 2(6g-6+n)\max\{\ell_{J_\gamma}(\alpha_i)\mid i=1,\cdots,n\},
\end{align*}
since $\alpha_i^*$ is the geodesic representative of $\alpha_i$.
Since the width of $A_\gamma$ is one,
from \eqref{eq:JS_extremallength_norm_length},
we conclude
\begin{align}
\ext_X(\gamma)
&=\ell_{J_\gamma}(\gamma)^2/\|J_\gamma\| \nonumber \\
&\le
4(6g-6+n)^2
\max_ {1\le i\le m}\{\ell_{J_\gamma}(\alpha_i)^2/\|J_\gamma\| \} \nonumber \\
&\le 
4(6g-6+n)^2
\max_ {1\le i\le m}\ext_X(\alpha_i).
\label{eq:extremal_length_peripheral}
\end{align}

\medskip
\noindent
\paragraph{{\bf Case 2 : $\gamma\in \mathcal{S}(X_0)$.}}
We next assume that $\gamma$ is not parallel to any component of $\partial X_0$.
Let $\{\beta_i\}_{i=1}^s$ be components of $\partial X_0$
each of which is non-peripheral in $X$.
Let $\epsilon>0$ and set
\begin{equation} \label{eq:epsilon_perturbation}
F_\epsilon=\gamma+\epsilon\sum_{i=1}^s\beta_s
\end{equation}
(cf. \cite{Ivanov}).
It is possible that two curves $\beta_{i_1}$ and $\beta_{i_2}$
are homotopic in $X$.
In this case,
we recognize $\beta_{i_1}+\beta_{i_2}=\beta_{i_1}$ in 
\eqref{eq:epsilon_perturbation}.
However,
for the simplicity of the discussion,
we shall assume that any two of $\{\beta_{i}\}_{i=1}^s$ are not isotopic.
The general case can be treated in a similar way.

Let $J_\gamma^\epsilon$ be the holomorphic quadratic differential on $X$
whose vertical foliation is $F_\epsilon$.
Since $F_\epsilon\to \gamma$ in $\mathcal{MF}$,
$J_\gamma^\epsilon$ tends to $J_\gamma$ in $\mathcal{Q}_X$
(cf. \cite{HM}. See also Theorem 21.3 in \cite{Strebel}).
Let $A_\gamma^\epsilon$
and $A_i^\epsilon$ denote the characteristic annuli
of $J_\gamma^\epsilon$ for $\gamma$ and $\beta_i$,
respectively.
Set $\gamma^{\epsilon,*}$ and $\beta_i^{\epsilon,*}$
to be closed trajectories in homotopic to $\gamma$ and $\beta_i$,
respectively.
Let $Y_0^\epsilon$ be the closure of
the component of
$\epsilon/4$-neighborhood of the cores
$\beta_i^{\epsilon,*}$,
containing $A_\gamma^\epsilon$.
By definition,
we may identify
$Y_0^\epsilon$ with $X_0$.
Let $\alpha^{\epsilon,*}_i$ be the geodesic representation of $\alpha_i$
with respect to $J_\gamma^\epsilon$.



We fix an orientation on $\gamma^{\epsilon,*}$.
Let $\xi$ be a component of
$\gamma^{\epsilon,*}\setminus \cup_{i=1}^m\alpha^{\epsilon,*}_i$.
Let $I_0(\xi)$ be the set of points $p\in \xi$
such that the horizontal ray $r_p$
departing at $p$ from the right of $\xi$
terminates at a curve
in $\{\alpha^{\epsilon,*}_i,\beta_j^{\epsilon,*}\}_{i,j}$
before intersecting $\xi$
twice.
Let $C_0(\xi)$ be the set of $p\in \xi$
such that $r_p$ terminates at a critical point of $J_\gamma^\epsilon$.
Then,
we claim

\medskip
\noindent
{\bf Claim 2.}\
$\xi\setminus I_0(\xi)\subset C_0(\xi)$,
and $I_0(\xi)\setminus C_0(\xi)$ is open in $\xi$.

\begin{proof}[Proof of Claim 2]
Let $p\in \xi\setminus I_0(\xi)$.
Suppose $p\not\in C_0(\xi)$.
Since the completion $\overline{X}$ with respect to the punctures
is closed,
$r_p$ is recurrent
(cf. \S10 of Chapter IV in \cite{Strebel}).
By the definition of $I_0(\xi)$ and $p\not\in I_0(\xi)$,
$r_p$ intersects $\xi$ at least twice
before intersecting  curves
in $\{\alpha^{\epsilon,*}_i,\beta_j^{\epsilon,*}\}_{i,j}$.
Hence,
$r_p$
contains
a consecutive horizontal segments $h_1$ and $h_2$
such that each $h_i$ intersects $\xi$ only at its endpoints,
and does not intersect any curves in
$\{\alpha^{\epsilon,*}_i,\beta_j^{\epsilon,*}\}_{i,j}$.

When one of the segments, say $h_1$,
connects both sides of $\xi$,
$\xi$ contains a vertical segment $v_1$ connecting endpoints of $h_i$,
and two trajectories $h_1$ and $v_1$ make a closed curve $\delta$ on $X$.
Since the two ends of $h_i$ terminate at $\xi$ from different sides,
the intersection number satisfies
$$
i(F_\epsilon,\delta)=\int_\delta |{\rm Re}\sqrt{J_\gamma^\epsilon}|
$$
and is greater than or equal to the width of $A_\gamma^\epsilon$,
by Proposition \ref{prop_strebel_stepcurve}.
Therefore $\delta$ is non-trivial and non-peripheral in $X$.
Since $h_i$ does not intersect $\beta_i^{\epsilon,*}$,
$\delta$ is contained in $Y_0^\epsilon$,
where we have {\color{black}identified} with $X_0$.
Furthermore,
$\delta$ is not peripheral in $X_0$
because $\delta$ has non-trivial intersection with $\gamma$.
By definition,
$\delta$ does not intersect all $\alpha_i$,
which is a contradiction because $\{\alpha_i\}_{i=1}^m$
fills $X_0$ up.

We assume that
two ends of each $h_i$ terminate at $\xi$ from the same side.
In this case,
we can also construct a simple closed step curve $\delta$
with the property
stated in
Proposition \ref{prop_strebel_stepcurve} from $h_1$, $h_2$ and a subsegment
of $\xi$ (cf. Figure \ref{fig:delta}).
\begin{figure}[t]
\includegraphics[height=5cm]{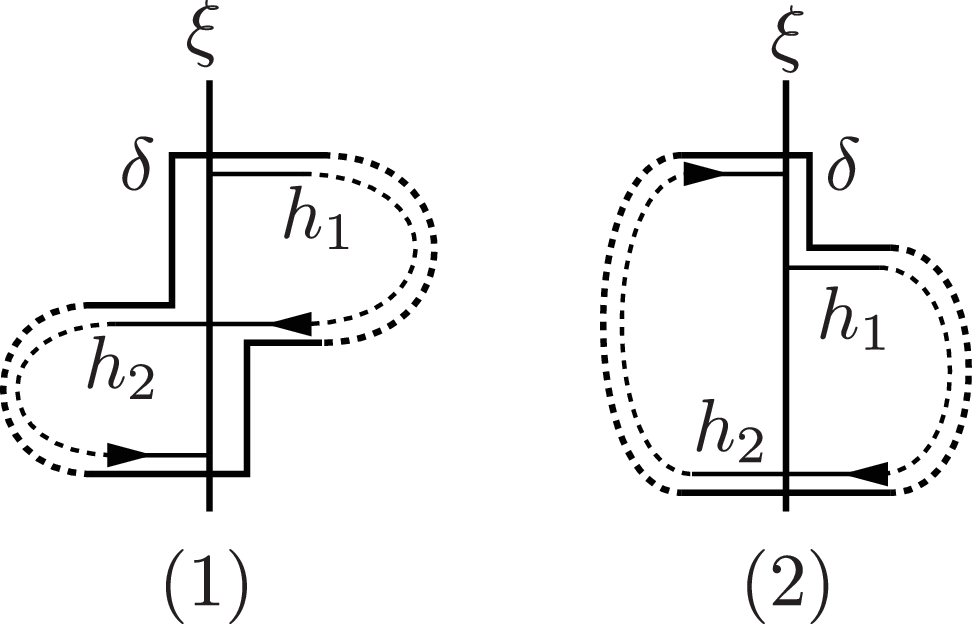}
\caption{How to get a closed curve $\delta$ :
There are two cases.
In the case (1), the initial point of $h_1$ and the terminal point of $h_2$
are separated by the terminal point of $h_1$.
The case (2) describes the other case.}
\label{fig:delta}
\end{figure}
This is a contradiction as above.
Thus we conclude that $\xi\setminus I_0(\xi)\subset C_0(\xi)$.

We show that $I_0(\xi)\setminus C_0(\xi)$
is open in $\xi$.
Let $p\in I_0(\xi)\setminus C_0(\xi)$
such that the horizontal ray $r_p$ defined above
does not terminate at critical points of $J_\gamma^\epsilon$.
By definition,
the horizontal ray $r_p$ terminate the interior of
a straight arc
contained in either $\alpha_i^{\epsilon,*}$
or $\beta_j^{\epsilon,*}$.
Hence,
when $p'\in \xi$ is in some small neighborhood of $p$,
$r_{p'}$ also terminates at such a straight arc,
and hence $p'\in I_0(\xi)$
for all point $p'$ in a small neighborhood of $p$.
\end{proof}

Let us return to the proof of Case 2 of the lemma.
Let $\xi$ be a component of  $\gamma^\epsilon\setminus \cup_{i=1}^m\alpha^{\epsilon,*}_i $.
By definition,
for $p\in I_0(\xi)$,
$r_p$ terminates at $\xi$
at most once
before intersecting curves in $\{\alpha_i^{\epsilon,*},\beta_j^{\epsilon,*}\}_{i,j}$.
Since any horizontal ray $r_p$ with $p\not\in C_0(\xi)$ can terminate
at a curve
in $\{\alpha^{\epsilon,*}_i,\beta_j^{\epsilon,*}\}_{i,j}$
from at most two sides.
Hence for almost all point $q$ in a curve
in $\{\alpha^{\epsilon,*}_i,\beta_j^{\epsilon,*}\}_{i,j}$,
there are at most $4$ points in $I_0(\xi)$
such that the horizontal rays emanating there land at $q$.
From Claim 2,
we get
$$
|\xi|=|I_0(\xi)\setminus C_0(\xi)|
\le 
4
\left(
\sum_{i=1}^m\ell_{J_\gamma^\epsilon}(\alpha_i) 
+
\sum_{i=1}^s\ell_{J_\gamma^\epsilon}(\beta_i)
\right),
$$
where $|\cdot |$ means linear measure.
Since $\alpha_i^{\epsilon,*}$
is the geodesic representative of $\alpha_i$,
the number of components of
$\gamma^{\epsilon,*}\setminus \cup_{i=1}^m\alpha^{\epsilon,*}_i$
is
$$
\sum_{i=1}^m i(\alpha_i,\gamma).
$$
Therefore,
we obtain
$$
\ell_{J_\gamma^\epsilon}(\gamma)
\le  
4
\left(
\sum_{i=1}^m i(\alpha_i,\gamma)
\right)
\left(
\sum_{i=1}^m\ell_{J_\gamma^\epsilon}(\alpha_i) 
+
\sum_{i=1}^s\ell_{J_\gamma^\epsilon}(\beta_i)
\right).
$$
Since $J_\gamma^\epsilon$ tends to $J_\gamma$
as $\epsilon\to 0$,
for all $\eta>0$ we find an $\epsilon>0$ such that
\begin{align*}
\ext_X(\gamma)
&=
\frac{\ell_{J_\gamma}(\gamma)^2}{\|J_\gamma\|}
\le
\frac{\ell_{J_\gamma^\epsilon}(\gamma)^2}{\|J^\epsilon_\gamma\|}+\eta \\
&\le 
16\left(
\sum_{i=1}^m i(\alpha_i,\gamma)
\right)^2(m+s)^2
\left(
\sum_{i=1}^m
\frac{\ell_{J_\gamma^\epsilon}(\alpha_i)^2}{\|J^\epsilon_\gamma\|}
+
\sum_{i=1}^s
\frac{\ell_{J_\gamma^\epsilon}(\beta_i)^2}{\|J^\epsilon_\gamma\|}
\right)+\eta \\
&\le 
C'_\gamma
\max_ {1\le i\le m}\ext_X(\alpha_i)
+\eta
\end{align*}
by Corollary 21.2 in \cite{Strebel}
and the Cauchy-Schwarz inequality,
where
\begin{equation*}
C'_\gamma
=16(m+s)^2(m+4s(6g-6+n)^2)\left(
\sum_{i=1}^m i(\alpha_i,\gamma)\right)^2.
\end{equation*}
from \eqref{eq:extremal_length_peripheral}.

Notice that the number $s$ of components
of $\partial X_0$ satisfies $s\le 2g+n$.
Indeed,
we fix a hyperbolic metric on $X$ and
realize $X_0$ as a convex hyperbolic subsurface of $X$.
Let $g'$ and $n'$ be the genus and the number of punctures in $X_0$.
Since $X_0\subset X$
and $X_0$ is essential,
by comparing to the hyperbolic area,
we have
\begin{align*}
2\pi(s-2)&\le 2\pi (2g'-2+s+n')
={\rm Area}(X_0) \\
&\le {\rm Area}(X)=2\pi(2g-2+n),
\end{align*}
and hence $s\le 2g+n$.

Thus,
by \eqref{eq:extremal_length_peripheral},
we conclude that
\eqref{eq:lemma_extremal_length} holds with
\begin{equation} \label{eq:comparison_extremal_length}
C(g,n,m):=
16(m+2g+n)^2(m+4(2g+n)(6g-6+n)^2),
\end{equation}
which implies what we wanted.
\end{proof}

\end{document}